\documentclass[italian,english]{amsart}

\usepackage{amssymb}
\usepackage{babel}
\usepackage{color}
\usepackage[T1]{fontenc}
\usepackage[pdfborder={0 0 0},ocgcolorlinks]{hyperref}
\usepackage{prettyref}

\numberwithin{equation}{section}

\makeatletter
\newcommand\thmsname{\protect\theoremname}
\newcommand\nm@thmtype{theorem}
\theoremstyle{plain}

\newenvironment{namedthm}[1][Undefined Theorem Name]{
  \ifx{#1}{Undefined Theorem Name}\renewcommand\nm@thmtype{theorem*}
  \else\renewcommand\thmsname{#1}\renewcommand\nm@thmtype{namedtheorem}
  \fi
  \begin{\nm@thmtype}}
  {\end{\nm@thmtype}}
\makeatother

\theoremstyle{plain}
\newtheorem{thm}{\protect\theoremname}[section]
\newtheorem{prop}[thm]{\protect\propositionname}
\newtheorem{fact}[thm]{\protect\factname}
\newtheorem{lem}[thm]{\protect\lemmaname}
\newtheorem{cor}[thm]{\protect\corollaryname}
\theoremstyle{remark}
\newtheorem{rem}[thm]{\protect\remarkname}
\newtheorem*{acknowledgement*}{\protect\acknowledgementname}
\theoremstyle{definition}
\newtheorem{defn}[thm]{\protect\definitionname}

\newrefformat{def}{\protect\definitionname\ \ref{#1}}
\newrefformat{prop}{\protect\propositionname\ \ref{#1}}
\newrefformat{fact}{\protect\factname\ \ref{#1}}
\newrefformat{cor}{\protect\corollaryname\ \ref{#1}}
\newrefformat{lem}{\protect\lemmaname\ \ref{#1}}
\newrefformat{rem}{\protect\remarkname\ \ref{#1}}
\newrefformat{sub}{\protect{Subsection}\ \ref{#1}}

\providecommand{\acknowledgementname}{Acknowledgement}
\providecommand{\corollaryname}{Corollary}
\providecommand{\definitionname}{Definition}
\providecommand{\factname}{Fact}
\providecommand{\lemmaname}{Lemma}
\providecommand{\propositionname}{Proposition}
\providecommand{\remarkname}{Remark}
\providecommand{\theoremname}{Theorem}

\newcommand\no{\mathbf{No}}
\newcommand\on{\mathbf{On}}

\newcommand\J{\mathbb{J}}
\newcommand\K{\kappa_{\no}}
\newcommand\F{\mathbb{F}}
\newcommand\li{\mathbb{L}}
\newcommand\M{\mathfrak{M}}
\newcommand\N{\mathbb{N}}
\newcommand\Q{\mathbb{Q}}
\newcommand\R{\mathbb{R}}
\newcommand\T{\R^{*}\M}
\newcommand\Z{\mathbb{Z}}
\newcommand\nin{\mathrel{\notin}}
\newcommand\eps{\varepsilon}
\newcommand\m{\mathfrak{m}}
\newcommand\n{\mathfrak{n}}
\newcommand\om{\mathfrak{o}}
\newcommand\simple{\mathrel{<_{s}}}
\newcommand\simpleq{\mathrel{\leq_{s}}}
\newcommand\trunc{\mathrel{\vartriangleleft}}
\newcommand\trunceq{\mathrel{\trianglelefteq}}
\newcommand\ntrunc{\mathrel{\blacktriangleleft}}
\newcommand\ntrunceq{\mathrel{\ooalign{{\raise-1ex\hbox{$\relbar$}}\cr\raise.1ex\hbox{$\ntrunc$}}}}
\newcommand\vell{\ell}
\newcommand\vless{\mathrel{\prec}}
\newcommand\vleq{\mathrel{\preceq}}
\newcommand\vgreater{\mathrel{\succ}}
\newcommand\vgeq{\mathrel{\succeq}}
\newcommand\veq{\mathrel{\asymp}}
\newcommand\lless{\mathrel{\prec^{{\scriptscriptstyle \mathrm{L}}}}}
\newcommand\lleq{\mathrel{\preceq^{{\scriptscriptstyle \mathrm{L}}}}}
\newcommand\lgreater{\mathrel{\succ^{{\scriptscriptstyle \mathrm{L}}}}}
\newcommand\lgeq{\mathrel{\succeq^{{\scriptscriptstyle \mathrm{L}}}}}
\newcommand\lequal{\mathrel{\asymp^{{\scriptscriptstyle \mathrm{L}}}}}
\newcommand\kless{\mathrel{\prec^{{\scriptscriptstyle \mathrm{K}}}}}
\newcommand\kleq{\mathrel{\preceq^{{\scriptscriptstyle \mathrm{K}}}}}
\newcommand\kgreater{\mathrel{\succ^{{\scriptscriptstyle \mathrm{K}}}}}
\newcommand\kgeq{\mathrel{\succeq^{{\scriptscriptstyle \mathrm{K}}}}}
\newcommand\kequal{\mathrel{\asymp^{{\scriptscriptstyle \mathrm{K}}}}}
\newcommand\suchthat{\,:\,}
\newcommand\somed{D}
\newcommand\simpled{\partial}

\DeclareMathOperator\ntrank{NR}
\DeclareMathOperator\dom{dom}
\DeclareMathOperator\img{Im}
\DeclareMathOperator\supp{S}

\newcommand\bracket[1]{\R\langle\langle#1\rangle\rangle}
\newcommand\ka[1]{\kappa_{-#1}}
\newcommand\sign[1]{\operatorname{sign}(#1)}
\newcommand\convex[2]{(#1;#2)}

\title{Surreal numbers, derivations and transseries}
\date{May 13, 2015}
\author{Alessandro Berarducci}

\thanks{A.B.\ was partially supported by PRIN 2012
  ``\foreignlanguage{italian}{Logica, Modelli e Insiemi}'' and by a
  Leverhulme Visiting Professorship (VP2-2013-055) at the School of
  Mathematical Sciences, Queen Mary University of London. }

\address{\foreignlanguage{italian}{Universit\`a di Pisa, Dipartimento
    di Matematica, Largo Bruno Pontecorvo 5, 56127 Pisa, PI}, Italy}

\email{berardu@dm.unipi.it}

\author{Vincenzo Mantova}

\thanks{V.M.\ was supported by FIRB2010 ``New advances in the Model
  Theory of exponentiation'' RBFR10V792 at the University of Camerino
  and by ERC AdG ``Diophantine Problems'' 267273 at the
  \foreignlanguage{italian}{Scuola Normale Superiore}.}

\address{\foreignlanguage{italian}{Scuola Normale Superiore, Classe di
    Scienze, Piazza dei Cavalieri 7, 56126 Pisa, PI}, Italy.}

\email{vincenzo.mantova@sns.it}



\begin{document}

\begin{abstract}
  Several authors have conjectured that Conway's field of surreal
  numbers, equipped with the exponential function of Kruskal and
  Gonshor, can be described as a field of transseries and admits a
  compatible differential structure of Hardy-type. In this paper we
  give a complete positive solution to both problems. We also show
  that with this new differential structure, the surreal numbers are
  Liouville closed, namely the derivation is surjective.
\end{abstract}

\maketitle

\section{Introduction}

Conway's class ``$\mathbf{No}$'' of surreal numbers is a remarkable
mathematical structure introduced in \cite{Conway1976}. Besides being
a \emph{universal domain} for ordered fields (in the sense that every
ordered field whose domain is a set can be embedded in $\mathbf{No}$),
it admits an exponential function $\exp:\mathbf{No}\to\mathbf{No}$
\cite{Gonshor1986} and an interpretation of the real analytic
functions restricted to finite numbers \cite{Alling1987}, making it,
thanks to the results of \cite{Ressayre1993,Dries1994}, into a model
of the theory of the field of real numbers endowed with the
exponential function and all the real analytic functions restricted to
a compact box \cite{DriesE2001}.

It has been suggested that $\no$ could be equipped with a derivation
compatible with $\exp$ and with its natural structure of generalized
power series field. One would like such a derivation to formally
behave as the natural derivation on the germs at infinity of functions
$f:\R\to\R$ belonging to a ``Hardy field''
\cite{Bourbaki1976,Rosenlicht1983,Miller2012}.  This can be given a
precise meaning through the notion of $H$-field
\cite{Aschenbrenner2002,Aschenbrenner2005}, a formal algebraic
counterpart of the notion of Hardy field.

A related conjecture is that $\no$ can be viewed as a universal domain
for various generalized power series fields equipped with an
exponential function, including \'Ecalle's field of \emph{transseries}
\cite{Ecalle1992} (introduced in connection with Dulac's problem) and
its variants, such as the \emph{logarithmic-exponential series} of L.\
van den Dries, A.\ Macintyre and D.\ Marker
\cite{Dries1997,DriesMM2001}, the \emph{exponential-logarithmic
  series} of S.\ Kuhlmann \cite{Kuhlmann2000}, and the transseries of
J.\ van der Hoeven \cite{VanderHoeven1997,VanderHoeven2006} and M.\
Schmeling \cite{Schmeling2001}. Referring to logarithmic-exponential
series, in \cite{DriesE2001} the authors say that ``There are also
potential connections with the theory of surreal numbers of Conway and
Kruskal, and super exact asymptotics\textquotedblright . Some years
later, a more precise formulation was given in
\cite{VanderHoeven2006}: ``We expect that it is actually possible to
construct isomorphisms between the class of surreal numbers and the
class of generalized transseries of the reals with so called
transfinite iterators of the exponential function and nested
transseries. A start of this project has been carried out in
collaboration with my former student M.\ Schmeling
\cite{Schmeling2001}. If this project could be completed, this would
lead to a remarkable correspondence between growth-rate functions and
numbers.'' Further steps in this direction were taken by S.\ Kuhlmann
and M.\ Matusinski in \cite{Kuhlmann2011b,Kuhlmann2012} leading to the
explicit conjecture that $\no$ is a field of ``exponential-logarithmic
transseries'' \cite[Conj.\ 5.2]{Kuhlmann2014} and can be equipped with
a ``Hardy-type series derivation'' \cite[p.\ 368]{Matusinski2012}.

In this paper we give a complete solution to the above problems
showing that the surreal numbers have a natural transseries structure
in the sense of \cite[Def.\ 2.2.1]{Schmeling2001} (although not in the
sense of \cite{Kuhlmann2014}) and finding a compatible Hardy-type
derivation.

We expect that these results will lead to a considerable
simplification of the treatment of transseries (which will be
investigated in a forthcoming paper) and thus provide a valuable tool
for the study of the asymptotic behavior of functions. In the light of
the model completeness conjectures of \cite{Aschenbrenner2013}, we
also expect that $\no$, equipped with this new differential structure,
is an elementary extension of the field of the logarithmic-exponential
series.

\medskip

In order to describe the results in some detail, we recall that the
surreal numbers can be represented as binary sequences of transfinite
ordinal length, so that one can endow $\no$ with a natural tree-like
well-founded partial order $\simple$ called ``simplicity relation''.
Another very useful representation describes surreal numbers as
infinite sums $\sum_{x\in\no}a_{x}\omega^{x}$, where
$x\mapsto\omega^{x}$ is Conway's omega-function, $a_{x}\in\R$ for all
$x$, and the support $\{\omega^{x}:a_{x}\neq0\}$ is a reverse
well-ordered set, namely every non-empty subset has a maximum. In
other words, $\no$ coincides with the \emph{Hahn field}
$\R((\omega^{\no}))$ with coefficients in $\R$ and monomial group
$(\omega^{\no},\cdot)$ (see \prettyref{sub:Hahn-fields}).  In
particular, we have a well defined notion of infinite ``summable''
families in $\no$. In this paper we prove:

\begin{namedthm}[Theorem A (\ref{thm:D-surreal})]
  Conway's field of surreal numbers $\no$ admits a derivation
  $\somed:\no\to\no$ satisfying the following properties:
  \begin{enumerate}
  \item Leibniz' rule: $\somed(xy)=x\somed(y)+y\somed(x)$;
  \item strong additivity:
    $\somed\left(\sum_{i\in I}x_{i}\right)=\sum_{i\in I}\somed(x_{i})$
    if $(x_{i}\suchthat i\in I)$ is summable;
  \item compatibility with exponentiation:
    $\somed(\exp(x))=\exp(x)\somed(x)$;
  \item constant field $\R$: $\ker(\somed)=\R$;
  \item $H$-field: if $x>\N$, then $\somed(x)>0$.
  \end{enumerate}
\end{namedthm}

We call \textbf{surreal derivation} any function $D:\no\to\no$
satisfying properties (1)-(5) in Theorem~A. We show in fact that there
are several surreal derivations, among which a ``simplest'' one
$\simpled:\no\to\no$.  We can prove that the simplest derivation
$\simpled$ satisfies additional good properties, such as
$\simpled(\omega)=1$ and the existence of anti-derivatives.

\begin{namedthm}[Theorem B (\ref{thm:liou-closed-small})]
  The field $\no$ of surreal numbers equipped with $\simpled$ is a
  Liouville closed $H$-field with small derivation in the sense of
  \cite[p.\ 3]{Aschenbrenner2002}, namely, $\simpled$ is surjective
  and sends infinitesimals to infinitesimals.
\end{namedthm}

In the course of the proof, we also discover that $\no$ is a field of
transseries as anticipated in \cite{VanderHoeven2006}.

\begin{namedthm}[Theorem C (\ref{thm:t4})]
  $\no$ is a field of transseries in the sense of \cite[Def.\
  2.2.1]{Schmeling2001}.
\end{namedthm}

As an application of the above results, we observe that the existence
of surreal derivations yields an immediate proof that $\no$ satisfies
the statement of Schanuel's conjecture ``modulo $\R$'', thanks to Ax's
theorem \cite{Ax1971}, similarly to what was observed in
\cite{Kuhlmann2013} for various fields of transseries
(\prettyref{cor:schanuel}).  Since $\no$ is a monster model of the
theory of $\R_{\exp}$, the same statement follows for every elementary
extension of $\R_{\exp}$.  It is actually known that any model of the
theory of $\R_{\exp}$ satisfies an even stronger Schanuel type
statement ``modulo $\mathrm{dcl}(\emptyset)$'' (see \cite{Jones2008}
and \cite{Kirby2010}).

\medskip

The strategy to prove the existence of a surreal derivation $\somed$
is the following. Let $\J\subset\no$ be the non-unital ring of the
\textbf{purely infinite numbers}, consisting of the surreal numbers
$\sum_{x\in\no}a_{x}\omega^{x}$ having only infinite monomials
$\omega^{x}$ in their support (namely, $x>0$ whenever
$a_{x}\neq0$). It is known that
\[
\omega^{\no}=\exp(\J),
\]
so we can write $\no=\R((\omega^{\no}))=\R((\exp(\J)))$. In other
words, every surreal number can be written in the form
$\sum_{\gamma\in\J}r_{\gamma}\exp(\gamma)$.  We baptize this
``Ressayre form'' in honor of J.-P.\ Ressayre, who showed in
\cite{Ressayre1993} that every ``real closed exponential field''
admits a similar representation. A surreal derivation must satisfy
\begin{equation}
  \somed\left(\sum_{\gamma\in\J}r_{\gamma}\exp(\gamma)\right)=
  \sum_{\gamma\in\J}r_{\gamma}\somed\left(\exp(\gamma)\right)=
  \sum_{\gamma\in\J}r_{\gamma}\exp(\gamma)\somed(\gamma).
  \label{eq:deriv-ressayre}
\end{equation}
Using the displayed equation, the problem of defining $\somed$ is
reduced to the problem of defining $\somed(\gamma)$ for $\gamma\in\J$.

The iteration of this procedure is not sufficient by itself to find a
definition of $D$. For instance, the above equation gives almost no
information on the values of $\somed$ on the subclass $\li$ of the
\textbf{log-atomic numbers}, namely the elements $\lambda\in\no$ such
that all the iterated logarithms $\log_{n}(\lambda)$ are of the form
$\exp(\gamma)$ for some $\gamma\in\J$. Indeed, for $\lambda\in\li$,
the above equation reduces merely to
$\somed(\exp(\lambda))=\exp(\lambda)\somed(\lambda)$, and it is easy
to see that this condition is not sufficient for a map
$\somed_{\li}:\li\to\no$ to extend to a surreal derivation.

\medskip

As pointed out in the work of S.\ Kuhlmann and M.\ Matusinski
\cite{Kuhlmann2011b,Kuhlmann2014}, the class $\li$ of log-atomic
numbers is crucial for defining a derivation, so we should first give
some details about the relationship between $\li$ and $\no$. On the
face of the definition it is not immediate that $\li$ is non-empty,
but it can be shown that $\omega\in\li$, and more generally that every
``$\eps$-number'' (see \cite{Gonshor1986} for a definition) belongs to
$\li$. In fact, in \cite{Kuhlmann2014} there is an explicit
parametrization of a class of log-atomic numbers, properly including
the $\eps$-numbers, called ``$\kappa$-numbers''.  In the same paper it
is conjectured that the $\kappa$-numbers generate $\li$ under
application of $\log$ and $\exp$. However, we will show that the class
$\li$ is even larger (\prettyref{prop:kappa-not-gen-L}) and we shall
provide an explicit parametrization of the whole of $\li$
(\prettyref{cor:lambda-is-L}). It turns out that log-atomic numbers
can be seen as the natural representatives of certain equivalence
classes (\prettyref{def:levels}) which are similar but finer than
those in \cite{Kuhlmann2014}, and correspond exactly to the ``levels''
of a Hardy field \cite{Rosenlicht1987,Marker1997}, except that in our
case there are uncountably many levels (actually a proper class of
them).

Once $\li$ is understood, consider the smallest subfield
$\bracket{\li}$ of $\no$ containing $\R\cup\li$ and closed under
taking $\exp$, $\log$ and infinite sums. We shall see that
$\bracket{\li}$ is the largest subfield of $\no$ satisfying axiom ELT4
of \cite[Def.\ 5.1]{Kuhlmann2014} (\prettyref{prop:bracket-elt4}),
which was there proposed as part of a general notion of transseries
(and is satisfied by the logarithmic-exponential series of
\cite{Dries1997} and the exponential-logarithmic series of
\cite{Kuhlmann2000}). Clearly, any derivation on $\bracket{\li}$
satisfying (1)-(5) (as in Theorem A) is uniquely determined by its
restriction to $\li$. A natural question is now whether
$\bracket{\li}=\no$.  This is equivalent to the first part of
Conjecture 5.2 in \cite{Kuhlmann2014}.  However, we shall prove that
axiom ELT4 fails in the surreal numbers, and therefore the inclusion
$\bracket{\li}\subseteq\no$ is strict (see
\prettyref{thm:elt4-fails}).

\medskip

Despite the fact that $\li$ does not generate $\no$ under $\exp$,
$\log$ and infinite sums, a fundamental issue in our construction is
understanding how a surreal derivation should behave on $\li$.  One
can verify that if a map $\somed_{\li}:\li\to\no$ extends to a surreal
derivation, then necessarily $\somed_{\li}(\lambda)>0$ for all
$\lambda\in\li$, and moreover
\begin{equation}
  |\log(\somed_{\li}(\lambda))-\log(\somed_{\li}(\mu))|<
  \frac{1}{n}|\lambda-\mu|\label{eq:val-ineq}
\end{equation}
for all $\lambda,\mu\in\li$ and all $n\in\N$. This inequality plays a
crucial role in this paper, and it can be proved to hold for the
natural derivation on any Hardy field closed under $\log$, provided
$\lambda$, $\mu$ and $\left|\lambda-\mu\right|$ are positive infinite.

We start our construction by defining a ``pre-derivation''
$\somed_{\li}:\li\to\no^{>0}$ satisfying \prettyref{eq:val-ineq} and
$\somed_{\li}(\exp(\lambda))=\exp(\lambda)\somed_{\li}(\lambda)$ for
all $\lambda\in\li$. It turns out that the \emph{simplest}
pre-derivation, which we call $\simpled_{\li}:\li\to\no^{>0}$, can be
calculated by a rather explicit formula. For this, we need a bit of
notation involving a subclass of the $\kappa$-numbers of
\cite{Kuhlmann2014}.  For $\alpha\in\on$ (where $\on$ is the class of
all ordinal numbers) define inductively $\ka{\alpha}\in\no$ as the
simplest positive infinite surreal number less than
$\log_{n}(\ka{\beta})$ for all $n\in\N$ and $\beta<\alpha$. With this
notation we have (see \prettyref{def:simple-pre-D})
\[
\simpled_{\li}(\lambda) = \exp\left(-\sum_{\exists
    n\suchthat\exp_{n}(\ka{\alpha})>\lambda}
  \sum_{i=1}^{\infty}\log_{i}(\ka{\alpha}) +
  \sum_{i=1}^{\infty}\log_{i}(\lambda)\right),
\]
where $\alpha$ ranges in $\on$. (For the sake of exposition, we shall
use the above formula as definition of $\simpled_{\li}$, and only at
the end of the paper we shall prove that it is the simplest
pre-derivation; see \prettyref{thm:simplest-pre-D}.)

\medskip

Once $\simpled_{\li}:\li\to\no^{>0}$ is given, we can use
\prettyref{eq:deriv-ressayre} to give a tentative definition of a
surreal derivation $\simpled:\no\to\no$ extending $\simpled_{\li}$. We
adopt the same formalism used by Schmeling in
\cite{Schmeling2001}. First of all, we recall Schmeling's notion of
``path'' . Given $x=\sum_{\gamma\in\J}r_{\gamma}\exp(\gamma)\in\no$, a
\textbf{path} of $x$ is a function $P:\N\to\no$ such that
\begin{itemize}
\item $P(0)$ is a term of $x$, namely $P(0)=r_{\gamma}\exp(\gamma)$
  for some $r_{\gamma}\neq0$;
\item if $P(n)=r\exp(\eta)$, then $P(n+1)$ is a term of $\eta$.
\end{itemize}
For the moment, we restrict our attention to $\bracket{\li}$. As we
already mentioned, $\bracket{\li}$ satisfies axiom ELT4 of
\cite{Kuhlmann2014}, which can be paraphrased as saying that for all
$x\in\bracket{\li}$, every path $P$ of $x$ enters $\li$, namely there
is $n=n_{P}\in\N$ such that $P(n)\in\li$. Let $\mathcal{P}(x)$ be the
set of all paths of $x$. Iterating \prettyref{eq:deriv-ressayre}, we
immediately see that our desired surreal derivation $\simpled$
extending $\simpled_{\li}$ must satisfy
\begin{equation}
  \simpled(x)=
  \sum_{P\in\mathcal{P}(x)}\prod_{i<n_{P}}P(i)\cdot\simpled_{\li}(P(n_{P})),
  \label{eq:path-der}
\end{equation}
provided the terms on right-hand side are ``summable'' in the sense of
the Hahn field structure of $\no$.

Guided by this observation, we use the right-hand side of
\prettyref{eq:path-der} as the \emph{definition} of $\simpled(x)$ for
$x\in\bracket{\li}$.  For general $x\in\no$, we define $\simpled(x)$
using the same equation, but discarding the paths $P$ of $x$ that
never enter $\li$. Our problem is now reduced to showing that the
above sum is indeed summable, so that $\simpled(x)$ is well-defined.

A formally similar problem was tackled by Schmeling in
\cite{Schmeling2001} in order to extend derivations on transseries
fields to their exponential extensions. For our problem, we use some
of his techniques, even though our starting function $\simpled_{\li}$
is not a derivation, and $\no$ cannot be seen as an exponential
extension. Most importantly, however, we need to verify that $\no$ is
a field of transseries.

The key step is proving the existence of a suitable ordinal valued
function $\ntrank:\no\to\on$, which we call ``nested truncation rank''
(\prettyref{thm:nested-trunc-simplifies},
\prettyref{def:nested-rank}).  Using the inequality
\prettyref{eq:val-ineq} and the fact that the values of
$\simpled_{\li}$ are monomials, we are able to prove by induction on
the rank that the terms in \prettyref{eq:path-der} are summable, and
therefore $\simpled$ is well-defined. It is then easy to prove that
$\simpled$ is also a surreal derivation, proving Theorem A. Theorem C
also follows easily from the existence of the rank.

\medskip

The study of the rank requires an in-depth investigation of the
behavior of $\exp$ with respect to the simplicity relation $\simple$;
this is non-trivial, as in general it is not true that if $x$ is
simpler than $y$, then $\exp(x)$ is simpler than $\exp(y)$ (e.g.,
$\omega$ is simpler than $\log(\omega)$, but $\exp(\omega)$ is not
simpler than $\exp(\log(\omega))=\omega$). To carry out this analysis,
we provide a short characterization of $\exp$ which may be of
independent interest (\prettyref{thm:characterization-exp}), and prove
that surreal numbers simplify under some natural operations that we
call ``nested truncations'' (see \prettyref{def:nested-trunc} and
\prettyref{thm:nested-trunc-simplifies}).  For instance, the classical
truncation of a series to one of its initial segments is a special
case of nested truncation. Moreover, if $\gamma$ is a nested
truncation of some $\delta\in\J$, then $\exp(\gamma)$ is also a nested
truncation of $\exp(\delta)$.

Since simplicity is well-founded, nested truncations are well-founded
as well. We then define the rank $\ntrank:\no\to\on$ as the foundation
rank of nested truncations (\prettyref{def:nested-rank}). Thus in
particular the rank strictly decreases under non-trivial nested
truncations.  The properties of the rank are crucial to show that
$\simpled$ is well-defined. For instance, the numbers of rank $0$ are
exactly the elements of $\pm\li^{\pm1}\cup\R$
(\prettyref{cor:nested-rank-zero}), on which $\simpled$ can be easily
calculated using $\simpled_{\li}$.  On the other hand, if
$\gamma\in\J$, then $\gamma$ and $\exp(\gamma)$ have the same rank
(\prettyref{prop:ntrank-gamma-exp-gamma}).

\medskip

The existence of the rank $\ntrank$ is essentially equivalent to the
fact that $\no$ has a suitable structure of field of transseries as it
was variously conjectured, so it is worth discussing the critical
axioms in some detail. We have already commented on the fact that
$\bracket{\li}$ is the largest subfield of $\no$ satisfying axiom ELT4
of \cite{Kuhlmann2014}, but that unfortunately
$\bracket{\li}\subsetneq\no$.  On the other hand, $\no$ satisfies a
similar but weaker axiom isolated under the name ``T4'' in
\cite{Schmeling2001}, where it is given as part of an axiomatization
of a more general notion of transseries inspired by the nested
expressions of \cite{VanderHoeven1997}. In the context of the surreal
numbers, T4 reads as follows:
\begin{itemize}
\item[\textbf{T4.}] \textbf{}for all sequences of monomials
  $\m_{i}\in\exp(\J)$, with $i\in\N$, such that
  \[
  \m_{i}=\exp(\gamma_{i+1}+r_{i+1}\m_{i+1}+\delta_{i+1})
  \]
  where $r_{i+1}\in\R^{*}$, $\gamma_{i+1},\delta_{i+1}\in\J$, and
  $\supp(\gamma_{i+1})>\supp(r_{i+1}\m_{i+1})>\supp(\delta_{i+1})$
  (where $\supp$ denotes the support), there exists $k\in\N$ such that
  $r_{i+1}=\pm1$ and $\delta_{i+1}=0$ for all $i\geq k$.
\end{itemize}
Note that the sequence $P(n):=r_{n}\m_{n}$ (with $r_{0}:=1$) is a path
of $\m_{0}$, according to the definition given above. The condition
$\delta_{i+1}=0$ states that a path must stop bifurcating on the
``right'' from a certain point on. (The aforementioned axiom ELT4
further prescribes that $\gamma_{i+1}=0$, in which case the path
eventually stops bifurcating on both sides, and therefore enters
$\li$.) An immediate consequence of the existence of the rank
$\ntrank$ is that T4 holds in the surreal numbers, from which Theorem
C follows.

\medskip

The proof of Theorem B, namely that $\simpled$ is surjective, is based
on other techniques, and it relies on some further properties of the
function $\simpled$. Indeed, to prove Theorem B we first verify that
$\simpled$ satisfies the hypothesis of a theorem of Rosenlicht
\cite{Rosenlicht1983}, so that we may establish the existence of
asymptotic integrals (\prettyref{prop:class-asy-int}), and then we
iterate asymptotic integration transfinitely many times and prove,
using a version of Fodor's lemma, that the procedure eventually yields
an integral. We shall also give an example of a rather natural surreal
derivation that is not surjective (see
\prettyref{def:natural-nonsimple-pre-D}).

\medskip

The structure of the paper is reflected in the following table of
contents.

\makeatletter\@starttoc{toc}{}\makeatother

\begin{acknowledgement*}
  The authors want to thank the referee for the very careful reading
  of the manuscript. This work was completed while the first author
  was a Leverhulme Visiting Professor (VP2-2013-055) at the School of
  Mathematical Sciences, Queen Mary University of London. He wishes to
  thank the Leverhulme Trust for the support and Dr. Ivan
  Toma\v{s}i\'{c} for the invitation and the VP application.
\end{acknowledgement*}

\section{Surreal numbers}

We assume some familiarity with the ordered field of surreal numbers
(see \cite{Conway1976,Gonshor1986}) which we denote by $\no$. In this
section we give a brief presentation of the basic definitions and
results, and we fix the notations that will be used in the rest of the
paper.

\subsection{Order and simplicity}

The usual definition of the class $\no$ of surreal numbers is by
transfinite recursion, as in \cite{Conway1976}. However, it is also
possible to give a more concrete equivalent definition, as in
\cite{Gonshor1986}.

The domain of $\no$ is the \emph{class} $\no=2^{<\on}$ of all binary
sequences of some ordinal length $\alpha\in\on$, namely the functions
of the form $s:\alpha\to2=\{0,1\}$ (Gonshor writes ``$-,+$'' instead
of ``$0,1$''). The \textbf{length} (also called \textbf{birthday}) of
a surreal number $s$ is the ordinal number $\alpha=\dom(s)$.  Note
that $\no$ is not a set but a proper class, and all the relations and
functions we shall define on $\no$ are going to be class-relations and
class-functions, usually constructed by transfinite induction.

We say that $x\in\no$ is \textbf{simpler} than $y\in\no$ if
$x\subseteq y$, i.e., if $x$ is an initial segment of $y$ as a binary
sequence; we shall write $x\simpleq y$ when this is the case. We say
that $x$ is \textbf{strictly simpler} than $y,$ written $x\simple y$,
if $x\simpleq y$ and $x\neq y$. Note that $\simpleq$ is well-founded,
and the empty sequence, which will play the role of the number zero,
is simpler than any other surreal number. Moreover, the simplicity
relation is a binary tree-like partial order on $\no$, with the
immediate successors of a node $x\in\no$ being the sequences
$x^{\smallfrown}0$ and $x^{\smallfrown}1$ obtained by appending $0$ or
$1$ at the end of the binary sequence $x$.

We can introduce a total order $<$ on $\no$ in the following way.
Writing $x^{\smallfrown}y$ for the concatenation of binary sequences,
we stipulate that $x^{\smallfrown}0<x<x^{\smallfrown}1$ and more
generally
$x^{\smallfrown}0^{\smallfrown}u<x<x^{\smallfrown}1^{\smallfrown}v$
for every $u,v$. This defines a total order on $\no$ which coincides
with the lexicographic order on sequences of the same length.

We say that a \emph{subclass} $C$ of $\no$ is \textbf{convex} if
whenever $x<y$ are in $C$, every surreal number $z$ such that $x<z<y$
is also in $C$. It is easy to see that every non-empty convex class
contains a simplest number (given by the intersection $\bigcap C$).

Given two \emph{sets} $A\subseteq\no$ and $B\subseteq\no$ with $A<B$
(meaning that $a<b$ for all $a\in A$ and $b\in B)$, the \emph{class}
\[
\convex AB:=\left\{ y\in\no\suchthat A<y<B\right\}
\]
is non-empty and convex and therefore it contains a simplest number
$x$ which is denoted by
\[
x=A\mid B.
\]
Such a pair $A\mid B$ is called a \textbf{representation of $x$}, and
we call $\convex AB$ the \textbf{associated convex class}.

Every surreal number $x$ has several different representations
$x=A\mid B=A'\mid B'$.  For instance, if $A$ is cofinal with $A'$ and
$B$ is coinitial with $B'$, then clearly $\convex AB=\convex{A'}{B'}$.
In this situation, we shall say that $A\mid B=A'\mid B'$ \textbf{by
  cofinality}.  On the other hand, it may well happen that
$A\mid B=A'\mid B'$ even if $A$ is not cofinal with $A'$ or $B$ is not
coinitial with $B'$, because two distinct convex classes may still
have the same simplest number in common. The \textbf{canonical}
representation $x=A\mid B$ is the unique one such that $A\cup B$ is
exactly the set of all surreal numbers strictly simpler than $x$.

\begin{rem}
  \label{rem:simple-test}By definition, if $x=A\mid B$ and $A<y<B$,
  then $x\simpleq y$.
\end{rem}

However, it does not follow from $x=A\mid B$ and $x\simpleq y$ that
$A<y<B$.

\begin{defn}
  \label{def:simple-rep}We call a representation $x=A\mid B$
  \textbf{simple} if $x\simpleq y$ implies $A<y<B$.
\end{defn}

In other words, a representation is simple when the associated convex
class $\convex AB$ is as large as possible. An example of simple
representation is the canonical one. In fact, we have the following.

\begin{prop}
  \label{prop:simple-rep}Let $c,x,y\in\no$. We have:
  \begin{enumerate}
  \item if $c\simple x\simpleq y$, then $c<x$ if and only if $c<y$;
  \item if $x=A\mid B$, and $A\cup B$ contains only numbers strictly
    simpler than $x$, then $A\mid B$ is simple; in particular, every
    surreal number admits a simple representation (for instance the
    canonical one).
  \end{enumerate}
\end{prop}
\begin{proof}
  Point (1) follows at once from the definition of $<$. For (2), let
  $x=A\mid B$ be as in the hypothesis and let $c\in A\cup B$. We have
  $c\simple x$. If we now assume $x\simpleq y$, then, by (1),
  $c<x\leftrightarrow c<y$.  Since $c\in A\cup B$ was arbitrary, we
  obtain $A<y<B$, and therefore $A\mid B$ is simple.
\end{proof}

\subsection{Field operations}

We can define ring operations $+,\cdot$ on $\no$ by induction on
simplicity as follows:
\[
x+y:=\{x'+y,\ x+y'\}\mid\{x''+y,\ x+y''\}
\]
where $x'$ ranges over the numbers simpler than $x$ such that $x'<x$
and $x''$ ranges over the numbers simpler than $x$ such that $x<x''$;
in other words, when $x=\{x'\}\mid\{x''\}$ and $y=\{y'\}\mid\{y''\}$
are the canonical representations of $x$ and $y$ respectively.

The definition of the product is slightly more complicated:
\[
xy:=\{x'y+xy'-x'y',\ x''y+xy''-x''y''\}\mid\{x'y+xy''-x'y'',\
x''y+xy'-x''y'\}.
\]
The first expression in the left bracket ensures $x'y+xy'-x'y'<xy$,
namely $(x-x')(y-y')>0$. The mnemonic rule for the other expressions
can be obtained in a similar way.

\begin{rem}
  \label{rem:uniform}The definitions of sum and product are
  \textbf{uniform} in the sense of \cite[p.\ 15]{Gonshor1986}, namely
  the equations that define $x+y$ and $xy$ continue to hold if we
  choose arbitrary representations $x=A\mid B$ and $y=C\mid D$ (not
  necessarily the canonical ones) and we let $x',x'',y',y''$ range
  over $A,B,C,D$ respectively.
\end{rem}

It is well know that these operations, together with the order, make
$\no$ into an ordered field, which is in fact a real closed field (see
\cite[Thm.\ 5.10]{Gonshor1986}). In particular, we have a unique
embedding of the rational numbers in $\no$, so we can identify $\Q$
with a subfield of $\no$. The subgroup of the \textbf{dyadic
  rationals} $\frac{m}{2^{n}}\in\Q$, with $m,n\in\N$, correspond
exactly to the surreal numbers $s:k\to\{0,1\}$ of finite ordinal
length $k\in\N$.

The \textbf{real numbers} $\R$ can be isomorphically identified with a
subfield of $\no$ by sending $z\in\R$ to the number $A\mid B$ where
$A\subseteq\no$ is the set of rationals $<z$ and $B\subseteq\no$ is
the set of rationals $>z$. This turns out to be a homomorphism, and
therefore it agrees with the previous definition for $z\in\Q$.  We may
thus write $\Q\subset\R\subset\no$. By \cite[p.\ 33]{Gonshor1986}, the
length of a real number is at most $\omega$ (the least infinite
ordinal). There are however surreal numbers of length $\omega$ which
are not real numbers.

The \textbf{ordinal numbers} can be identified with a subclass of
$\no$ by sending the ordinal $\alpha$ to the sequence
$s:\alpha\to\{0,1\}$ with constant value $1$. Under this
identification, the ring operations of $\no$, when restricted to the
ordinals $\on\subset\no$, coincide with the Hessenberg sum and product
of ordinal numbers. On the other hand, the sequence
$s:\alpha\to\{0,1\}$ with constant value $0$ corresponds to the
additive inverse of the ordinal $\alpha$, namely $-\alpha$. We remark
that $x\in\on$ if and only if $x$ admits a representation of the form
$x=A\mid B$ with $B$ empty, and similarly $x\in-\on$ if and only if we
can write $x=A\mid B$ with $A$ empty.  Under the above identification
of $\Q$ as a subfield of $\no$, the natural numbers $\N\subseteq\Q$
are exactly the finite ordinals.

\subsection{\label{sub:Hahn-fields}Hahn fields}

Let $\Gamma$ be an ordered abelian group, written multiplicatively.
We recall the definition of the \textbf{Hahn field} $\R((\Gamma))$
with coefficients in $\R$ and ``monomials'' in $\Gamma$. The domain of
$\R((\Gamma))$ consists of all the functions $f:\Gamma\to\R$ whose
\textbf{support} $\supp(f):=\{\m\in\Gamma\ :\ f(\m)\neq0\}$ is a
\textbf{reverse well-ordered} subset of $\Gamma$, namely every
non-empty subset of $\Gamma$ has a maximum (when $\Gamma$ is a proper
class, we still require that $\supp(f)$ is a set). For each $f$ which
is not identically $0$, $\supp(f)$ has a maximum element $\m$; if
$f(\m)>0$, we say that $f$ is positive. For later reference, given
$\m\in\M$, the \textbf{truncation} of $f$ at $\m$ is the function
$f|\m:\M\to\R$ which coincides with $f$ on arguments $>\m$ and is zero
on arguments $\leq\m$.

The addition of two elements $f,g\in\R((\Gamma))$ is defined as
\[
(f+g)(\m):=f(\m)+g(\m),
\]
and the multiplication is given by
\[
(f\cdot g)(\m):=\sum_{\n+\om=\m}f(\n)g(\om).
\]
Since the supports are reverse well-ordered, only finitely many terms
of the latter sum can be non-zero, hence the multiplication is well
defined. These operations make $\R((\Gamma))$ into an ordered field
(which is real closed when $\Gamma$ is divisible).

It is well known that $\no$ can be endowed with a Hahn field
structure.  Towards this goal, recall that two non-zero surreal
numbers $x,y\in\no^{*}$ are in the same \textbf{archimedean class} if
each of them is bounded in modulus by an integer multiple of the
other, namely $|x|\leq k|y|$ and $|y|\leq k|x|$ for some $k\in\N$.

A positive surreal number $x\in\no^{*}$ is called a \textbf{monomial}
if it is the simplest positive element in its archimedean class. The
class $\M\subset\no^{*}$ of all monomials is a group under
multiplication (see \prettyref{fact:omega}). A \textbf{term} is a
non-zero real number $r\in\R^{*}$ multiplied by a monomial; we denote
by $\T$ the class of all terms.

One of Conway's remarkable insights is that we can identify $\no$ with
$\R((\M))$ in the following way. Given $f\in\R((\M))$, write $f_{\m}$
for the real number $f(\m)$. Note that $f_{\m}\m=f(\m)\m\in\R\M$ is a
well defined element of $\no$. We define the map
$\sum:$$\R((\M))\to\no$ by induction on the order type of the support.

\begin{defn}
  \label{def:infinite-sum}Let $f\in\R((\M))$.
  \begin{enumerate}
  \item If the support of $f$ is empty, we define $\sum f:=0\in\no$.
  \item If the support of $f$
    contains a smallest monomial $\n$, we define
    \[
    \sum f:=\sum f|\n+f_{\n}\n.
    \]
  \item If the support of $f$
    is non-empty and has no smallest monomial, we define
    \[
    \sum f:=\left\{ \sum f|\m+q'\m\right\} \mid\left\{ \sum f|\m+q''\m\right\} 
    \]
    where $\m$
    varies in $\supp(f)$
    and $q'$,
    $q''$ varies among the rational numbers such that $q'<f_{\m}<q''$.
  \end{enumerate}
\end{defn}

We remark that in (3), for $\sum f$ to be well defined, one needs to
show by a simultaneous induction that each number on the left-hand
side is smaller than each number on the right-hand side (see \cite[p.\
59]{Gonshor1986} for a detailed argument).

By \cite[Lemma 5.3]{Gonshor1986}, if $\supp(f)$ contains a smallest
element $\m$, then $\sum f=\sum f|\m+f_{\m}\m$ can be characterized as
the simplest surreal number such that, for every $q',q''\in\Q$ with
$q'<f_{\m}<q''$, we have $\sum f|\m+q'\m<\sum f<\sum f|\m+q''\m$.  It
then follows that the three cases in \prettyref{def:infinite-sum} can
be subsumed under a single equation, as follows.

\begin{prop}
  \label{prop:sum-simplest} For every $f\in\R((\M))$ we have
  \[
  \sum f=\left\{ \sum f|\m+q'\m\right\} \mid\left\{ \sum f|\m+q''\m\right\} 
  \]
  where $\m$ varies in $\supp(f)$ and $q'$, $q''$ varies among the
  rational numbers such that $q'<f_{\m}<q''$.
\end{prop}

This also holds when $\supp(f)=\emptyset$; indeed, in this case
$\sum f$ is just the simplest surreal numbers satisfying the empty set
of inequalities, hence $\sum f=0$. We could in fact take the above
equation as the definition of $\sum f$, but then it would be more
difficult to verify that $\sum(f+g)=\sum f+\sum g$.

As a matter of notations, we write $\sum_{\m\in\M}f_{\m}\m$ for
$\sum f$, namely we think of $\sum f$ as a decreasing formal infinite
sum of terms $f_{\m}\m$ with reverse well-ordered support. It turns
out that the map $\sum:\R((\M))\to\no$ is an isomorphism of ordered
fields (in particular it is surjective), so we can identify
$f\in\R((\M))$ with $\sum f=\sum_{\m}f_{\m}\m\in\no$ and write
\[
\no=\R((\M)).
\]

A short proof of surjectivity is in Conway's book \cite[pp.\
32-33]{Conway1976}, which however should be integrated with some
details that can be found in Gonshor (in particular \cite[Lemmas 5.2
and 5.3]{Gonshor1986}).  We remark that Conway and Gonshor prove the
result in the opposite direction, by defining a map $\no\to\R((\M))$
which is the inverse of our $\sum:\R((\M))\to\no$. For a full proof
see \cite[Thm. 5.6]{Gonshor1986}.

Under the identification $f=\sum f$ we can drop the summation sign in
\prettyref{prop:sum-simplest}. For instance, when $f$ is a single
monomial $\m$ the equation reduces to
\[
\m=\left\{ q'\m\right\} \mid\{q''\m\}
\]
where $q',q''$ range over the rational numbers with $q'<1<q''$.

The identification $\no=\R((\M))$ makes it possible to extend to $\no$
the various notions that are given on Hahn fields:
\begin{defn}
  Let $x\in\no$ and write $x=\sum_{\m}x_{\m}\m$.
  \begin{enumerate}
  \item The \textbf{support} $\supp(x)$ of $x$ is the support of the
    corresponding element of $\R((\M))$, namely
    $\supp(x):=\{\m\in\M\suchthat x_{\m}\neq0\}$.
  \item The \textbf{terms} of $x$ are the numbers in the set
    $\{x_{\m}\m\suchthat x_{\m}\neq0\}\subset\T$.
  \item The \textbf{coefficient} of $\m$ in $x$ is $x_{\m}$.
  \item The \textbf{leading monomial} of $x$ is the maximal monomial
    in $\supp(x)$.
  \item The \textbf{leading term} of $x$ is the leading monomial
    multiplied by its coefficient.
  \item Given $\n\in\M$, the \textbf{truncation} of $x$ at $\n$ is the
    number $x|\n:=\sum_{\m>\n}x_{\m}\m$. If $y\in\no$ is a truncation
    of $x$, we write $y\trunceq x$, and $y\trunc x$ if moreover
    $x\neq y$.
  \end{enumerate}
\end{defn}

The relation $\trunceq$ is clearly a partial order with a tree-like
structure, and it is actually a weakening of the simplicity relation.

\begin{prop}
  \label{prop:trunc-simple}If $x\trunceq y$, then $x\simpleq y$.
\end{prop}

In \cite[Thm.\ 5.12]{Gonshor1986} this statement obtained as a
corollary of an explicit calculation of the binary sequence
corresponding to an infinite sum, but it can also be immediately
deduced from \prettyref{prop:sum-simplest}.  We include the proof to
illustrate the method, as it will be applied again in the sequel.

\begin{proof}[Proof of \prettyref{prop:trunc-simple}]
  Given $x\in\no$, by \prettyref{prop:sum-simplest} we can write
  $x=A\mid B$ where
  \[
  A=\left\{ x|\n+q'\n\right\} ,\quad B=\left\{ x|\n+q''\n\right\}
  \]
  with $\n$ varying in $\supp(x)$, and $q'$, $q''$ varying among the
  rational numbers such that $q'<x_{\n}<q''$. Similarly, $y=A'\mid B'$
  where
  \[
  A'=\left\{ y|\n+q'\n\right\} ,\quad B'=\left\{ y|\n+q''\n\right\}
  \]
  with $\n\in\supp(y)$ and $q'<y_{\n}<q''$ . Since $x\trunceq y$ we
  have $\supp(x)\subseteq\supp(y)$, and for every $\n\in\supp(x)$ we
  have $x|\n=y|\n$ and $x_{\n}=y_{\n}$. It follows that
  $A\subseteq A'$ and $B\subseteq B'$, hence $x\simpleq y$.
\end{proof}

\subsection{Summability}

The identification $\no=\R((\M))$ makes it possible to extend to $\no$
the notion of infinite sum.

\begin{defn}
  \label{def:summability}Let $I$ be a set and
  $(x_{i}\suchthat i\in I)$ be an indexed family of surreal
  numbers. We say that $(x_{i}\suchthat i\in I)$ is \textbf{summable}
  if $\bigcup_{i}\supp(x_{i})$ is reverse well-ordered and if for each
  $\m\in\bigcup_{i}\supp(x_{i})$ there are only finitely many $i\in I$
  such that $\m\in\supp(x_{i})$.

  When $(x_{i}\suchthat i\in I)$ is summable, we define its
  \textbf{sum} $y:=\sum_{i\in I}x_{i}$ as the unique surreal number
  such that $\supp(y)\subseteq\bigcup_{i}\supp(x_{i})$ and, for every
  $\m\in\M$,
  \[
  y_{\m}=\left(\sum_{i\in I}x_{i}\right)_{\m}=\sum_{i\in
    I}(x_{i})_{\m}.
  \]
\end{defn}

By assumption, for each $\m$ there are finitely many $i$ such that
$(x_{i})_{\m}\neq0$, so that each $y_{\m}$ is a well defined real
number.

The result is coherent with our previous definitions: if $x\in\no$,
the family $(x_{\m}\m\suchthat\m\in\supp(x))$ is obviously summable,
and its sum $\sum_{\m}x_{\m}\m$ in the just defined sense is exactly
$x$.

\begin{rem}
  \label{rem:summability-criterion}The following criterion for
  summability follows at once from the definition:
  $(x_{i}\suchthat i\in I)$ is summable if and only if there are no
  injective sequence $n\mapsto i_{n}\in I$ and monomials
  $\m_{n}\in\supp(x_{i_{n}})$ such that $\m_{n}\leq\m_{n+1}$ for every
  $n\in\N$.
\end{rem}

\begin{rem}
  It can be verified that infinite sums are infinitely associative and
  distributive over products, see \cite{Gonshor1986} for the details.
\end{rem}

\begin{defn}
  A function $F:\no\to\no$ is \textbf{strongly linear} if for all
  $x=\sum_{\m}x_{\m}\m$ we have $F(x)=\sum_{\m}x_{\m}F(\m)$ (in
  particular, $(x_{\m}F(\m)\suchthat\m\in\M)$ is summable).
\end{defn}

\begin{prop}
  \label{prop:strong-linearity}If $F:\no\to\no$ is strongly linear,
  then for any summable $(x_{i}\suchthat i\in I)$ the family
  $(F(x_{i})\suchthat i\in I)$ is summable and
  \[
  F\left(\sum_{i\in I}x_{i}\right)=\sum_{i\in I}F(x_{i}).
  \]
\end{prop}
\begin{proof}
  We have
  \[
  F\left(\sum_{i\in
      I}x_{i}\right)=F\left(\sum_{\m\in\M}\left(\sum_{i\in
        I}x_{i}\right)_{\m}\m\right)=\sum_{\m\in\M}\sum_{i\in
    I}(x_{i})_{\m}F(\m)=\sum_{i\in I}F(x_{i}).\qedhere
  \]
\end{proof}

\subsection{Purely infinite numbers}

We use Hardy's notation ``$\vleq$'' for the \textbf{dominance}
relation.

\begin{defn}
  Given $x,y\in\no$ we write
  \begin{itemize}
  \item $x\vleq y$ if $|x|\leq k|y|$ for some $k\in\N$;
  \item $x\vless y$ if $|x|<\frac{1}{k}|y|$ for all positive $k\in\N$;
  \item $x\veq y$ if $\frac{1}{k}|y|\leq|x|\leq k|y|$ for some
    positive $k\in\N$;
  \item $x\sim y$ if $x-y\vless x$ (equivalently
    $\left|1-\frac{y}{x}\right|\vless1$ when $x\neq0$).
  \end{itemize}
\end{defn}

We say that $x\in\no$ is \textbf{finite} if $x\vleq1$. We say that $x$
is \textbf{infinitesimal} if $x\vless1$. We shall denote the class of
all infinitesimal numbers by $o(1)$. In general, we denote by $o(x)$
the class of all $y\in\no$ such that $\frac{y}{x}$ is infinitesimal,
namely such that $y\vless x$. Note that $x\veq y$ if and only if $x$
and $y$ are in the same archimedean class.

We say that $x\in\no$ is \textbf{purely infinite} if all the monomials
$\m\in\M$ in its support are infinite (or equivalently $\supp(x)>1)$.
The non-unital ring of purely infinite numbers of $\no$ shall be
denoted by $\J$. We have a direct sum decomposition
\[
\no=\J+\R+o(1)
\]
as a real vector space.

The \textbf{surreal integers} are the numbers in $\J+\Z$. They
coincide with Conway's ``omnific integers'', namely the numbers $x$
such that $x=\{x-1\}\mid\{x+1\}$.

\subsection{The omega-map}

Another remarkable feature of surreal numbers is that the class $\M$
of monomials can be parametrized in a rather canonical way by the
surreal numbers themselves.

\begin{defn}[{\cite[p.\ 31]{Conway1976}}]
  \label{def:omega}Given $x\in\no,$ we let
  \[
  \omega^{x}:=\{0,k\omega^{x'}\}\mid\left\{
    \frac{1}{2^{k}}\omega^{x''}\right\}
  \]
  where $k$ ranges over the natural numbers, $x'$ ranges over the
  surreal numbers such that $x'\simple x$ and $x'<x$, and $x''$ over
  the surreal numbers such that $x''\simple x$ and $x<x''$.
\end{defn}

As in the case of the ring operations, the\textbf{ }above definition
is uniform, namely if $x=A\mid B$ is any representation of $x$, the
equation in the definition of $\omega^{x}$ remains true if we let
$x',x''$ range over $A,B$ respectively. The following remark follows
at once from the fact that for $x\simpleq y$ the convex class
associated to the above representation of $\omega^{x}$ includes the
convex class associated to the corresponding representation of
$\omega^{y}$.

\begin{rem}
  \label{rem:omega-simple}If $x\simpleq y$, then
  $\omega^{x}\simpleq\omega^{y}$.
\end{rem}

\begin{fact}[{\cite[Thms.\ 19 and 20]{Conway1976}, \cite[Thms.\ 5.3
    and 5.4]{Gonshor1986}}]
  \label{fact:omega}The map $x\mapsto\omega^{x}$ is an isomorphism
  from $(\no,+,<)$ to $(\M,\cdot,<)$. In particular, $\omega^{x}$ is
  the simplest positive representative of its archimedean class,
  $\omega^{0}=1$ and $\omega^{x+y}=\omega^{x}\cdot\omega^{y}$.
\end{fact}

From the equalities $\no=\R((\M))$ and $\M=\omega^{\no}$ we obtain
\[
\no=\R((\omega^{\no})),
\]
namely every surreal number $x$ can be written uniquely in the form
\[
x=\sum_{y\in\no}a_{y}\omega^{y}
\]
where $a_{y}\in\R$ and $a_{y}\neq0$ if and only if $\omega^{y}$ is in
the support of $x$. This representation is called the \textbf{normal
  form} of $x$ and it coincides with Cantor's normal form when
$x\in\on\subset\no$.

\section{Exponentiation}

\subsection{Gonshor's exponentiation}

The surreal numbers admit a well behaved exponential function defined
as follows.

\begin{defn}[{\cite[p.\ 145]{Gonshor1986}}]
  \label{def:exp}Let $x=\{x'\}\mid\{x''\}$ be the canonical
  representation of $x$. We define inductively
  \[
  \exp(x):=\left\{
    0,\,\exp(x')\cdot[x-x']_{n},\,\exp(x'')[x-x'']_{2n+1}\right\}
  \mid\left\{
    \frac{\exp(x'')}{[x''-x]_{n}},\,\frac{\exp(x')}{[x'-x]_{2n+1}}\right\},
  \]
  where $n$ ranges in $\N$ and
  \[ [x]_{n}=1+\frac{x}{1!}+\dots+\frac{x^{n}}{n!},
  \]
  with the further convention that the expressions containing terms of
  the form $[y]_{2n+1}$ are to be considered only when $[y]_{2n+1}>0$.
\end{defn}

It can be shown that the function $\exp$ is a surjective homomorphism
$\exp:(\no,+)\to(\no^{>0},\cdot)$ which extends $\exp$ on $\R$ and
makes $(\no,+,\cdot,\exp)$ into an elementary extension of
$(\R,+,\cdot,\exp)$ (see \cite[Cor.\ 2.11, Cor.\ 4.6]{Dries1994},
\cite{DriesE2001} and \cite{Ressayre1993}).

We recall here a list of properties that were proved in
\cite{Gonshor1986}.  We shall use them to give an alternative
characterization of the function $\exp$.

\begin{fact}
  \label{fact:exp}The function $\exp$ has the following properties.
  \begin{enumerate}
  \item \prettyref{def:exp} is uniform \cite[Cor. 10.1]{Gonshor1986}.
  \item The restriction of $\exp$ to $\R\subseteq\no$ is the real
    exponential function \cite[Thm.\ 10.2]{Gonshor1986}.
  \item If $\eps\vless1$, then the sequence
    $(\frac{\eps^{n}}{n!}\,:\,n\in\N)$ is summable and
    $\exp(\eps)=\sum_{n}\frac{\eps^{n}}{n!}$ \cite[Thm.\
    10.3]{Gonshor1986}.
  \item The function $\exp$ is an isomorphism from $(\no,+,<)$ to
    $(\no^{>0},\cdot,<)$.  In particular, $\exp(0)=1$ and
    $\exp(x+y)=\exp(x)\cdot\exp(y)$ for every $x,y\in\no$ \cite[Cor.\
    10.1]{Gonshor1986}.
  \item If $x>0$, then $\exp(\omega^{x})=\omega^{\omega^{g(x)}}$,
    where $g:\no^{>0}\to\no$ is defined by
    \[
    g(x):=\left\{ c(x),g(x')\right\} |\left\{ g(x'')\right\} ,
    \]
    and $c(x)$ is the unique number such that $\omega^{c(x)}$ and $x$
    are in the same archimedean class \cite[Thm.\ 10.11]{Gonshor1986}.
  \item If $x=\sum_{y}a_{y}\omega^{y}$ is purely infinite, then
    \cite[Thm.\ 10.13]{Gonshor1986}
    \[
    \exp\left(\sum_{y}a_{y}\omega^{y}\right)=\omega^{\sum_{y}a_{y}\omega^{g(y)}}.
    \]
  \end{enumerate}
\end{fact}

\begin{defn}
  Let $\log:\no^{>0}\to\no$ (called \textbf{logarithm}) be the inverse
  of $\exp$.

  We let $\exp_{n}$ and $\log_{n}$ be the $n$-fold iterated
  compositions of $\exp$ and $\log$ with themselves.
\end{defn}

\begin{rem}
  \label{rem:expansion-log}One can easily verify that if
  $\varepsilon\vless1$ we must have
  \[
  \log(1+\varepsilon)=\sum_{n=1}^{\infty}(-1)^{n-1}\frac{\varepsilon^{n}}{n}.
  \]
\end{rem}

In the sequel we shall make repeated use of the fact that $\exp$ grows
more than any polynomial. In particular we have:

\begin{rem}
  \label{rem:exp-grows}If $x>\N$, then $\exp(x)\vgreater x^{n}$ for
  every $n\in\N$.
\end{rem}

\subsection{Ressayre form}

By \cite[Thms.\ 10.8, 10.9]{Gonshor1986}, the monomials are the image
under $\exp$ of the purely infinite surreal numbers:
\[
\exp(\J)=\M=\omega^{\no}.
\]
Since $\no=\R((\M))=\R((\omega^{\no}))$, it then follows that
$\no=\R((\exp(\J)))$ as well, namely every surreal number $x\in\no$
can be written uniquely in the form
\[
x=\sum_{\gamma\in\J}r_{\gamma}\exp(\gamma)
\]
where $r_{\gamma}\neq0$ if and only if $\exp(\gamma)\in\supp(x)$.  We
call this the \textbf{Ressayre form} of $x\in\no$. We stress the fact
that, unlike the case of normal form
$x=\sum_{y\in\no}a_{y}\omega^{y}$, the summation is indexed by
elements of $\J$.

\begin{defn}
  Given a non-zero number $x=\sum_{\gamma\in\J}r_{\gamma}\exp(\gamma)$
  we define $\vell(x)\in\J$ as the maximal $\gamma$ such that
  $r_{\gamma}\neq0$, or in other words, as the logarithm $\gamma$ of
  the largest monomial $\m=\exp(\gamma)$ in its support.
\end{defn}

It is easy to verify that $\vell:\no\setminus\{0\}\to\J$ satisfies:

\begin{enumerate}
\item $\vell(x+y)\leq\max\{\vell(x),\vell(y)\},$ with the equality
  holding if $\vell(x)\neq\vell(y)$;
\item $\vell(xy)=\vell(x)+\vell(y)$.
\end{enumerate}

This makes $-\vell$ into a Krull valuation. We call $\vell(x)$ the
\textbf{$\vell$-value} of $x$.

\begin{rem}
  \label{rem:vell}Given $x,y\in\no^{*}$ we have
  \begin{itemize}
  \item $x\vleq y$ if and only if $\vell(x)\leq\vell(y)$,
  \item $x\vless y$ if and only if $\vell(x)<\vell(y)$,
  \item $x\veq y$ if and only if $\vell(x)=\vell(y)$,
  \item $x\sim y$ if and only if $\vell(x-y)<\vell(x)$.
  \end{itemize}
  In particular, if $x\vgreater1$, then $\vell(x)>0$, hence
  $\vell(x)\vgreater1$.  We also observe that if $x\not\veq1$, then
  $\vell(x)\sim\log(x)$.
\end{rem}

\subsection{A characterization of $\exp$}

In order to understand the interaction between $\exp$ and the
simplicity relation $\simple$, we first give a rather short
characterization of $\exp$. We start with Gonshor's description
\prettyref{fact:exp} and we further simplify it by dropping any
references to the omega-map or to the function $g$.

\begin{thm}
  \label{thm:characterization-exp}The function $\exp:\no\to\no$ is
  uniquely determined by the following properties:
  \begin{enumerate}
  \item if $\m\in\M^{>1}$ is an infinite monomial, then
    \[
    \exp(\m)=\left\{ \m^{k},\,\exp(\m')^{k}\right\} \mid\left\{
      \exp(\m'')^{\frac{1}{k}}\right\}
    \]
    where $k$ ranges in the positive integers and $\m',\m''$ range in
    the set of monomials simpler than $\m$ and such that respectively
    $\m'<\m$ and $\m<\m''$;
  \item if $\gamma=\sum_{\m\in\M}\gamma_{\m}\m\in\J$ is a purely
    infinite surreal number, then
    \[
    \exp(\gamma)=\left\{ 0,\,\exp(\gamma|\m)\exp(\m)^{q'}\right\}
    \mid\left\{ \exp(\gamma|\m)\exp(\m)^{q''}\right\}
    \]
    where $\m$ ranges in $\supp(\gamma)$ and $q',q''$ range among the
    rational numbers such that $q'<\gamma_{\m}<q''$;
  \item if $x=\gamma+r+\eps$, where $\gamma\in\J$, $r\in\R$ and
    $\eps\in o(1)$, then
    \[
    \exp(\gamma+r+\eps)=\exp(\gamma)\cdot\exp(r)\cdot\left(\sum_{n=0}^{\infty}\frac{\eps^{n}}{n!}\right),
    \]
    where $\exp(r)$ is the value of the real exponential function at
    $r$.
  \end{enumerate}
\end{thm}

Point (1) is the minimum requirement that ensures that the
$\exp_{\restriction\M^{>1}}$ is increasing and grows more than any
power function. Point (2) shows that, for $\gamma\in\J$,
$\exp(\gamma)$ is the simplest element satisfying some natural
inequalities determined by the values of $\exp$ on the truncations of
$\gamma$. Point (3) just says that the behavior on the finite numbers
is the one given by the classical Taylor expansion of $\exp$.

\begin{proof}
  (1) Let $x\in\no^{>0}$ be such that $\m=\omega^{x}.$ By
  \prettyref{fact:exp} we have
  $\exp(\m)=\exp(\omega^{x})=\omega^{\omega^{g(x)}}$. Since
  $y=g(x)=\left\{ c(x),\,g(x')\right\} \mid\left\{ g(x'')\right\} $ we
  have
  \[
  \omega^{g(x)}=\left\{ 0,\,k\omega^{c(x)},\,k\omega^{g(x')}\right\}
  \mid\left\{ \frac{1}{k}\omega^{g(x'')}\right\} ,
  \]
  where $k$ ranges in $\N^{*}$. Computing $\omega^{y}$ with
  $y=\omega^{g(x)}$ we obtain
  \[
  \omega^{\omega^{g(x)}}=\left\{
    0,\,j\omega^{0},\,j\omega^{k\omega^{c(x)}},\,j\omega^{k\omega^{g(x')}}\right\}
  \mid\left\{ \frac{1}{j}\omega^{\frac{1}{k}\omega^{g(x'')}}\right\}
  \]
  where $j$ and $k$ range in $\N^{*}$.

  By cofinality, since $k$ varies over the positive integers, we can
  drop the coefficients $j$, $\frac{1}{j}$ and the first two
  expressions; moreover, $\{k\omega^{c(x)}\suchthat k\in\N^{*}\}$ is
  cofinal with $\{kx\suchthat k\in\N^{*}\}$. We deduce that
  \[
  \omega^{\omega^{g(x)}}=\left\{
    \omega^{kx},\,\omega^{k\omega^{g(x')}}\right\} \mid\left\{
    \omega^{\frac{1}{k}\omega^{g(x'')}}\right\} .
  \]

  Now, recalling that $\omega^{ky}=(\omega^{y})^{k}$ and
  $\omega^{\omega^{g(y)}}=\exp(\omega^{y})$, we get
  \[
  \exp(\m)=\omega^{\omega^{g(x)}}=\left\{
    \m^{k},\,\exp(\omega^{x'})^{k}\right\} \mid\left\{
    \exp(\omega^{x''})^{1/k}\right\} .
  \]

  Finally, by \prettyref{rem:omega-simple}, we note that the monomials
  $\m',\m''$ simpler than $\m$ with $\m'<\m$, $\m<\m''$ are exactly
  those of the form $\omega^{x'}$, $\omega^{x''}$ with $x'$, $x''$
  simpler than $x$ and such that respectively $x'<x$ and $x<x''$, and
  we are done.

  (2) Given a purely infinite surreal number
  $\gamma=\sum_{\m\in\M}\gamma_{\m}\m=\sum_{y\in\no}a_{y}\omega^{y}$,
  let $G(\gamma):=\sum_{y}a_{y}\omega^{g(y)}$. By
  \prettyref{fact:exp}, we have $\exp(\gamma)=\omega^{G(\gamma)}$. It
  is immediate to see that $G$ is strictly increasing, strongly
  linear, surjective, and sends monomials to monomials. In particular,
  for all $\n\in\M$, $G(\gamma)|\n=G(\gamma|\m)$ where
  $\m=G^{-1}(\n)$. By \prettyref{prop:sum-simplest} it follows that
  \[
  G(\gamma)=\left\{ G(\gamma)|G(\m)+q'G(\m)\right\} \mid\left\{
    G(\gamma)|G(\m)+q''G(\m)\right\}
  \]
  where $\m$ ranges in $\supp(\gamma)$ and $q',q''$ range among the
  rational numbers such that $q'<\gamma_{\m}<q''$.

  By definition of $\omega^{y}$, setting $y=G(\gamma)$, we obtain
  \[
  \omega^{G(\gamma)}=\left\{
    0,\,k\omega^{G(\gamma)|G(\m)+q'G(\m)}\right\} \mid\left\{
    \frac{1}{k}\omega^{G(\gamma)|G(\m)+q''G(\m)}\right\}
  \]
  with $k$ ranging in $\N^{*}$ and $\m,q',q''$ as above. By
  cofinality, we can drop $k$:
  \[
  \omega^{G(\gamma)}=\left\{
    0,\,\omega^{G(\gamma)|G(\m)+q'G(\m)}\right\} \mid\left\{
    \omega^{G(\gamma)|G(\m)+q''G(\m)}\right\} .
  \]

  Since
  $\omega{}^{G(\gamma)|G(\m)+qG(\m)}=\omega^{G(\gamma|\m)}\left(\omega^{G(\m)}\right)^{q}=\exp(\gamma|\m)\exp(\m)^{q}$,
  we get
  \[
  \exp(\gamma)=\omega^{G(\gamma)}=\{0,\,\exp(\gamma|\n)\exp(\m)^{q'}\}\mid\{\exp(\gamma|\m)\exp(\m)^{q''}\},
  \]
  as desired.

  Part (3) follows easily from \prettyref{fact:exp}.
\end{proof}

\section{\label{sec:nested}Nested truncations}

Unlike the omega-map, the function $\exp$ is not monotone with respect
to simplicity, as $x\simpleq y$ does not always imply
$\exp(x)\simpleq\exp(y)$; for instance, we have
$\omega\simple\log(\omega)$ while
$\exp(\log(\omega))=\omega\simple\exp(\omega)$.  However, under some
additional assumptions $\exp$ does preserve simplicity.  For example,
$\exp$ preserves simplicity if we know that $x$ is a truncation of $y$
(this is well known, although it seems to have never been stated in
this form).

\begin{prop}
  \label{prop:trunc-exp-simplifies}Let $\gamma,\delta\in\J$. If
  $\gamma\trunceq\delta$, then
  $\exp(\gamma)\simpleq\exp(\delta)$.
\end{prop}
\begin{proof}
  The argument is similar to the one for
  \prettyref{prop:trunc-simple}, so we will be brief. As in
  \prettyref{thm:characterization-exp}(2), we can write
  $\exp(\gamma)=A\mid B$ where
  \[
  A=\left\{ 0,\,\exp(\gamma|\m)\exp(\m)^{q'}\right\},\quad B=\left\{
    \exp(\gamma|\m)\exp(\m)^{q''}\right\} ,
  \]
  with $\m\in\supp(\gamma)$ and $q'<\gamma_{\m}<q''$. Similarly,
  $\exp(\delta)=A'\mid B'$ where
  \[
  A'=\left\{ 0,\,\exp(\delta|\m)\exp(\m)^{q'}\right\},\quad B'=\left\{
    \exp(\delta|\m)\exp(\m)^{q''}\right\} ,
  \]
  with $\m\in\supp(\delta)$ and $q'<\delta_{\m}<q''$. Since
  $\gamma\trunceq\delta$, we have $A\subseteq A'$ and $B\subseteq B'$,
  and therefore $\exp(\gamma)\simpleq\exp(\delta)$, as desired.
\end{proof}

The above proposition is far from being sufficient for our purposes.
For instance, since $\exp(\gamma)\not\trunceq\exp(\delta)$ for
$\gamma\neq\delta$ in $\J$, we cannot iterate it to compare
$\exp(\exp(\gamma))$ and $\exp(\exp(\delta))$. We remedy this problem
by defining a more powerful notion of ``nested truncation''.

\begin{defn}
  We say that a sum $x_{1}+x_{2}+\dots+x_{n}$ of surreal numbers is in
  \textbf{standard form} if
  $\supp(x_{1})>\supp(x_{2})>\dots>\supp(x_{n})$.

  Given $x\in\no^{*}$, we let $\sign x:=1$ if $x>0$ and $\sign x:=-1$
  if $x<0$.
\end{defn}

\begin{defn}
  \label{def:nested-trunc}For $n\in\N$, we define $\ntrunceq_{n}$ on
  $\no^{*}$ inductively as follows:
  \begin{enumerate}
  \item $x\ntrunceq_{0}y$ if $x\trunceq y$;
  \item $x\ntrunceq_{n+1}y$ if there are $\gamma\ntrunceq_{n}\delta$
    in $\J^{*}$, $z,w\in\no$ and $r\in\R^{*}$ such that
    \[
    x=z+\sign r\exp(\gamma),\quad y=z+r\exp(\delta)+w,
    \]
    where both sums are in standard form.
  \end{enumerate}

  We say that $x\ntrunceq y$, or that $x$ is a \textbf{nested
    truncation} of $y$, if $x\ntrunceq_{n}y$ for some $n\in\N$. We
  write $x\ntrunc y$ if $x\ntrunceq y$ and $x\neq y$.
\end{defn}

It is important to observe that $\ntrunceq$ is defined only for
non-zero surreal numbers, so that $0\not\ntrunceq x$ for every
$x$. This ensures that if $x\ntrunceq\delta$ for some
$\delta\in\J^{*}$, then $x\in\J^{*}$ (see
\prettyref{prop:nested-trunc-J}).

\begin{rem}
  Even if $z+r\exp(\delta)+w$ is in standard form, and
  $\gamma\ntrunceq\delta$, the sum $z+\sign r\exp(\gamma)$ might not
  be in standard form, and therefore it might not be a nested
  truncation of $z+r\exp(\delta)+w$.
\end{rem}

\begin{rem}
  \label{rem:ntrunc-std-form}For all $x,y\in\no^{*}$ and $z\in\no$, if
  $z+x$ and $z+y$ are both in standard form, then $x\ntrunceq y$ if
  and only if $z+x\ntrunceq z+y$. However, the equivalence may not
  hold if one of the sums is not in standard form.
\end{rem}

\begin{rem}
  \label{rem:nested-trunc-sign}For all $x,y\in\no^{*}$, if
  $x\ntrunceq y$, then $x>0$ if and only if $y>0$. Moreover,
  $x\ntrunceq y$ if and only if $-x\ntrunceq-y$.
\end{rem}

\begin{rem}
  \label{rem:ntrunc-mon}For all $x\in\no^{*}$ and $\m\in\M$, if
  $x\ntrunceq\m$, then $x\in\M$.
\end{rem}

Nested truncations behave rather similarly to truncations. First of
all, like $\trunceq$, the relation $\ntrunceq$ is a partial order.
\begin{prop}
  The relation $\ntrunceq$ is a partial order on $\no^{*}$.
\end{prop}
\begin{proof}
  Reflexivity is trivial since $x\ntrunceq_{0}x$ always holds.

  For antisymmetry, assume $x\ntrunceq_{n}y$ and $y\ntrunceq x$ for
  some $n$. We prove by induction on $n$ that $x=y$. Note first that
  the supports of $x$ and $y$ must clearly have the same order type.
  In the case $n=0$, since $x\trunceq y$, this immediately implies
  that $x=y$. If $n>0$, write $x=z+\sign r\exp(\gamma)$ and
  $y=z+r\exp(\delta)+w$ in standard form with
  $\gamma\ntrunceq_{n-1}\delta$. The observation on the order type
  immediately implies that $w=0$, and since $y\ntrunceq x$, we get
  that $r=\sign r=\pm1$ and $\delta\ntrunceq\gamma$. By the inductive
  hypothesis we obtain $\gamma=\delta,$ hence $x=y$.

  For transitivity, assume $x\ntrunceq_{n}y\ntrunceq_{m}u$ for some
  $n,m\in\N$. If $m=0$, then $y\trunceq u$, from which it easily
  follows that $x\ntrunceq_{n}u$. Similarly, if $n=0$, then
  $x\trunceq y$, which implies $x\ntrunceq_{m}u$.

  If $m>0$ and $n>0$, write $y=z+\sign r\exp(\gamma)$ and
  $u=z+r\exp(\delta)+w$ in standard form with
  $\gamma\ntrunceq_{m-1}\delta$ and $\gamma,\delta\in\J^{*}$.  If
  $x\ntrunceq_{n}z$, it follows from $z\trunceq u$ that
  $x\ntrunceq_{n}u$, and we are done. If $x\not\ntrunceq_{n}z$, we
  must have $x=z+\sign r\exp(\gamma')$ in standard form with
  $\gamma'\ntrunceq_{n-1}\gamma$ and $\gamma'\in\J^{*}$.  By induction
  on $n$, this implies that $\gamma'\ntrunceq\delta$, which means that
  $x\ntrunceq u$, concluding the argument.
\end{proof}

\begin{rem}
  The relation $\ntrunceq$ is the smallest transitive one such that:
  \begin{enumerate}
  \item for all $x,y\in\no^{*}$, $x\trunceq y$ implies $x\ntrunceq y$;
  \item for all $\gamma,\delta\in\J^{*}$, $\gamma\ntrunceq\delta$
    implies $\exp(\gamma)\ntrunceq\exp(\delta)$ and
    $-\exp(\gamma)\ntrunceq-\exp(\delta)$;
  \item for all $\m\in\M^{\neq1}$ and $r\in\R^{*}$,
    $\sign r\m\ntrunceq r\m$;
  \item for all $x,y\in\no^{*}$ and $z\in\no$, if $x\ntrunceq y$, then
    $z+x\ntrunceq z+y$, provided both sums are in standard form.
  \end{enumerate}
\end{rem}

Moreover, the two relations share the following convexity property.

\begin{prop}
  \label{prop:trunc-convex}For all $x\in\no$, the class
  $\{y\in\no\suchthat x\trunceq y\}$ is convex.
\end{prop}
\begin{proof}
  Let $x\in\no$ be given and let $u$ be a number in the convex hull of
  $\{y\in\no\suchthat x\trunceq y\}$. This easily implies that there
  exists $y\in\no$ with $x\trunceq y$ such that $u$ is between $x$ and
  $y$, and in particular $|u-x|\leq|y-x|$. Since $x\trunceq y$, we
  have $|y-x|\vless\m$ for all $\m\in\supp(x)$, which implies that
  $|u-x|\vless\m$ for all $\m\in\supp(x)$. Therefore, $x\trunceq u$,
  as desired.
\end{proof}

\begin{prop}
  \label{prop:nested-trunc-convex}For all $x\in\no^{*}$, the class
  $\{y\in\no^{*}\suchthat x\ntrunceq y\}$ is convex.
\end{prop}
\begin{proof}
  Let $x\in\no^{*}$ be given and let $u$ be a number in the convex
  hull of $\{y\in\no^{*}\suchthat x\ntrunceq y\}$. This easily implies
  that there exist $y\in\no^{*}$ and $n\in\N$ with $x\ntrunceq_{n}y$
  such that $u$ is between $x$ and $y$. We reason by induction on $n$
  to prove that $x\ntrunceq_{n}u$.

  If $n=0$, the conclusion follows trivially from
  \prettyref{prop:trunc-convex}.

  If $n>0$, there are $\gamma\ntrunceq_{n-1}\delta$ both in $\J^{*}$,
  $z,w\in\no$, $r\in\R^{*}$ such that
  \[
  x=z+\sign r\exp(\gamma),\quad y=z+r\exp(\delta)+w,
  \]
  with both sums in standard form. Since $z\trunceq x,y$, we have that
  $z\trunceq u$ as well by \prettyref{prop:trunc-convex}. Moreover,
  since $u$ is between $x$ and $y$, we have $u\neq z$. Therefore,
  there are unique $\delta'\in\J$, $w'\in\no$, $r'\in\R^{*}$ such that
  \[
  u=z+r'\exp(\delta')+w'
  \]
  is in standard form. We clearly have $\sign{r'}=\sign r$ and
  $\delta'$ between $\delta$ and $\gamma$. By
  \prettyref{rem:nested-trunc-sign} $\delta$ and $\gamma$ have the
  same sign, hence $\delta'\neq0$.  By induction, we deduce that
  $\gamma\ntrunceq_{n-1}\delta'$, whence $x\ntrunceq_{n}u$, as
  desired.
\end{proof}

Next we show that $\J$, which is closed under $\trunceq$, is also
closed under $\ntrunceq$.

\begin{prop}
  \label{prop:nested-trunc-J}For all $x\in\no^{*}$ and
  $\delta\in\J^{*}$, if $x\ntrunceq\delta$, then
  $x\in\J^{*}$.
\end{prop}
\begin{proof}
  If $x\trunceq\delta$ the conclusion is trivial. If
  $x\ntrunceq_{n+1}\delta$, there are
  $\gamma\ntrunceq_{n}\delta'\in\J^{*}$, $z,w\in\no$ and $r\in\R^{*}$
  such that
  \[
  x=z+\sign r\exp(\gamma),\quad\delta=z+r\exp(\delta')+w.
  \]
  Note that $\delta'>0$ since $\delta\in\J$. By
  \prettyref{rem:nested-trunc-sign} we must have $\gamma>0$. Moreover,
  we clearly have $z\in\J$. It follows that $x\in\J^{*}$, as desired.
\end{proof}

Finally, just like $x\trunceq y$ implies $x\simpleq y$
(\prettyref{prop:trunc-simple}), also $x\ntrunceq y$ implies
$x\simpleq y$ (\prettyref{thm:nested-trunc-simplifies}); in
particular, the relation $\ntrunceq$ is well-founded. This implies
that $\ntrunceq$ has an associated ordinal rank which is crucial for
our inductive proofs. The rest of the section is devoted to the proof
of the fact that $\ntrunceq$ is well-founded.

\subsection{Products and inverses of monomials}

We first establish a few formulas for products and inverses of
monomials in the case we are working with representations of a special
type.

\begin{defn}
  Let $x=A\mid B$ with $x>0$. We say that $A\mid B$ is a
  \textbf{monomial representation} if for all $y\in\no$ and
  $k\in\N^{*}$, we have $A<y<B$ if and only if $A<ky<B$. Equivalently,
  $A\mid B$ is a monomial representation if
  \begin{enumerate}
  \item for every $a\in A$ there is $a'\in A$ such that $2a\leq a'$;
  \item for every $b\in B$ there is $b'\in B$ such that
    $b'\leq\frac{1}{2}b$.
  \end{enumerate}
\end{defn}

As the name suggests, monomial representations define monomials.

\begin{prop}
  If $x\in\no^{>0}$ admits a monomial representation, then
  $x\in\M$.
\end{prop}
\begin{proof}
  Suppose that $x=A\mid B$ is a monomial representation of
  $x$. Clearly, $A<y<B$ for every positive number $y$ such that
  $y\veq x$, and in particular, $A<\m<B$ for the unique monomial
  $\m\in\M$ such that $\m\veq x$. Since a monomial is the simplest
  positive element of its archimedean class, we have $\m\simpleq
  x$.
  But $x\simpleq\m$ holds as well, since $x$ is the simplest number
  such that $A<x<B$, and therefore $x=\m\in\M$.
\end{proof}

Conversely, all monomials admit monomial representations. In fact, we
shall prove that all simple representations of monomials are monomial.

\begin{lem}
  \label{lem:simple-km} Let $\m,\n\in\M$ with $\m\simple\n$.
  \begin{enumerate}
  \item If $\m<\n$, then $k\m\simple\n$ for all $k\in\N$.
  \item If $\m>\n$, then $\frac{1}{2^{k}}\m\simple\n$ for all
    $k\in\N$.
  \end{enumerate}
\end{lem}
\begin{proof}
  (1) Let $\m<\n$ be given. We wish to prove that $k\m\simple\n$.  We
  recall that $k=\{0,1,\dots,k-1\}\mid\emptyset$. By definition of
  product, we have $k\m=A\mid B$ with
  \begin{eqnarray*}
    A & = & \left\{ k'\m+k\m'-k'\m'\right\},\\
    B & = & \left\{ k'\m+k\m''-k'\m''\right\},
  \end{eqnarray*}
  where $k'$ ranges in $\{0,1,\ldots,k-1\}$.

  By \prettyref{rem:simple-test}, it suffices to check that $A<\n<B$.
  We can easily verify that
  \[
  k'\m+k\m'-k'\m'=k'\m+(k-k')\m'<k\m<\n,
  \]
  and therefore $A<\n$. Moreover,
  \[
  k'\m+k\m''-k'\m''=k'\m+(k-k')\m''\geq\m''.
  \]
  But $\m''>\m$ and $\m''\simple\m\simple\n$, hence by
  \prettyref{prop:simple-rep} $\m''>\n$. Therefore, $A<\n<B$, hence
  $k\m\simple\n$, as desired.

  (2) Let $\m>\n$ be given. We recall that
  $\frac{1}{2^{k}}=\left\{ 0\right\} \mid\left\{
    \frac{1}{2^{k'}}\right\} $
  for $k'=0,1,\dots,k-1$, hence $\frac{1}{2^{k}}\m=A\mid B$ with
  \begin{eqnarray*}
    A & = & \left\{ \frac{1}{2^{k}}\m',\ \frac{1}{2^{k'}}\m+\frac{1}{2^{k}}\m''-\frac{1}{2^{k'}}\m''\right\},\\
    B & = & \left\{ \frac{1}{2^{k}}\m'',\ \frac{1}{2^{k'}}\m+\frac{1}{2^{k}}\m'-\frac{1}{2^{k'}}\m'\right\}.
  \end{eqnarray*}

  Again, it suffices to verify that $A<\n<B$. We compare each
  expression in the above brackets with $\n$.

  \begin{itemize}
  \item We have $\m'\simple\m\simple\n$ and $\m'<\m$. By
    \prettyref{prop:simple-rep} it follows that $\m'<\n$ and \emph{a
      fortiori} $\frac{1}{2^{k}}\m'<\n$.
  \item If $k>0$, the expression
    $\frac{1}{2^{k'}}\m+\frac{1}{2^{k}}\m''-\frac{1}{2^{k'}}\m''$ is
    negative (so in particular $<\n$) because $\m\vless\m''$ and
    $k'<k$. If $k=0$, the expression can be dropped since $k'$ ranges
    in the empty set.
  \item Since $\m''\simple\m\simple\n$ and $\m''>\m$, we obtain
    $\n<\m''$ by \prettyref{prop:simple-rep}. Moreover, since
    $\m''\simple\n$ and $\n$ is the simplest positive number in its
    archimedean class, $\m''$ must belong to a different archimedean
    class, and therefore $\n\vless\m''$. It follows that
    $\n<\frac{1}{2^{k}}\m''$, as desired.
  \item If $k>0$, then
    $\frac{1}{2^{k'}}\m+\frac{1}{2^{k}}\m'-\frac{1}{2^{k'}}\m'\geq\frac{1}{2^{k}}\m>\n$.
    If $k=0$, the expression can be dropped since $k'$ ranges in the
    empty set.
  \end{itemize}

  We have thus proved that $A<\n<B$, whence
  $\frac{1}{2^{k}}\m\simple\n$, as desired.
\end{proof}

\begin{cor}
  \label{cor:simple-vell}Let $x\simple\m\simpleq z$, with $\m\in\M$.
  Then $x\vless\m$ if and only if $x\vless z$.
\end{cor}
\begin{proof}
  By \prettyref{prop:simple-rep}, $x<\m$ if and only if
  $x<z$. Excluding the trivial cases, we may assume that
  $x\neq0$. Moreover, recall that since $\m>0$ and $x\simple\m$ we
  have $x>0$, and since $\m$ is the simplest number in its archimedean
  class we have $x\not\veq\m$.

  Let $\n$ be the unique monomial such that $\n\veq x$; since $x>0$ we
  have $\n\simpleq x$ and therefore $\n\simple\m$. We apply
  \prettyref{lem:simple-km} to $\m$ and $\n$, distinguishing two
  cases.

  If $x\vless\m$ we have $\n\vless\m$, and in particular $\n<\m$.  Let
  $k$ be any integer such that $k\n>x$. By
  \prettyref{lem:simple-km}(1), we have $k\n\simple\m$. From
  $k\n\simple\m\simpleq z$ and $k\n<\m$, we obtain $k\n<z$ by
  \prettyref{prop:simple-rep}. Since $x<k\n<z$ and $k$ is arbitrarily
  large, we conclude that $x\vless z$, as desired.

  If $x\vgreater\m$ we proceed similarly, applying
  \prettyref{lem:simple-km}(2) to $\m$ and
  $\frac{1}{2^{k}}\n$.
\end{proof}

\begin{cor}
  \label{cor:omega-rep-simple}The representation
  $\omega^{x}=\{0,k\omega^{x'}\}\mid\left\{
    \frac{1}{2^{k}}\omega^{x''}\right\} $
  of \prettyref{def:omega} is simple.
\end{cor}
\begin{proof}
  By \prettyref{rem:omega-simple},
  $\omega^{x'},\omega^{x''}\simple\omega^{x}$, hence by
  \prettyref{lem:simple-km}
  $k\omega^{x'},\frac{1}{2^{k}}\omega^{x''}\simple\omega^{x}$ for
  every $k\in\N$. Therefore, by \prettyref{prop:simple-rep} the
  representation is simple.
\end{proof}

\begin{prop}
  \label{prop:mon-rep}For all $x\in\M$, any simple representation of
  $x$ is a monomial representation.
\end{prop}
\begin{proof}
  Suppose that $x\in\M$ and let $y\in\no$ be such that $x=\omega^{y}$.
  The representation of $\omega^{y}$ given by \prettyref{def:omega} is
  clearly monomial, and it is also simple by
  \prettyref{cor:omega-rep-simple}.  Since all simple representations
  $A\mid B$ define the same convex class
  $\convex AB=\left\{ z\in\no\suchthat A<z<B\right\} $, all simple
  representations are monomial, as desired.
\end{proof}

Thanks to the above observation, we can find simplified formulas for
the product of two monomials and for the inverse of a monomial.

\begin{prop}
  \label{prop:product-mon}Let $\m,\n\in\M$. If
  $\m=\{\m'\}\mid\{\m''\}$ and $\n=\{\n'\}\mid\{\n''\}$ are two
  monomial representations, then
  $\m\n=\{\m'\n+\m\n'\}\mid\{\m\n'',\,\m''\n\}$.
\end{prop}
\begin{proof}
  Since $\m,\n>0$, we may discard the expressions $\m',\n'$ that are
  strictly less than $0$ from the representations of $\m$ and $\n$.
  Since the representations are monomial, we then obtain
  $\m'\vless\m,$ $\m\vless\m''$, $\n'\vless\n$, $\n\vless\n''$. It
  follows immediately that
  $\{\m'\n+\m\n'\}<\m\n<\{\m\n'',\,\m''\n\}$. Let now $y\in\no$ be a
  number such that $\{\m'\n+\m\n'\}<y<\{\m\n'',\,\m''\n\}$. We need to
  prove that $\m\n\simpleq y$.

  By definition of product, we have $\m\n=A\mid B$ with
  \begin{eqnarray*}
    A & = & \left\{ \m'\n+\m\n'-\m'\n',\,\m''\n+\m\n''-\m''\n''\right\} ,\\
    B & = & \left\{ \m'\n+\m\n''-\m'\n'',\,\m''\n+\m\n'-\m''\n'\right\} .
  \end{eqnarray*}
  By \prettyref{prop:simple-rep}, it suffices to show that $A<y<B$.
  The inequality $A<y$ follows immediately from the assumption
  $\left\{ \m'\n+\m\n'\right\} <y$, as the first expression in $A$ is
  smaller than $\m'\n+\m\n'$ and the second one is negative. For
  $y<B$, observe that, since the given representations of $\n,\m$ are
  monomial, the assumption $y<\left\{ \m\n'',\,\m''\n\right\} $
  implies $y<\left\{ \frac{1}{k}\m\n'',\,\frac{1}{k}\m''\n\right\} $
  for any positive $k\in\N$. The inequality $y<B$ then follows easily
  from the fact that each element of $B$ is dominated by an expression
  of the form $\m\n''$ or $\m''\n$. 
\end{proof}

\begin{prop}
  \label{prop:inverse-mon}If $\m=\{\m'\}\mid\{\m''\}$ is a monomial
  representation, then
  \[
  \m^{-1}=\left\{ 0,\,\left(\m''\right)^{-1}\right\} \mid\left\{
    \left(\m'\right)^{-1}\right\}
  \]
  where $\left(\m'\right)^{-1}$ is only taken when $\m'>0$.
\end{prop}
\begin{proof}
  Let
  \[
  \n:=\left\{ 0,\,\left(\m''\right)^{-1}\right\} \mid\left\{
    \left(\m'\right)^{-1}\right\} ,
  \]
  where $\left(\m'\right)^{-1}$ is only taken when $\m'>0$. We need to
  prove that $\m\n=1$. We observe that the above representation of
  $\n$ is monomial, and therefore $\n\in\M$ by
  \prettyref{prop:mon-rep}.  By \prettyref{prop:product-mon},
  $\m\n=A\mid B$ where
  \begin{eqnarray*}
    A & = & \left\{ \m'\n,\,\m'\n+\m\left(\m''\right)^{-1}\right\} \\
    B & = & \left\{ \m\left(\m'\right)^{-1},\,\m''\n\right\} ,
  \end{eqnarray*}
  and $\left(\m'\right)^{-1}$ is only taken when $\m'>0$. Using
  $(\m'')^{-1}<\n<(\m')^{-1},$ it is easy to verify that
  $A<1<B$. Since $A$ contains at least one non-negative number of the
  form $\m'\n$, it follows that $1=A\mid B=\m\n$, as
  desired.
\end{proof}

\begin{cor}
  \label{cor:inverse-mon-simple}If $\m,\n\in\M$ and $\m\simpleq\n$,
  then $\m^{-1}\simpleq\n^{-1}$.
\end{cor}
\begin{proof}
  Take a simple monomial representation $\m=\{\m'\}\mid\{\m''\}$,
  which exists by \prettyref{prop:mon-rep}. Since $\m\simpleq\n$ we
  have $\{\m'\}<\n<\{\m''\}$. This immediately implies that
  \[
  \left\{ 0,\left(\m''\right)^{-1}\right\} <\n^{-1}<\left\{
    \left(\m'\right)^{-1}\right\}
  \]
  when $\m'>0$, and therefore $\m^{-1}\simpleq\n^{-1}$ by
  \prettyref{prop:inverse-mon}.
\end{proof}

\subsection{Simplicity}

Using the tools of the above subsection, we can prove several results
regarding the interaction between $\exp$ and $\simpleq$.

We start by generalizing the implication
$x\trunceq y\rightarrow x\simpleq y$ to sums in standard form.

\begin{prop}
  \label{prop:simple-tail}Let $x,y,z\in\no$. If both $z+x$ and $z+y$
  are in standard form and $x\simpleq y$, then
  $z+x\simpleq z+y$.
\end{prop}
\begin{proof}
  Let $z=\{z'\}\mid\{z''\}$ and $x=\{x'\}\mid\{x''\}$ be canonical,
  and in particular simple, representations. By definition of sum, we
  have $z+x=A\mid B$ where
  \[
  A=\{z'+x,\,z+x'\},\quad B=\{z''+x,\,z+x''\}.
  \]
  It suffices to verify that $A<z+y<B$. Since $x\simpleq y$ and
  $x=\{x'\}\mid\{x''\}$ is simple, we have $x'<y$, $y<x''$. After
  adding $z$ on all sides, we get $z+x'<z+y$, $z+y<z+x''$.

  It remains to show that $z'+x<z+y$, $z+y<z''+x$. Let $\tilde{z}$ be
  either $z'$ or $z''$. Since $\tilde{z}\simple z$, we have
  $z\not\trunceq\tilde{z}$ by \prettyref{prop:trunc-simple}, and
  therefore there is a monomial $\m\in\supp(z)$ such that
  $\m\vleq\tilde{z}-z$. Since $z+x$ and $z+y$ are in standard form, we
  know that $x,y\vless\n$ for all $\n\in\supp(z)$, and in particular
  $x,y\vless\m$. It follows that $x-y\vless\m\vleq\tilde{z}-z$, which
  implies that $|x-y|<|\tilde{z}-z|$. If $\tilde{z}=z'$, we get
  $x-y<z-z'$, or in other words, $z'+x<z+y$. If $\tilde{z}=z''$, we
  get $y-x<z''-z$, or in other words, $z+y<z''+x$.

  Therefore, $A<z+y<B$, which implies $z+x\simpleq z+y$, as desired.
\end{proof}

A similar statement holds when taking the exponential of sums
expressed in standard form, as follows.

\begin{prop}
  \label{prop:simple-exp-tail}Let $\gamma,\delta,\eta\in\J$. If
  $\eta+\gamma$ and $\eta+\delta$ are in standard form and
  $\exp(\gamma)\simpleq\exp(\delta)$, then
  $\exp(\eta+\gamma)\simpleq\exp(\eta+\delta)$.
\end{prop}
\begin{proof}
  Let $\m=\exp(\eta)$, $\n=\exp(\gamma)$ and $\om=\exp(\delta)$. Our
  hypothesis says that $\n\simpleq\om$ and
  $\supp(\vell(\m))\vgreater\vell(\n),\vell(\om)$, and we must prove
  that $\m\n\simpleq\m\om$.

  Consider the two canonical representations
  $\m=\left\{ \m'\right\} \mid\left\{ \m''\right\} $,
  $\n=\left\{ \n'\right\} \mid\left\{ \n''\right\} $. Recall that
  since $\m$, $\n$ are monomials we must have $\m'\vless\m$,
  $\m\vless\m''$, $\n'\vless\n$, $\n\vless\n''$. Moreover, since the
  representations are canonical, they are simple, and by
  \prettyref{prop:mon-rep} they are monomial.

  By \prettyref{prop:product-mon}, $\exp(\eta+\gamma)=\m\n=A\mid B$
  with
  \begin{eqnarray*}
    A & = & \left\{ \m'\n+\m\n'\right\} ,\\
    B & = & \left\{ \m\n'',\,\m''\n\right\} .
  \end{eqnarray*}

  It now suffices to prove that $A<\m\om=\exp(\eta+\delta)<B$, namely
  $\m'\n+\m\n'<\m\om$, $\m\om<\m\n''$ and $\m\om<\m''\n$. Simplifying
  further, it suffices to show that $\n'\vless\om$, $\om\vless\n''$,
  $\m'\n\vless\m\om$ and $\m\om\vless\m''\n$.

  Since $\n',\n''\simple\n\simpleq\om$, the inequalities
  $\n'\vless\om$, $\om\vless\n''$ follow immediately from
  \prettyref{cor:simple-vell} and $\n'\vless\n$, $\n\vless\n''$. For
  $\m'\n\vless\m\om$ and $\m\om\vless\m''\n$, we note that it is
  equivalent to saying that $\vell(\m'\n)<\vell(\m\om)$ and
  $\vell(\m\om)<\vell(\m''\n)$. Rearranging the terms, we wish to
  prove that $\vell(\m')-\vell(\m)<\vell(\om)-\vell(\n)$ and
  $\vell(\om)-\vell(\n)<\vell(\m'')-\vell(\m)$.

  Let $\tilde{\m}$ be either $\m'$ or $\m''$. If
  $\vell(\m)\trunceq\vell(\tilde{\m})$, then by
  \prettyref{prop:trunc-exp-simplifies}
  $\m=\exp(\vell(\m))\simpleq\exp(\vell(\tilde{\m}))=\tilde{\m}$,
  contradicting $\tilde{\m}\simple\m$. Therefore, we have
  $\vell(\m)\not\trunceq\vell(\tilde{\m})$, or in other words, there
  is a monomial $\mathfrak{p}\in\supp(\vell(\m))$ such that
  $\mathfrak{p}\vleq\vell(\tilde{\m})-\vell(\m)$. From the hypothesis
  $\supp(\vell(\m))\vgreater\vell(\n),\vell(\om)$ we obtain
  $\supp(\vell(\m))\vgreater\vell(\om)-\vell(\n)$, and in particular
  $\mathfrak{p}\vgreater\vell(\om)-\vell(\n)$. Therefore,
  $\vell(\tilde{\m})-\vell(\m)\vgreater\vell(\om)-\vell(\n)$, from
  which it follows that
  $|\vell(\tilde{\m})-\vell(\m)|>|\vell(\om)-\vell(\n)|$.  Recalling
  that $\tilde{\m}$ is one of $\m'$ or $\m''$, this easily implies
  $\vell(\m')-\vell(\m)<\vell(\om)-\vell(\n)$ and
  $\vell(\om)-\vell(\n)<\vell(\m'')-\vell(\m)$, as desired.
\end{proof}

Moreover, $\exp$ preserves simplicity under suitable assumptions.

\begin{prop}
  \label{prop:exp-simple-mon}Let $\m,\n\in\M^{>1}$ be such that
  $\m\simpleq\n$ and $\log(\m)\vless\n$. Then
  $\exp(\m)\simpleq\exp(\n)$ and
  $\exp(-\m)\simpleq\exp(-\n)$.
\end{prop}
\begin{proof}
  It suffices to show that $\exp(\m)\simpleq\exp(\n)$, as the second
  part then follows from \prettyref{cor:inverse-mon-simple}. By
  \prettyref{thm:characterization-exp}(1), we have $\exp(\m)=A\mid B$
  with
  \[
  A=\left\{ \m^{k},\exp(\m')^{k}\right\} ,\ B=\left\{
    \exp(\m'')^{\frac{1}{k}}\right\} ,
  \]
  where $\m',\m''$ range over the infinite monomials simpler than $\m$
  and such that respectively $\m'<\m$, $\m<\m''$, and $k$ runs over
  the positive integers. For the conclusion, it suffices to verify
  that $A<\exp(\n)<B$.

  Since $\m'\simple\m\simpleq\n$ and $\m'<\m$, we have $\m'<\n$.  It
  follows that $k\m'<\n$ for all $k$ (since $\m',\n$ are monomials),
  and therefore $\exp(\m')^{k}=\exp(k\m')<\exp(\n)$. Similarly, since
  $\m''\simple\m\simpleq\n$ and $\m''>\m$, we have $\m''>\n$. This
  implies that $\frac{1}{k}\m''>\n$ for all $k$, and therefore
  $\exp(\n)<\exp(\m'')^{\frac{1}{k}}$.

  Finally, since $\log(\m)\vless\n$ and $\log(\m),\n>0$ we have that
  $k\log(\m)<\n$ for all $k\in\N$. In particular,
  $\exp(k\log(\m))=\m^{k}<\exp(\n)$ for all $k\in\N$. Therefore,
  $A<\exp(\n)<B$, as desired.
\end{proof}

\begin{prop}
  \label{prop:exp-LM}If $\gamma\in\J^{*}$ and $\m\in\M^{>1}$ is the
  leading monomial of $\gamma$, then
  $\exp(\sign{\gamma}\m)\simpleq\exp(\gamma)$.
\end{prop}
\begin{proof}
  By \prettyref{cor:inverse-mon-simple}, it suffices to prove the case
  where $\gamma>0$.

  Call $\exp(\m)=A\mid B$ the representation given by
  \prettyref{thm:characterization-exp}(1).  Since clearly
  $\m\veq\gamma$ and $\gamma>0$, there is some positive $k\in\N$ such
  that $\frac{1}{k}\m\leq\gamma\leq k\m$, hence
  $\exp(\m)^{\frac{1}{k}}\leq\exp(\gamma)\leq\exp(\m)^{k}$.  It is now
  easy to verify that
  $A<\exp(\m)^{\frac{1}{k}}\leq\exp(\gamma)\leq\exp(\m)^{k}<B$, and
  therefore that $\exp(\m)\simpleq\exp(\gamma)$, as desired.
\end{proof}

Putting all of the above results together, we are finally able to
prove that $x\ntrunceq y$ implies $x\simpleq y$.

\begin{thm}
  \label{thm:nested-trunc-simplifies}For all $x,y\in\no^{*}$, if
  $x\ntrunceq y$ then $x\simpleq y$. Therefore, the relation
  $\ntrunceq$ is well founded.
\end{thm}
\begin{proof}
  By definition there is some $n\in\N$ such that $x\ntrunceq_{n}y$.
  We prove the conclusion by induction on $n$. At the same time, we
  also prove that if we further assume $x\vgreater1$, then
  $\log|x|\vless y$.

  First, assume $n=0$, so that $x\trunceq y$. It immediately follows
  that $x\simpleq y$ by \prettyref{prop:trunc-simple}. Moreover, we
  have $x\veq y$; it follows that if $x\vgreater1$, then
  $\log|x|\vless x\veq y$, as desired.

  Now assume that $n>0$. We first prove the case
  $x=\exp(\gamma),y=\exp(\delta)\in\M$.  By assumption we must have
  $\gamma\ntrunceq_{n-1}\delta$.

  If $n-1=0$, namely $\gamma\trunceq\delta$, then
  $\exp(\gamma)\simpleq\exp(\delta)$ follows from
  \prettyref{prop:trunc-exp-simplifies}; moreover, if
  $\exp(\gamma)\vgreater1$, then
  $\log(\exp(\gamma))=\gamma\veq\delta\vless\exp(\delta)$, as desired.

  If $n-1>0$, we can write $\gamma=z'+\sign{r'}\exp(\gamma')$,
  $\delta=z'+r'\exp(\delta')+w'$ in standard form with
  $\gamma'\ntrunceq_{n-2}\delta'$, and necessarily
  $\gamma',\delta'>0$. By inductive hypothesis, we get
  $\exp(\gamma')\simpleq\exp(\delta')$ and
  $\log(\exp(\gamma'))=\gamma'\vless\exp(\delta')$. Combining
  \prettyref{cor:inverse-mon-simple}, \prettyref{prop:exp-simple-mon},
  \prettyref{prop:exp-LM} and \prettyref{prop:trunc-exp-simplifies} we
  get
  \begin{multline*}
    \exp(\sign{r'}\exp(\gamma'))\simpleq\exp(\sign{r'}\exp(\delta'))\simpleq\\
    \simpleq\exp(r'\exp(\delta'))\simpleq\exp(r'\exp(\delta')+w').
  \end{multline*}
  By \prettyref{prop:simple-exp-tail}, we deduce that
  \[
  \exp(\gamma)=\exp(z'+\sign{r'}\exp(\gamma'))\simpleq\exp(z'+r'\exp(\delta')+w')=\exp(\delta),
  \]
  namely $x\simpleq y$. Finally, if $\exp(\gamma)\vgreater1$ we have
  $\gamma,\delta>0$. If $z'\neq0$, then
  $\log(\exp(\gamma))=\gamma\veq\delta\vless\exp(\delta)$.  If $z'=0$,
  we recall that $\gamma'\vless\exp(\delta')\veq r'\exp(\delta')+w'$;
  it follows that
  $\log(\exp(\gamma))=\gamma\veq\exp(\gamma')\vless\exp(r'\exp(\delta')+w')=\exp(\delta)$.
  In both cases we obtain $\log|x|\vless y$, as desired.

  For general $x$ and $y$, we must have $x=z+\sign r\exp(\gamma)$,
  $y=z+r\exp(\delta)+w$ in standard form, with
  $\gamma\ntrunceq_{n-1}\delta$.  Note in particular that
  $\exp(\gamma)\ntrunceq_{n}\exp(\delta)$, and by the previous
  argument, $\exp(\gamma)\simpleq\exp(\delta)$.  By
  \prettyref{prop:trunc-simple} we get
  $\sign r\exp(\gamma)\simpleq\sign r\exp(\delta)\simpleq
  r\exp(\delta)\simpleq r\exp(\delta)+w$.
  By \prettyref{prop:simple-tail} we get
  \[
  x=z+\sign r\exp(\gamma)\simpleq z+r\exp(\delta)+w=y.
  \]

  Finally, assume $x\vgreater1$. If $z\neq0$, then $x\veq y$, hence
  $\log|x|\veq\log|y|\vless y$. If $z=0$, then
  $\log|x|\sim\vell(|x|)=\gamma$ and $\gamma>0$; by inductive
  hypothesis, we have $\log(\gamma)\vless\delta=\vell(|y|)\sim\log|y|$
  and therefore $\log|x|\sim\gamma\vless\exp(\log|y|)\veq y$, as
  desired.
\end{proof}

\subsection{The nested truncation rank}

We now use the well-foundedness of $\ntrunceq$ to define an
appropriate notion of rank with ordinal values. We shall see in
\prettyref{sec:transseries} that the existence of this rank is
essentially equivalent to saying that $\no$ is a field of transseries
in the sense of Schmeling.

\begin{defn}
  \label{def:nested-rank}For all $x\in\no^{*}$, the \textbf{nested
    truncation rank} $\ntrank(x)$ of $x$ is the foundation rank of
  $\ntrunceq$, namely,
  $\ntrank(x):=\sup\left\{ \ntrank(y)+1\suchthat y\ntrunc x\right\}
  \in\on$.  We also set $\ntrank(0):=0$.
\end{defn}

Note that the real numbers have all rank zero, since they do not have
proper nested truncations. We shall describe in
\prettyref{cor:nested-rank-zero} all the other numbers of rank
zero. Note moreover that $\ntrank(x)=\ntrank(-x)$ for all $x\in\no$ by
\prettyref{rem:nested-trunc-sign}.

\begin{prop}
  \label{prop:ntrank-gamma-exp-gamma}For all $\gamma\in\J$,
  $\ntrank(\pm\exp(\gamma))=\ntrank(\gamma)$.
\end{prop}
\begin{proof}
  The conclusion is trivial if $\gamma=0$. If $\gamma\neq0$, by
  definition $x\ntrunc\exp(\gamma)$ if and only if $x=\exp(\delta)$
  for some $\delta\in\J^{*}$ with $\delta\ntrunc\gamma$, and
  $x\ntrunc-\exp(\gamma)$ if and only if $x=-\exp(\delta)$ for some
  $\delta\in\J^{*}$ with $\delta\ntrunc\gamma$. By
  \prettyref{prop:nested-trunc-J}, the only numbers $\delta\in\no^{*}$
  such that $\delta\ntrunc\gamma$ are in $\J^{*}$. It follows easily
  by induction on $\ntrank(\gamma)$ that
  $\ntrank(\pm\exp(\gamma))=\ntrank(\gamma)$.
\end{proof}

\begin{prop}
  \label{prop:ntrank-monomial}For all $\m\in\M^{\neq1}$ and
  $r\in\R^{*}$, if $r\neq\pm1$, then
  $\ntrank(r\m)=\ntrank(\m)+1>\ntrank(\m)$.
\end{prop}
\begin{proof}
  By definition, for all $x\in\no^{*}$, $x\ntrunc r\m$ if and only if
  $x=\sign r\m$ or $x\ntrunc\sign r\m$. It follows that
  $\ntrank(r\m)=\ntrank(\m)+1$.
\end{proof}

\begin{prop}
  \label{prop:ntrank-term}Let $x\in\no^{*}$. If $r\m\in\T$ is a term
  of $x$, then:
  \begin{enumerate}
  \item $\ntrank(r\m)\leq\ntrank(x)$;
  \item if $\m$ is not minimal in $\supp(x)$, then
    $\ntrank(r\m)<\ntrank(x)$.
  \end{enumerate}
\end{prop}
\begin{proof}
  We prove the conclusion by induction on $\alpha:=\ntrank(x)$.

  If $\m$ is not minimal in $\supp(x)$, then there exists
  $y\in\no^{*}$ such that $y\trunc x$ and $r\m$ is a term of $y$. By
  definition of the rank, we have $\ntrank(y)<\alpha$, hence
  $\ntrank(r\m)<\alpha$ by inductive hypothesis, proving (2).

  Suppose now that $\m$ is minimal, which means that we can write
  $x=z+r\m$ in standard form for some $z\in\no$. If $\m=1$, then
  clearly $\ntrank(r\m)=0\leq\ntrank(x)$, so we may assume that
  $\m\ne1$.

  If $r\neq\pm1$, since $\m\neq1$ we have $x':=z+\sign r\m\ntrunc x$.
  By definition, it follows that $\ntrank(x')<\alpha$, and by
  inductive hypothesis we must have
  $\ntrank(\m)=\ntrank(\sign r\m)\leq\ntrank(x')<\alpha$.  Therefore,
  by \prettyref{prop:ntrank-monomial},
  $\ntrank(r\m)=\ntrank(\m)+1\leq\alpha$.

  If $r=\pm1$, suppose by contradiction that
  $\ntrank(r\m)\geq\alpha+1$.  Then there exists $y\ntrunc r\m$ such
  that $\ntrank(y)\geq\alpha$.  By Remarks
  \ref{rem:nested-trunc-sign}, \ref{rem:ntrunc-mon}, we must have
  $y=r\n$ for some $\n\in\M$. Let $x':=z+r\n$. We claim that
  $x'\ntrunc x$, hence $\ntrank(x')<\alpha$. Granted the claim, by
  inductive hypothesis we have
  $\ntrank(y)=\ntrank(r\n)\leq\ntrank(x')<\alpha$, a contradiction.

  To prove the claim, by \prettyref{rem:ntrunc-std-form} we only have
  to prove that $z+r\n$ is in standard form. Suppose by contradiction
  that this is not the case. Then $z\neq0$ and there is a term
  $s\om\in\T$ of $z$ such that $\n\geq\om$, while $\om>\m$ since
  $z+r\m$ is in standard form. By \prettyref{prop:nested-trunc-convex}
  it follows that $\n\ntrunceq\om$, hence
  $\ntrank(\om)\geq\ntrank(\n)\geq\alpha$.  On the other hand,
  $z\ntrunc x$, hence $\ntrank(z)<\alpha$. By inductive hypothesis, we
  get $\ntrank(\om)\leq\ntrank(s\om)\leq\ntrank(z)<\alpha$, a
  contradiction.
\end{proof}

\section{Log-atomic numbers}

As anticipated in the introduction, a crucial subclass of $\no$ is the
one of log-atomic numbers. We defined them as follows.

\begin{defn}
  A positive infinite surreal number $x\in\no$ is \textbf{log-atomic
  }if for every $n\in\N$, $\log_{n}(x)$ is an infinite monomial. We
  call $\li$ the class of all log-atomic numbers.
\end{defn}

Note that $\li\subset\M^{>1}$ and $\exp(\li)=\log(\li)=\li$. It turns
out that the log-atomic numbers are the natural representatives of a
certain equivalence relation, similarly to how the monomials are the
natural representatives of the archimedean equivalence $\veq$.

\subsection{Levels}

We first define an appropriate notion of magnitude, which is weaker
than the dominance relation $\vleq$.

\begin{defn}
  \label{def:levels}Given two elements $x,y\in\no$ with $x,y>\N$, we
  write
  \begin{enumerate}
  \item $x\lleq y$ if $x\leq\exp_{h}\left(k\log_{h}(y)\right)$ for
    some $h,k\in\N^{*}$ (equivalently, $\log_{h}(x)\vleq\log_{h}(y)$
    for some $h\in\N$);
  \item $x\lless y$ if $x<\exp_{h}\left(\frac{1}{k}\log_{h}(y)\right)$
    for all $h,k\in\N^{*}$ (equivalently,
    $\log_{h}(x)\vless\log_{h}(y)$ for all $h\in\N$);
  \item $x\lequal y$ if
    $\exp_{h}\left(\frac{1}{k}\log_{h}(y)\right)\leq
    x\leq\exp_{h}\left(k\log_{h}(y)\right)$
    for some $h,k\in\N^{*}$ (equivalently,
    $\log_{h}(x)\veq\log_{h}(y)$ for some $h\in\N$).
  \end{enumerate}
  We call \textbf{level of $x$} the class
  $[x]:=\left\{ y\in\no\suchthat y>\N,y\lequal x\right\} $.
\end{defn}

\begin{prop}
  The relation $\lequal$ is an equivalence relation. Moreover,
  $x\lequal y$ if and only if there exists $h\in\N$ such that
  $\log_{h}(x)\sim\log_{h}(y)$.
\end{prop}
\begin{proof}
  Since $\log_{h}(x)\veq\log_{h}(y)$ implies
  $\log_{k}(x)\veq\log_{k}(y)$ for all $k\geq h$, we immediately get
  that $\lequal$ is an equivalence relation. Moreover, if
  $\log_{h}(x)\veq\log_{h}(y)$, then clearly
  $\log_{h+1}(x)\sim\log_{h+1}(y)$, as desired.
\end{proof}

\begin{rem}
  The equivalence relation $\lequal$ generalizes the notion of level
  in Hardy fields as in \cite{Rosenlicht1987,Marker1997}, whence the
  name. For instance, if $x\in\no$ satisfies $x>\N$, we have
  $x\lequal x^{n}$ for all $n\in\N^{*}$ and $x\lless\exp(x)$. While in
  those papers the only levels under consideration are given by
  $\log_{n}(x)$ and $\exp_{n}(x)$, the surreal numbers have
  uncountably many levels, and in fact a proper class of
  them.
\end{rem}

\begin{prop}
  Each level $[x]$ is a union of positive parts of archimedean classes
  and $\lleq$ induces a total order on levels.
\end{prop}

The proof is trivial and left to the reader.

We can verify that $\li$ is a class of representatives for the
equivalence relation $\lequal$. For instance, any two distinct
log-atomic numbers have necessarily different levels.

\begin{prop}
  \label{prop:log-atom-diff-level}Let $\mu,\lambda\in\li$. If
  $\mu<\lambda$, then $\mu\lless\lambda$.
\end{prop}
\begin{proof}
  Suppose by contradiction that $\mu\lgeq\lambda$ and $\mu<\lambda$.
  Then $\mu\lequal\lambda$, so there exists some $n\in\N$ such that
  $\log_{n}(\mu)\sim\log_{n}(\lambda)$. Since $\log_{n}(\mu)$ and
  $\log_{n}(\lambda)$ are both monomials, we obtain
  $\log_{n}(\mu)=\log_{n}(\lambda)$, hence $\lambda=\mu$, a
  contradiction.
\end{proof}

On the other hand, any positive infinite surreal number has the same
level of some log-atomic number.

\begin{lem}
  \label{lem:nested-trunc-level}Let $x,y>\N$. If $x\ntrunceq y$, then
  $x\lequal y$.
\end{lem}
\begin{proof}
  Let $n\in\N$ be such that $x\ntrunceq_{n}y$. We claim that
  $\log_{n}(x)\veq\log_{n}(y)$, hence by definition $x\lequal y$. We
  work by induction on $n$.

  If $x\ntrunceq_{0}y$, then obviously $x\veq y$.

  If $x\ntrunceq_{n+1}y$, write $x=z+\sign r\exp(\gamma)$,
  $y=z+r\exp(\delta)+w$ in standard form with
  $\gamma\ntrunceq_{n}\delta$. If $z\neq0$, then again
  $x\veq y\veq z$, hence \emph{a fortiori}
  $\log_{n+1}(x)\veq\log_{n+1}(y)$.  If $z=0$, then
  $\log(x)\sim\vell(x)=\gamma$ and $\log(y)\sim\vell(y)=\delta$ by
  \prettyref{rem:vell}. By inductive hypothesis,
  $\log_{n}(\gamma)\veq\log_{n}(\delta)$.  This immediately implies
  that
  $\log_{n+1}(x)\sim\log_{n}(\gamma)\veq\log_{n}(\delta)\sim\log_{n+1}(y)$,
  as desired.
\end{proof}

\begin{prop}
  \label{prop:nested-trunc-log-atom}If $x>\N$, there exists
  $\lambda\in\li$ such that $\lambda\ntrunceq x$, and therefore such
  that $\lambda\lequal x$.
\end{prop}
\begin{proof}
  Suppose by contradiction that there exists a counterexample $x>\N$
  such that $\lambda\not\ntrunceq x$ for all $\lambda\in\li$. By
  \prettyref{thm:nested-trunc-simplifies}, we may assume that $x$ is a
  minimal counterexample with respect to $\ntrunceq$ (for instance, we
  may take $x$ of minimal simplicity).  Note that obviously
  $x\notin\li$.

  Since $x$ is positive infinite, we may write
  $x=r_{0}\m_{0}+\delta_{0}$ and then inductively
  \[
  \m_{n}=:\exp(r_{n+1}\m_{n+1}+\delta_{n+1})
  \]
  with $r_{n}\in\R^{>0}$, $\m_{n}\in\M^{>1}$, $\delta_{n}\in\J$ and
  $r_{n}\m_{n}+\delta_{n}$ in standard form (in other words,
  $r_{n+1}\m_{n+1}$ is the leading term of $\log(\m_{n})$). We claim
  that $r_{n}=1$ and $\delta_{n}=0$ for all $n\in\N$; this implies
  that $\log_{n}(x)=\m_{n}\in\M$ and therefore that $x$ is
  $\log$-atomic. We reason by induction on $n$.

  For $n=0$, it suffices to note that $\m_{0}\ntrunceq x$. By
  transitivity, $\lambda\not\ntrunceq\m_{0}$ for all
  $\lambda\in\li$. By minimality of $x$, we must have $x=\m_{0}$,
  hence $r_{0}=1$ and $\delta_{0}=0$.

  If $n>0$, assume that for all $m<n$ we have $r_{m}=1$ and
  $\delta_{m}=0$.  In particular, we must have
  $x=\exp_{n}(r_{n}\m_{n}+\delta_{n})$.  It follows immediately that
  $\exp_{n}(\m_{n})\ntrunceq x$, hence
  $\lambda\not\ntrunceq\exp_{n}(\m_{n})$ for all $\lambda\in\li$.  By
  minimality of $x$, we must have $r_{n}=1$ and $\delta_{n}=0$, as
  desired.
\end{proof}

\begin{cor}
  \label{cor:log-atom-simplest-level}$\li$ is a class of
  representatives for $\lequal$. Moreover, for each $\lambda\in\li$,
  $\lambda$ is the simplest number in its level.
\end{cor}
\begin{proof}
  Let $[x]$ be a given level. By
  \prettyref{prop:nested-trunc-log-atom}, there exists $\lambda\in\li$
  such that $\lambda\ntrunceq x$ and $\lambda\lequal x$, and by
  \prettyref{prop:log-atom-diff-level}, $\lambda$ is also unique. By
  \prettyref{thm:nested-trunc-simplifies}, we also have
  $\lambda\simpleq x$. This shows that $\lambda$ is the simplest
  number in $[x]$, as desired.
\end{proof}

\begin{cor}
  \label{cor:nested-rank-zero}For all $x\in\no$, $\ntrank(x)=0$ if and
  only if either $x\in\R$ or $x=\pm\lambda^{\pm1}$ for some
  $\lambda\in\li$.
\end{cor}
\begin{proof}
  It is easy to see that $\ntrank(r)=\ntrank(\pm\lambda^{\pm1})=0$ for
  all $r\in\R$ and $\lambda\in\li$.

  Conversely, suppose that $x$ satisfies $\ntrank(x)=0$. If $x\neq0$,
  let $r\exp(\gamma)$ be the leading term of $x$. Since
  $\ntrank(x)=0$, by \prettyref{prop:ntrank-term} we must have
  $x=r\exp(\gamma)$, and by \prettyref{prop:ntrank-gamma-exp-gamma} we
  have $\ntrank(\gamma)=0$.

  If $\gamma=0$, then $x=r\in\R$, and we are done. If $\gamma\neq0$,
  then $r=\pm1$ by \prettyref{prop:ntrank-monomial}. Since
  $|\gamma|>\N$, by \prettyref{prop:nested-trunc-log-atom} there is
  $\mu\in\li$ such that $\mu\ntrunceq|\gamma|$. Since
  $\ntrank(\gamma)=\ntrank(-\gamma)=0$, we must have
  $\gamma=\pm\mu$. Letting $\lambda:=\exp(\mu)$, it follows that
  $x=\pm\lambda^{\pm1}$, as desired.
\end{proof}

\begin{cor}
  \label{cor:vell-n-fold}For any $x\in\no$ such that $\vell(x)\neq0$,
  there is $n\in\N$ such that $\vell_{n}(x)\in\li$, where
  $\vell_{n}=\vell\circ\dots\circ\vell$ is the $n$-fold composition of
  $\vell$ with itself.
\end{cor}
\begin{proof}
  Note first that $\vell_{2}(x)>\N$ for any $x\in\no$ such that
  $\vell(x)\neq0$.  By \prettyref{cor:log-atom-simplest-level}, there
  is $\lambda\in\li$ such that $\vell_{2}(x)\lequal\lambda$, whence
  $\log_{n}(\vell_{2}(x))\sim\log_{n}(\lambda)=\vell_{n}(\lambda)$ for
  some $n\in\N$. Since $\vell(y)\sim\log(y)$ for all $y>\N$, we get
  $\vell_{n+2}(x)\sim\vell_{n}(\lambda)$. But then
  $\vell_{n+3}(x)=\vell_{n+1}(\lambda)\in\li$, as desired.
\end{proof}

\subsection{Parametrizing the levels}

Mimicking the definition of the omega-map, there is a natural way of
defining a function $\lambda:\no\to\no^{>\N}$ whose values are the
simplest representatives for the $\lequal$-equivalence classes.

\begin{defn}
  \label{def:lambda}Let $x\in\no$ and let $x=\{x'\}\mid\{x''\}$ be its
  canonical representation. We define
  \[
  \lambda_{x}:=\left\{ k,\exp_{h}(k\log_{h}(\lambda_{x'}))\right\}
  |\left\{
    \exp_{h}\left(\frac{1}{k}\log_{h}(\lambda_{x''})\right)\right\}
  \]
  where $h,k$ range in $\N^{*}$.
\end{defn}

Note that the terms of the left-hand side are increasing in $h$ and
$k$, while those on the right-hand side are decreasing in $h$ and $k$.

\begin{prop}
  \label{prop:lambda-increasing}The function $x\mapsto\lambda_{x}$ is
  well defined, increasing, and if $x<y$, then
  $\lambda_{x}\lless\lambda_{y}$.
\end{prop}
\begin{proof}
  By abuse of notation we say ``$\lambda_{x}$ is well defined'' if
  there exists a (necessarily unique) function $z\mapsto\lambda_{z}$
  defined for all $z\simpleq x$ which satisfies the equation in
  \prettyref{def:lambda} on its domain of definition. Obviously, if
  $\lambda_{x}$ is well defined, then $\lambda_{x}>\N$ and
  $\lambda_{z}$ is well defined for all $z\simpleq x$. If
  $x=\{x'\}\mid\{x''\}$ is a canonical representation and
  $\lambda_{x}$ is well defined, then
  \[
  \left\{ k,\exp_{h}(k\log_{h}(\lambda_{x'}))\right\}
  <\lambda_{x}<\left\{
    \exp_{h}\left(\frac{1}{k}\log_{h}(\lambda_{x''})\right)\right\}
  \]
  for all $h,k\in\N^{*}$. By definition of $\lless$, it follows that
  $\lambda_{x'}\lless\lambda_{x}$ and
  $\lambda_{x}\lless\lambda_{x''}$.  Therefore, if $\lambda_{x}$ and
  $\lambda_{y}$ are both well defined and $y\simple x$, then
  \[
  x<y\iff\lambda_{x}\lless\lambda_{y},\quad
  y<x\iff\lambda_{y}\lless\lambda_{x}.
  \]
  The above equivalences then hold even without the assumption
  $y\simple x$, since we can always find some $z$ between $x$ and $y$
  with $z\simpleq x,y$.

  To prove that $\lambda_{x}$ is well defined for every $x$ we proceed
  by induction on simplicity. Consider the canonical representation
  $x=\{x'\}\mid\{x''\}$. By inductive hypothesis we can assume that
  all $\lambda_{x'}$, $\lambda_{x''}$ are well defined and therefore,
  by the above arguments, $\lambda_{x'},\lambda_{x''}>\N$ and
  $\lambda_{x'}\lless\lambda_{x''}$.  By definition of $\lless$ it
  follows that
  $k\log_{h}(\lambda_{x'})<\frac{1}{k}\log_{h}(\lambda_{x''})$ for all
  $h,k\in\N^{*}$, hence
  \[
  \exp_{h}(k\log_{h}(\lambda_{x'}))<\exp_{h}\left(\frac{1}{k}\log_{h}(\lambda_{x''})\right)
  \]
  for all $h,k\in\N^{*}$. By the comment after \prettyref{def:lambda},
  this implies that the convex class associated to the definition of
  $\lambda_{x}$ is non-empty, and therefore $\lambda_{x}$ is also well
  defined.
\end{proof}

\begin{rem}
  \label{rem:lambda-simple}It is immediate to see that $x\simpleq y$
  if and only if $\lambda_{x}\simpleq\lambda_{y}$.\
\end{rem}

\begin{cor}
  \label{cor:lambda-uniform}The definition of $\lambda_{x}$ is
  uniform.
\end{cor}
\begin{proof}
  The uniformity follows easily from the fact that if $x<y$, then
  $\lambda_{x}\lless\lambda_{y}$.
\end{proof}

In the same way one proves that every surreal number $x$ is in the
same archimedean class of some $\omega^{y}$, we can prove that every
$x>\N$ is in the same level of some $\lambda_{y}$.

\begin{prop}
  \label{prop:lambda-level}For every $x\in\no$ with $x>\N$ there is a
  (unique) $y\in\no$ such that $x\lequal\lambda_{y}$ and
  $\lambda_{y}\simpleq x$.  In particular, $\lambda_{y}$ is the
  simplest number in its level.
\end{prop}
\begin{proof}
  Since $x$ is positive infinite, its canonical representation is of
  the form $x=\N\cup A\mid B$, where $A$ is greater than $\N$. By
  induction on simplicity, we can assume that every element
  $c\in A\cup B$ is in the same level of some $\lambda_{z}\simpleq
  c$.
  Define $F=\{z\suchthat(\exists a\in A)(a\lequal\lambda_{z})\}$ and
  $G=\{z\suchthat(\exists b\in B)(b\lequal\lambda_{z})\}.$ Note that
  $F$ and $G$ are sets. We distinguish a few cases.

  If $F\not<G,$ there are $z\in F,w\in G$ with $z\geq w$, whence
  $\lambda_{z}\geq\lambda_{w}$. Let $a\in A$ and $b\in B$ be such that
  $a\lequal\lambda_{z}\simpleq a$ and $b\lequal\lambda_{w}\simpleq b$.
  Since $a<x<b$ and the levels are convex, we immediately get that
  $\lambda_{z}=\lambda_{w}$ (so $z=w$) and
  $x\lequal a\lequal\lambda_{z}$.  In particular,
  $x\lequal\lambda_{z}\simpleq a\simple x$, and we are done.

  Suppose now that $F<G$. If $x\lequal\lambda_{y}$ for some
  $y\in F\cup G$, we are done, so assume otherwise. We must have
  $[\lambda_{y'}]<x<[\lambda_{y''}]$ for every $y'\in F,\,y''\in
  G$.
  By inductive hypothesis, we have
  $A\subset\bigcup_{y'\in F}[\lambda_{y'}]$ and
  $B\subset\bigcup_{y''\in G}[\lambda_{y''}]$.  Therefore, if we let
  $y:=F\mid G$, we have
  $x=\N\cup A\mid B=\N\cup\bigcup_{y'\in
    F}[\lambda_{y'}]\mid\bigcup_{y''\in
    G}[\lambda_{y''}]=\lambda_{y}$, and we are done.
\end{proof}

\begin{cor}
  \label{cor:lambda-is-L}We have
  $\li=\lambda_{\no}=\left\{ \lambda_{x}\suchthat x\in\no\right\}
  $.
\end{cor}
\begin{proof}
  By \prettyref{cor:log-atom-simplest-level} and
  \prettyref{prop:lambda-level}, both classes $\li$ and
  $\lambda_{\no}$ are exactly the class of the simplest numbers in
  each level, and therefore they are equal.
\end{proof}

\begin{rem}
  \label{rem:lambda-n-exp-n-omega}It is easy to verify that
  $\lambda_{0}=\omega$ and $\lambda_{1}=\exp(\omega)$.

  Indeed, by definition we have
  $\lambda_{0}=\left\{ k\right\} \mid\emptyset$ for $k$ ranging in
  $\N$, hence $\lambda_{0}=\omega$. It follows that
  $\lambda_{1}=\left\{ k,\exp_{h}(k\log_{h}(\omega))\right\}
  \mid\emptyset$
  for $h,k$ ranging in $\N$. We observe that $\exp(\omega)$ is
  log-atomic and $\exp(\omega)\lgreater\omega$, hence
  $\exp(\omega)>\exp_{h}(k\log_{h}(\omega))$ for all $h,k\in\N$. It
  follows that $\lambda_{1}\simpleq\exp(\omega)$.  On the other hand,
  by \prettyref{thm:characterization-exp} we have
  $\exp(\omega)=\left\{ \omega^{k}\right\} \mid\emptyset$ for $k$
  ranging in $\N$. Since $\lambda_{1}>\omega^{k}=\exp(k\log(\omega))$
  for all $k\in\N$, we get $\exp(\omega)\simpleq\lambda_{1}$, hence
  $\lambda_{1}=\exp(\omega)$, as desired.

  With a similar argument, one can verify that
  $\lambda_{n}=\exp_{n}(\omega)$ and $\lambda_{-n}=\log_{n}(\omega)$
  for all $n\in\N$. It follows, for instance, that there exist several
  log-atomic numbers between $\omega$ and $\exp(\omega)$, such as
  $\lambda_{\frac{1}{2}}$, and in particular several levels between
  $[\omega]$ and $[\exp(\omega)]$.
\end{rem}

\subsection{$\kappa$-numbers}

We recall here the notion of $\kappa$-numbers defined in \cite[Def.\
3.1]{Kuhlmann2014}.  They can be defined using again an appropriate
notion of magnitude.

\begin{defn}
  \label{def:kappa-equiv}Given two elements $x,y\in\no$ with $x,y>\N$,
  we write
  \begin{enumerate}
  \item $x\kleq y$ if $x\leq\exp_{h}(y)$ for some $h\in\N$;
  \item $x\kless y$ if $x<\log_{h}(y)$ for all $h\in\N$;
  \item $x\kequal y$ if $\log_{h}(y)\leq x\leq\exp_{h}(y)$ for some
    $h\in\N$.
  \end{enumerate}
\end{defn}

It is easy to verify that the relation $\kequal$ is an equivalence
relation, and that $\kleq$ induces a total order on its equivalence
classes. The following proposition is also easy, and its proof is left
to the reader.

\begin{prop}
  \label{prop:level-implies-kappa}For all $x,y\in\no$ with $x,y>\N$,
  $x\lequal y$ implies $x\kequal y$.
\end{prop}

The $\kappa$-numbers are a natural class of representatives for the
$\kequal$-classes.

\begin{defn}[{\cite[Def.\ 3.1]{Kuhlmann2014}}]
  \label{def:kappa}Let $x\in\no$ and let $x=\{x'\}\mid\{x''\}$ be its
  canonical representation. We define
  \[
  \kappa_{x}:=\left\{ \exp_{n}(0),\exp_{n}(\kappa_{x'})\right\}
  \mid\left\{ \log_{n}(\kappa_{x''})\right\} ,
  \]
  where $n$ runs in $\N$. We call $\K$ the class of the numbers of the
  form $\kappa_{x}$.
\end{defn}

\begin{rem}
  It can be easily verified that $\kappa_{0}=\omega$ and
  $\kappa_{1}=\varepsilon_{0}$, where $\varepsilon_{0}$ is the least
  ordinal such that $\omega^{\varepsilon_{0}}=\varepsilon_{0}$; see
  \cite[Ex.\ 3.3]{Kuhlmann2014}.
\end{rem}

One can verify that this definition is uniform, and again that
$x\simpleq y$ if and only if $\kappa_{x}\simpleq\kappa_{y}$. One can
also see that for every $x>\N$ there exists a $\kappa_{y}\simpleq x$
such that $\kappa_{y}\kequal x$. In particular, $\kappa_{y}$ is the
simplest number in its $\kequal$-equivalence class. Moreover, $x<y$
implies $\kappa_{x}\kless\kappa_{y}$. We refer to \cite{Kuhlmann2014}
for more details.

It is proved in \cite{Kuhlmann2014} that $\log_{n}(\kappa_{x})$ is
always of the form $\omega^{\omega^{y}}$, and in particular belongs to
$\M$, showing that the $\kappa$-numbers are log-atomic. We rephrase
their statement as follows, and we give an extremely short proof
exploiting the relationship between $\lequal$ and $\kequal$.

\begin{thm}[{\cite[Thm.\ 4.3]{Kuhlmann2014}}]
  $\K\subseteq\li$.
\end{thm}
\begin{proof}
  Since $\kappa_{x}$ is the simplest number in its
  $\kequal$-equivalence class, by \prettyref{prop:level-implies-kappa}
  it must also be the simplest number in its level. Therefore, by
  \prettyref{cor:log-atom-simplest-level}, $\kappa_{x}=\lambda$ for
  some $\lambda\in\li$, as desired.
\end{proof}

It was conjectured in \cite[Conj.\ 5.2]{Kuhlmann2014} that $\K$
generates all the log-atomic numbers by iterated applications of
$\exp$ and $\log$. However, we can exhibit numbers in $\li$ that are
not of this form.

\begin{prop}
  \label{prop:kappa-not-gen-L}There are numbers in $\li$ that cannot
  be obtained from numbers in $\K$ by finitely many applications of
  $\exp$ and $\log$.
\end{prop}
\begin{proof}
  As seen in \prettyref{rem:lambda-n-exp-n-omega}, there are
  log-atomic numbers between $\omega$ and $\exp(\omega)$, such as
  $\lambda_{\frac{1}{2}}$.  On the other hand, it is easy to verify
  that no number of the form $\log_{n}(\kappa)$ or $\exp_{n}(\kappa)$,
  with $n\in\N$ and $\kappa\in\K$, lies between $\omega$ and
  $\exp(\omega)$. Indeed, this is trivial if $\kappa=\omega$, while if
  $\kappa'<\omega<\kappa''$ we have
  $\kappa'\kless\omega\kless\kappa''$, and in particular
  $\exp_{n}(\kappa')<\omega<\exp(\omega)<\log_{n}(\kappa'')$ for all
  $n\in\N$. Therefore, $\lambda_{\frac{1}{2}}\ne\exp_{n}(\kappa)$ and
  $\lambda_{\frac{1}{2}}\ne\log_{n}(\kappa)$ for all $n\in\N$ and
  $\kappa\in\K$, as desired.
\end{proof}

For our construction, the $\kappa$-numbers that matter are actually
the ones of the form $\ka{\alpha}$ for $\alpha\in\on$.

\begin{rem}
  \label{rem:kappa-alpha-rep}If $\alpha\in\on$, then
  \[
  \ka{\alpha}=\N\mid\left\{
    \log_{n}(\ka{\beta})\,:\,n\in\N,\,\beta<\alpha\right\} ,
  \]
  namely $\kappa_{-\alpha}$ is the simplest positive infinite number
  less than $\log_{n}(\kappa_{-\beta})$ for all $n\in\N$ and
  $\beta<\alpha$.  Moreover, if $\beta<\alpha$, then
  $\kappa_{-\beta}\simple\kappa_{-\alpha}$ and of course
  $\kappa_{-\alpha}<\kappa_{-\beta}$.
\end{rem}

\begin{prop}
  \label{prop:kappa-coinitial}The sequence
  $(\kappa_{-\alpha}\suchthat\alpha\in\on)$ is decreasing and
  coinitial in the positive infinite numbers, namely every positive
  infinite number is greater than some $\kappa_{-\alpha}$.  In
  particular, $\li$ is coinitial in the positive infinite
  numbers.
\end{prop}
\begin{proof}
  Let $x>\N$. We know that $x\kequal\kappa_{y}$ for some $y$. It now
  suffices to recall that there exists an $\alpha\in\on$ such that
  $-\alpha<y$, and therefore $\ka{\alpha}\kless\kappa_{y}\kequal x$.
  In particular, $\ka{\alpha}<x$, as desired.
\end{proof}

\section{\label{sec:derivations}Surreal derivations}

\subsection{Derivations}

We begin with our definition of surreal derivation. It is the
specialization to surreal numbers of other notions which have been
defined by several authors in the context of $H$-fields or
transseries.

\begin{defn}
  \label{def:surreal-D}A \textbf{surreal derivation} is a function
  $\somed:\no\to\no$ satisfying the following properties:
  \begin{enumerate}
  \item Leibniz rule: $\somed(xy)=x\somed(y)+y\somed(x)$;
  \item strong additivity:
    $\somed\left(\sum_{i\in I}x_{i}\right)=\sum_{i\in I}\somed(x_{i})$
    if $(x_{i}\suchthat i\in I)$ is summable;
  \item compatibility with exponentiation:
    $\somed(\exp(x))=\exp(x)\somed(x)$;
  \item constant field $\R$: $\ker(\somed)=\R$;
  \item $H$-field: if $x>\N$, then $\somed(x)>0$.
  \end{enumerate}
\end{defn}

Conditions (4) and (5), together with the fact that $\somed$ is a
derivation, make the pair $(\no,\somed)$ into an $H$-field, the
abstract counterpart of the notion of Hardy field. In the definition
of $H$-field \cite{Aschenbrenner2002} one also requires that if
$|x|\leq c$ for some $c\in\ker(D)$, then there is some $d\in\ker(D)$
such that $|x-d|<c$ for every positive $c\in\ker(D)$; however, this is
always true when $\ker(D)=\R$.

\begin{rem}
  \label{rem:D-exp-finite}When $x$ is infinitesimal, point (3) follows
  from (1) and (2). Indeed, if $x$ is infinitesimal we have
  \[
  \somed(\exp(x))=\somed\left(1+x+\frac{x^{2}}{2!}+\dots\right).
  \]
  By strong additivity and the Leibniz rule we get
  \[
  \somed(\exp(x))=\somed(x)+x\somed(x)+\frac{x^{2}}{2!}\somed(x)+\ldots=\exp(x)\somed(x).
  \]
\end{rem}

\begin{rem}
  \label{rem:D-incr-on-inf}By points (2) and (5), if $x,y>\N$ and
  $x\vgreater y$, then $\somed(x)>\somed(y)>0$.
\end{rem}

Before embarking on the construction of a surreal derivation, we
recall a few properties that can be easily derived from the above
axioms.  We remark that these properties hold in any Hardy field
closed under the functions $\exp$ and $\log$.

\begin{prop}
  \label{prop:D-vell-incr}Let $\somed$ be a surreal derivation and let
  $x,y\in\no$. We have:
  \begin{enumerate}
  \item if $1\not\veq x\vgreater y$, then
    $\somed(x)\vgreater\somed(y)$;
  \item if $1\not\veq x\sim y$, then $\somed(x)\sim\somed(y)$;
  \item if $1\not\veq x\veq y$, then $\somed(x)\veq\somed(y)$.
  \end{enumerate}
\end{prop}
\begin{proof}
  Without loss of generality, we may assume that $x,y>0$.

  (1) Assume $1\not\veq x\vgreater y$. For all $r\in\R$ we have
  $x-ry\veq x$ and $x-ry>0$.

  If $x\veq x-ry\vgreater1$, then
  $\somed(x-ry)=\somed(x)-r\somed(y)>0$.  Therefore,
  $|\somed(x)|>|r|\cdot|\somed(y)|$ for all $r\in\R$, hence
  $\somed(x)\vgreater\somed(y)$.

  If $x\veq x-ry\vless1$, then $\frac{1}{x-ry}\vgreater1$, hence
  $\somed\left(\frac{1}{x-ry}\right)=-\frac{\somed(x)-r\somed(y)}{(x-ry)^{2}}>0$,
  i.e., $\somed(x)<r\somed(y)$. Therefore,
  $|\somed(x)|>|r|\cdot|\somed(y)|$ for all $r\in\R$, hence
  $\somed(x)\vgreater\somed(y)$.

  (2) Assume $1\not\veq x\sim y$. We have $x-y\vless x$. By (1), it
  follows that $\somed(x)\vgreater\somed(x-y)=\somed(x)-\somed(y)$,
  which means $\somed(x)\sim\somed(y)$, as desired.

  (3) Assume $1\not\veq x\veq y$. The conclusion is trivial if
  $x=y=0$, so assume $x\neq0$. We have $x\sim ry$ for some
  $r\in\R^{*}$. By (2) we have
  $\somed(x)\sim\somed(ry)=r\somed(y)\veq\somed(y)$, as desired.
\end{proof}

The following proposition will play a crucial role in the sequel.

\begin{prop}
  \label{prop:vell-log-log}Let $\somed$ be a surreal derivation. Given
  $x,y\in\no$, if $x,y,x-y$ are positive infinite, then
  \[
  \log\left(\somed(x)\right)-\log\left(\somed(y)\right)\vless
  x-y\vleq\max\left\{ x,y\right\} .
  \]
\end{prop}
\begin{proof}
  Since $x,y,(x-y)>\N$, we have $\somed(x),\somed(y),\somed(x-y)>0$,
  or in other words $\somed(x)>\somed(y)>0$. Moreover, for every
  $r\in\R^{>0}$ we have $\exp(r(x-y))>\N$, namely
  $\ensuremath{\exp(rx)\vgreater\exp(ry)}$.  Taking the inverses we
  get $\ensuremath{\exp(-rx)\vless\exp(-ry)}$.

  Since $y\vgreater1$, we have $\exp(-ry)\vless1$, and another
  application of \prettyref{prop:D-vell-incr} yields
  $\somed(\exp(-rx))\vless\somed(\exp(-ry))$.  In particular,
  $|\somed(\exp(-rx))|<|\somed(\exp(-ry))|$; since $\somed$ is
  compatible with $\exp$, we get
  \[
  \exp(-rx)\cdot r\cdot\somed(x)<\exp(-ry)\cdot r\cdot\somed(y).
  \]

  Taking the logarithms on both sides and rearranging the summands, we
  get
  \[
  \log\left(\somed(x)\right)-\log\left(\somed(y)\right)<r(x-y).
  \]

  Since this holds for an arbitrary $r\in\R^{>0}$, and
  $\log\left(\somed(x)\right)>\log\left(\somed(y)\right)$, we obtain
  \[
  \log\left(\somed(x)\right)-\log\left(\somed(y)\right)\vless
  x-y\vleq\max\left\{ x,y\right\} ,
  \]
  as desired.
\end{proof}

\subsection{Derivatives of log-atomic numbers}

As anticipated in the introduction, we shall construct a surreal
derivation by first giving its values on the class $\li$. Clearly, if
we want to take a function $\somed:\li\to\no$ and extend it to a
surreal derivation, it must at least satisfy the inequality of
\prettyref{prop:vell-log-log}, and also $\somed(\lambda)>0$ for all
$\lambda\in\li$.

With some heuristics, it is not difficult to find a map
$\simpled_{\li}':\li\to\no^{>0}$ satisfying the inequalities of
\prettyref{prop:vell-log-log} and compatible with $\exp$. If
$\lambda\in\li$, we note that we must have
$\simpled_{\li}'(\lambda) = \lambda\cdot\simpled_{\li}'(\log\lambda)$.
Iterating, we obtain
$\simpled_{\li}'(\lambda) =
\lambda\cdot\log(\lambda)\cdot\log_{2}(\lambda)\cdot\ldots\cdot\log_{i-1}(\lambda)\cdot\simpled_{\li}'(\log_{i}(\lambda))$,
or equivalently
$\simpled_{\li}'(\lambda) =
\exp(\log(\lambda)+\log_{2}(\lambda)+\log_{3}(\lambda)+\ldots+\log_{i}(\lambda))\cdot\simpled_{\li}'(\log_{i}(\lambda))$.
This suggests the following definition.

\begin{defn}
  \label{def:natural-nonsimple-pre-D}If $\lambda\in\li$, we let
  \[
  \simpled_{\li}'(\lambda):=\exp\left(\sum_{i=1}^{\infty}\log_{i}(\lambda)\right).
  \]
\end{defn}

We remark that $(\log_{i}(\lambda))_{i=1}^{\infty}$ is a strictly
decreasing sequence of monomials, hence it is summable. It is an easy
exercise to check that $\simpled_{\li}'$ does satisfy the inequalities
of \prettyref{prop:vell-log-log}. It can be further shown that
$\simpled_{\li}'$ extends to a surreal derivation
$\simpled':\no\to\no$ (using \prettyref{thm:extending-pre-D}).
However, this derivation is not the ``simplest'' possible one with
respect to the simplicity relation, and moreover, its behavior is not
really nice; for instance, there is no $x\in\no$ such that
$\simpled'(x)=1$.

The simplest function $\simpled_{\li}:\li\to\no^{>0}$ satisfying the
inequalities of \prettyref{prop:vell-log-log} is given by a similar
but different formula involving the subclass of the $\kappa$-numbers
(those with indexes of the form $-\alpha$ where $\alpha$ is an
ordinal, see \prettyref{def:kappa} and
\prettyref{rem:kappa-alpha-rep}).  We postpone to
\prettyref{sec:pre-derivations} the proof that $\simpled_{\li}$ is
indeed the simplest one.

\begin{defn}
  \label{def:simple-pre-D}If $\lambda\in\li$, we let
  \[
  \simpled_{\li}(\lambda):=\exp\left(-\sum_{\kappa_{-\alpha}\kgeq\lambda}\sum_{i=1}^{\infty}\log_{i}(\ka{\alpha})+\sum_{i=1}^{\infty}\log_{i}(\lambda)\right)
  \]
  where $\alpha$ ranges in $\on$.
\end{defn}

\begin{rem}
  The sequence $(\kappa_{-\alpha})_{\alpha\in\on}$ is decreasing (see
  \prettyref{rem:kappa-alpha-rep}), guaranteeing that the family
  $(\log_{i}(\ka{\alpha}))$ in the definition of
  $\simpled_{\li}(\lambda)$ is summable. The largest possible value of
  $\ka{\alpha}$ is $\kappa_{0}=\omega$, which implies that if
  $\lambda\kgreater\omega$ we have
  \[
  \simpled_{\li}(\lambda)=\exp\left(\sum_{i=1}^{\infty}\log_{i}(\lambda)\right)=\simpled_{\li}'(\lambda).
  \]

  Another special case is when $\lambda=\ka{\alpha}$ for some
  $\alpha\in\on$.  In this case the terms of the form
  $\log_{i}(\lambda)$ cancel out and the formula specializes to:
  \[
  \simpled_{\li}(\ka{\alpha})=\exp\left(-\sum_{\beta<\alpha}\sum_{i=1}^{\infty}\log_{i}(\ka{\beta})\right)
  \]
  where $\beta$ ranges in $\on$. In particular,
  $\simpled_{\li}(\omega)=\simpled_{\li}(\kappa_{0})=1$.

  In any case, we have $\simpled_{\li}(\lambda)\in\M$ for all
  $\lambda\in\li$.
\end{rem}

Before proving that $\simpled_{\li}:\li\to\no^{>0}$ extends to a
surreal derivation $\simpled:\no\to\no$, let us first verify that the
necessary condition given by \prettyref{prop:vell-log-log} is met.

\begin{prop}
  \label{prop:simple-D-log-log}For all $\lambda,\mu\in\li$,
  $\log(\simpled_{\li}(\lambda))-\log(\simpled_{\li}(\mu))\vless\max\left\{
    \lambda,\mu\right\}$.
\end{prop}
\begin{proof}
  Without loss of generality we may assume that
  $\mu<\lambda$. Clearly, the inequality $\ka{\alpha}\kgeq\lambda$
  implies $\ka{\alpha}\kgeq\mu$, hence
  \[
  \log(\simpled_{\li}(\lambda))-\log(\simpled_{\li}(\mu)) =
  \sum_{\mu\kleq\ka{\alpha}\kless\lambda}\sum_{i=1}^{\infty}\log_{i}(\ka{\alpha})+\sum_{i=1}^{\infty}(\log_{i}(\lambda)-\log_{i}(\mu)),
  \]
  where $\alpha$ ranges in $\on$. It follows that
  \[
  \log(\simpled_{\li}(\lambda))-\log(\simpled_{\li}(\mu))\vleq\max_{i\geq1}\left\{
    \log_{i}(\lambda),\log_{i}(\mu),\max_{\mu\kleq\ka{\alpha}\kless\lambda}\left\{
      \log_{i}(\ka{\alpha})\right\} \right\} .
  \]

  However, if $\ka{\alpha}\kless\lambda$, then
  $\log_{i}(\ka{\alpha})\vless\lambda$ for all $i\geq1$, and moreover
  $\log_{i}(\mu)\vless\mu\vless\lambda$ and
  $\log_{i}(\lambda)\vless\lambda$ for all $i\geq1$. Therefore, the
  right hand side of the above inequality is
  $\vless\lambda=\max\{\lambda,\mu\}$, as desired.
\end{proof}

\begin{prop}
  \label{prop:simple-D-exp}For all $\lambda\in\li$,
  $\simpled_{\li}(\exp(\lambda))=\exp(\lambda)\simpled_{\li}(\lambda)$.
\end{prop}
\begin{proof}
  Let $\lambda\in\li$. Clearly, $\ka{\alpha}\kgeq\lambda$ if and only
  if $\ka{\alpha}\kgeq\exp(\lambda)$. Therefore,
  \begin{multline*}
    \simpled_{\li}(\exp(\lambda)) = \exp\left(-\sum_{\ka{\alpha}\kgeq\exp(\lambda)}\sum_{i=1}^{\infty}\log_{i}(\ka{\alpha})+\sum_{i=1}^{\infty}\log_{i}(\exp(\lambda))\right) =\\
    =
    \exp\left(-\sum_{\ka{\alpha}\kgeq\lambda}\sum_{i=1}^{\infty}\log_{i}(\ka{\alpha})+\sum_{i=1}^{\infty}\log_{i}(\lambda)+\lambda\right)=\exp(\lambda)\simpled_{\li}(\lambda).\qedhere
  \end{multline*}

\end{proof}

The proof that $\simpled_{\li}$ extends to a surreal derivation is
done by induction on the rank $\ntrank$, and uses ideas from
\cite{Schmeling2001}.  In the proof we shall not use the actual
definition of $\simpled_{\li}$, but only the fact that
$\simpled_{\li}$ satisfies the inequalities of
\prettyref{prop:vell-log-log}, is compatible with $\exp$, and takes
values in $\T$.

\subsection{Path-derivatives}

We consider the following sequences of terms, as in
\cite{Schmeling2001}.

\begin{defn}
  \label{def:path}We call \textbf{path} a sequence of terms
  $P:\N\to\T$ such that $P(i+1)$ is a term of $\vell(P(i))$ for all
  $i\in\N$.  We call $\mathcal{P}(x)$ the set of paths such that
  $P(0)$ is a term of $x$.
\end{defn}

Note that $P(0)\nin\R$ (as otherwise there would be no possible value
for $P(1)$) and $P(i+1)\in\R^{*}\M^{>1}$ for every $i$ (because
$\J\cap\R^{*}\M=\R^{*}\M^{>1}$).

\begin{rem}
  \label{rem:paths-in-L} If $\lambda\in\li$ there exists a unique path
  $P$ such that $P(0)=\lambda$, and for that path
  $P(i)=\log_{i}(\lambda)\in\li$ for all $i$.
\end{rem}

\begin{defn}
  \label{def:path-D}Given a path $P:\N\to\R^{*}\M$ we define its
  \textbf{path-derivative} $\simpled(P)\in\R\M$ as follows:
  \begin{enumerate}
  \item if there is $k\in\N$ such that $P(k)\in\li$, we let
    $\simpled(P):=\prod_{i<k}P(i)\cdot\simpled_{\li}(P(k))$;
  \item if $P(i)\notin\li$ for all $i\in\N$, then $\simpled(P):=0$.
  \end{enumerate}
\end{defn}

The value of $\simpled(P)$ does not depend on the choice of $k$ in
(1), since $\simpled_{\li}(P(i))=P(i)\cdot\simpled_{\li}(P(i+1))$
whenever $P(i)\in\li$, thanks to the fact that $\simpled_{\li}$ is
compatible with $\exp$. Therefore, if $P(i)\in\li$, then for every
$k\geq i$ we have
\[
\simpled(P)=P(0)\cdot P(1)\cdot\ldots\cdot
P(k-1)\cdot\simpled_{\li}(P(k)).
\]
We now wish to define $\simpled(x)$ as the sum of all the
path-derivatives of the paths in $\mathcal{P}(x)$. Indeed, we can
prove that the family
$\left(\simpled(P)\suchthat P\in\mathcal{P}(x)\right)$ is summable.

\begin{lem}
  \label{lem:paths-get-small}If $P$ is a path, then
  $1\vless P(i+1)\vleq\log(|P(i)|)\vless P(i)$ for all $i>0$.
\end{lem}
\begin{proof}
  Trivial, since $P(i)\in\J^{*}$ for all $i>0$.
\end{proof}

\begin{lem}
  \label{lem:common-part}If $t\vleq u$ are in \textup{$\T$}, and $t'$
  is a term of $\vell(t)$ but not of $\vell(u)$, then
  $(t')^{n}\vless\frac{u}{t}$ for all $n\in\N$.
\end{lem}
\begin{proof}
  We need to prove $n\cdot\vell(t')<\vell(u)-\vell(t)$. The hypothesis
  on $t'$ implies that $\vell(u)\neq\vell(t)$ and that $r\cdot t'$ is
  a term of $\vell(u)-\vell(t)$ for some $r\in\R^{*}$. In particular,
  $t'\vleq\vell(u)-\vell(t)$. Now observe that $\vell(u)-\vell(t)$ is
  positive and belongs to $\J$, hence $\vell(u)-\vell(t)>\N$. Since
  $\exp(x)>x^{n}$ for all $x>\N$ and $n\in\N$, we have
  \[
  (t')^{n}\vleq(\vell(u)-\vell(t))^{n}\vless\exp(\vell(u)-\vell(t))\veq\frac{u}{t}
  \]
  for all $n\in\N$, as desired.
\end{proof}

\begin{lem}
  \label{lem:path-D-decr}Let $P$, $Q$ be two paths such that
  $\simpled(P),\simpled(Q)\neq0$.

  If $P(0)\vleq Q(0)$ and $P(1)^{n}\vless\frac{Q(0)}{P(0)}$ for all
  $n\in\N$, then $\simpled(P)\vless\simpled(Q)$.

  More generally, suppose that there exists $i$ such that
  \begin{enumerate}
  \item for all $j\leq i$, $P(j)\vleq Q(j)$;
  \item $P(i+1)^{n}\vless\frac{Q(i)}{P(i)}$ for all $n\in\N$.
  \end{enumerate}
  Then $\simpled(P)\vless\simpled(Q)$.
\end{lem}
\begin{proof}
  We prove the first part, as the second then follows easily.

  For the sake of notation, write $x_{j}:=P(j)$ and $y_{j}:=Q(j)$.
  Let $k>1$ be such that $x_{k},y_{k}\in\li$. We need to prove that
  \[
  \simpled(P)=x_{0}\cdot x_{1}\cdot\ldots\cdot
  x_{k-1}\cdot\simpled_{\li}(x_{k})\vless y_{0}\cdot
  y_{1}\cdot\ldots\cdot y_{k-1}\cdot\simpled_{\li}(y_{k})=\simpled(Q).
  \]

  We observe that $y_{2},\dots,y_{k-1}\in\J$ are infinite. Therefore,
  it suffices to prove the stronger inequality
  \[
  x_{0}\cdot x_{1}\cdot\ldots\cdot
  x_{k-1}\cdot\simpled_{\li}(x_{k})\vless y_{0}\cdot
  y_{1}\cdot\simpled_{\li}(y_{k}),
  \]
  or equivalently,
  \[
  x_{1}\cdot\ldots\cdot
  x_{k-1}\cdot\frac{\simpled_{\li}(x_{k})}{\simpled_{\li}(y_{k})}\vless\frac{y_{0}y_{1}}{x_{0}}.
  \]

  By \prettyref{lem:paths-get-small},
  $1\vless x_{k}\vless\dots\vless x_{2}\vleq\log|x_{1}|\vless x_{1}$,
  and similarly $y_{k}\vleq\log|y_{1}|\vless y_{1}$. By
  \prettyref{prop:simple-D-log-log} we have
  \[
  \log(\simpled_{\li}(x_{k}))-\log(\simpled_{\li}(y_{k}))\vless\max\left\{
    x_{k},y_{k}\right\} \vleq\max\{\log|x_{1}|,\log|y_{1}|\}.
  \]
  In particular,
  $\frac{\simpled_{\li}(x_{k})}{\simpled_{\li}(y_{k})}\leq\max\{|x_{1}|,|y_{1}|\}.$
  By the hypothesis on $x_{1}=P(1)$ we get
  \[
  \left|x_{1}\cdot\ldots\cdot
    x_{k-1}\cdot\frac{\simpled_{\li}(x_{k})}{\simpled_{\li}(y_{k})}\right|\leq|x_{1}|^{k-1}\cdot\max\left\{
    |x_{1}|,|y_{1}|\right\}
  \leq|x_{1}|^{k}\cdot|y_{1}|\vless\frac{y_{0}y_{1}}{x_{0}},
  \]
  reaching the desired conclusion.
\end{proof}

\begin{cor}
  \label{cor:path-D-decr}Let $P$, $Q$ be two paths such that
  $\simpled(P),\simpled(Q)\neq0$.  Suppose that there exists $i\in\N$
  such that:
  \begin{enumerate}
  \item for all $j\leq i$, $P(j)\vleq Q(j)$;
  \item $P(i+1)$ is not a term of $\vell(Q(i))$.
  \end{enumerate}
  Then $\simpled(P)\vless\simpled(Q)$.
\end{cor}
\begin{proof}
  By \prettyref{lem:common-part} we have
  $P(i+1)^{n}\vless\frac{Q(i)}{P(i)}$ for all $n\in\N$. It then
  follows from \prettyref{lem:path-D-decr} that
  $\simpled(P)\vless\simpled(Q)$, as desired.
\end{proof}

\begin{lem}
  \label{lem:nested-rank-path}Given $P\in\mathcal{P}(x)$, we have
  $\ntrank(P(0))\leq\ntrank(x)$, and if the equality holds, then the
  support of $x$ has a minimum $\m$ and $P(0)=r\m$ for some
  $r\in\R^{*}$.

  Similarly, for all $i\in\N$ we have
  $\ntrank(P(i+1))\leq\ntrank(P(i))$, and if the equality holds, then
  the support of $\vell(P(i))$ has a minimum $\m$ and $P(i+1)=r\m$ for
  some $r\in\R^{*}$.
\end{lem}
\begin{proof}
  Immediate by \prettyref{prop:ntrank-term}.
\end{proof}

\begin{cor}
  For all $x\in\no$, there is at most one path $P\in\mathcal{P}(x)$
  such that $\ntrank(P(i))=\ntrank(x)$ for all $i\in\N$.
\end{cor}

\begin{prop}
  \label{prop:path-D-summable}For all $x\in\no$, the family
  $(\simpled(P)\suchthat P\in\mathcal{P}(x))$ is summable.
\end{prop}
\begin{proof}
  Since $\simpled(P)\in\R\M$ for all $P\in\mathcal{P}(x)$, we just
  need to prove that no sequence of distinct paths $(P_{j})_{j\in\N}$
  in $\mathcal{P}(x)$ is such that
  $\simpled(P_{0})\vleq\simpled(P_{1})\vleq\ldots$.  Suppose by
  contradiction that such a sequence exists. Since the paths are
  distinct, there exists a minimum integer $m$ such that
  $P_{j}(m)\neq P_{k}(m)$ for some $j,k$, and clearly
  $P_{j}(i)=P_{0}(i)$ for all $i<m$ and $j$.

  Let $\alpha:=\ntrank(x)$. We work by primary induction on $\alpha$
  and secondary induction on $m$ to reach a contradiction. Let
  $r\exp(\gamma)$ be the term of maximum $\vell$-value among
  $\{P_{j}(0)\suchthat j\in\N\}$.

  If $\ntrank(\gamma)=\alpha$, then by
  \prettyref{lem:nested-rank-path} $r\exp(\gamma)$ is also the term of
  \emph{minimum} $\vell$-value, hence $P_{j}(0)=P_{0}(0)$ for all $j$,
  and therefore $m>0$.

  If $\ntrank(\gamma)<\alpha$, after extracting a subsequence, we may
  assume that $r\exp(\gamma)=P_{0}(0)\vgeq P_{1}(0)\vgeq\ldots$
  (possibly changing the value of $m$, which however will not be
  relevant in this case).

  Now, if $P_{j}(1)$ is not a term of $\gamma=\vell(P_{0}(0))$ for
  some $j\in\N$, then by \prettyref{cor:path-D-decr} we get
  $\simpled(P_{0})\vgreater\simpled(P_{j})$, a
  contradiction. Therefore, $P_{j}(1)$ is a term of $\gamma$ for all
  $j\in\N$. Consider the paths $P_{j}'$ defined by
  $P_{j}'(i):=P_{j}(i+1)$ for $i\in\N$ and let $m'$ be the minimum
  integer such that $P_{j}'(m')\neq P_{k}'(m')$ for some
  $j,k$. Clearly, in the case $\ntrank(\gamma)=\alpha$ we have
  $m'=m-1$. Moreover, $P_{j}'\in\mathcal{P}(\gamma)$ for all $j\in\N$.

  By the equality
  \[
  \simpled(P_{j})=P_{j}(0)\cdot\simpled(P_{j}')
  \]
  and $P_{0}(0)\vgeq P_{1}(0)\vgeq\ldots$, it follows that
  $\simpled(P_{0}')\vleq\simpled(P_{1}')\vleq\ldots$.  Since we have
  either $\ntrank(\gamma)<\alpha$, or $\ntrank(\gamma)=\alpha$ and
  $m'<m$, this contradicts the inductive hypothesis that no such
  sequence may exist in $\mathcal{P}(\gamma)$.

  Therefore, $(\simpled(P)\suchthat P\in\mathcal{P}(x))$ is summable,
  as desired.
\end{proof}

\subsection{A surreal derivation}

Thanks to \prettyref{prop:path-D-summable}, we can finally define
$\simpled:\no\to\no$ by summing all the path-derivatives.

\begin{defn}
  \label{def:D}We define $\simpled:\no\to\no$ by
  \[
  \simpled(x):=\sum_{P\in\mathcal{P}(x)}\simpled(P).
  \]
\end{defn}

We claim that $\simpled:\no\to\no$ is indeed a surreal derivation.

\begin{defn}
  \label{def:dominant-path}Given $x\in\no\setminus\R$, its
  \textbf{dominant path} is the path $Q\in\mathcal{P}(x)$ such that
  $Q(0)$ is the term of maximum non-zero $\vell$-value of $x$ and
  $Q(i+1)$ is the leading term of $\vell(Q(i))$ for all
  $i\in\N$.
\end{defn}

\begin{lem}
  \label{lem:leading-D}If $x\in\no\setminus\R$ and $Q$ is the dominant
  path of $x$, then $\simpled(Q)\neq0$ and $\simpled(Q)$ is the
  leading term of $\simpled(x)$.
\end{lem}
\begin{proof}
  Let $Q\in\mathcal{P}(x)$ be the dominant path of $x$. Without loss
  of generality, we may assume that $x\not\veq1$ (if $x\veq1$, it
  suffices to subtract the leading real number), so that $Q(0)$ is the
  leading term of $x$, and $Q(i+1)$ is the leading term of
  $\vell(Q(i))$.  Letting $\vell_{i}$ be the $i$-fold composition
  $\vell\circ\ldots\circ\vell$, it follows that
  $\vell(Q(i))=\vell_{i+1}(x)$ for all $i\in\N$. By
  \prettyref{cor:vell-n-fold}, there exists $k\in\N$ such that
  $Q(k)\in\li$, and therefore $\simpled(Q)\neq0$. Let
  $P\in\mathcal{P}(x)$ be any other path different from $Q$ and such
  that $\simpled(P)\neq0$.

  We distinguish two cases. Suppose first that $P(i+1)$ is a term of
  $\vell(Q(i))$ for all $i$. By definition of $Q$, this clearly
  implies that $Q(i)\vgeq P(i)$ for all $i\in\N$. Moreover, we must
  have $Q(i)=P(i)$ for all $i>k$. Since $Q\neq P$, there must be an
  $i\leq k$ such that $Q(i)\vgreater P(i)$, and it follows immediately
  that $\simpled(Q)\vgreater\simpled(P)$.

  In the other case, take the minimal $j\in\N$ such that $P(j+1)$ is
  not a term of $\vell(Q(j))$. By definition of $Q$, we have
  $Q(i)\vgeq P(i)$ for all $j\leq i$. By \prettyref{cor:path-D-decr},
  we have $\simpled(Q)\vgreater\simpled(P)$ in this case as
  well. Since $\simpled(x)=\sum_{P\in\mathcal{P}(x)}\simpled(P)$ by
  definition, we get that $\simpled(Q)$ is the leading term of
  $\simpled(x)$, as desired.
\end{proof}

\begin{cor}
  \label{cor:D-ker}For all $x\in\no$, $\simpled(x)=0$ if and only if
  $x\in\R$.
\end{cor}
\begin{proof}
  By \prettyref{lem:leading-D}, if $x\notin\R$, then
  $\simpled(x)\neq0$.  Conversely, if $x\in\R$, then
  $\mathcal{P}(x)=\emptyset$, whence $\simpled(x)=0$, as
  desired.
\end{proof}

\begin{cor}
  \label{cor:D-H-field}If $x>\N$, then $\simpled(x)>0$.
\end{cor}
\begin{proof}
  By \prettyref{lem:leading-D}, it suffices to prove that if $x>\N$
  and $P$ is the dominant path of $x$, then $\simpled(P)>0$.

  By definition,
  $\simpled(P)=P(0)\cdot P(1)\cdot\ldots\cdot
  P(k-1)\cdot\simpled_{\li}(P(k))$,
  where $k$ is such that $P(k)\in\li$. We can easily prove by
  induction that $P(i)>\N$ for all $i$. Clearly $P(0)>\N$ holds by
  assumption.  If $P(i)>\N$, then $\vell(P(i))>0$; since $P(i+1)$ is
  the leading term of $\vell(P(i))$, we must have $P(i+1)>0$ as
  well. Since $P(i+1)\in\J$, it follows that $P(i+1)>\N$, concluding
  the induction. Moreover, $\simpled_{\li}(P(k))>0$, since
  $\simpled_{\li}$ takes only positive values. Therefore,
  $\simpled(P)>0$, as desired.
\end{proof}

\begin{prop}
  \label{prop:D-strong-additive}The function $\simpled$ is strongly
  linear, hence strongly additive.
\end{prop}
\begin{proof}
  It suffices to observe that if $x=\sum_{\m}x_{\m}\m$, then
  \[
  \simpled(x) = \sum_{P\in\mathcal{P}(x)}\simpled(P) =
  \sum_{\m\in\supp(x)}\sum_{P\in\mathcal{P}(\m)}x_{\m}\simpled(P) =
  \sum_{\m}x_{\m}\simpled(\m).
  \]
  By \prettyref{rem:summability-criterion} it follows that $\simpled$
  is strongly additive.
\end{proof}

\begin{prop}
  \label{prop:D-exp-J}For all $\gamma\in\J$,
  $\simpled(\exp(\gamma))=\exp(\gamma)\simpled(\gamma)$.
\end{prop}
\begin{proof}
  Let $\gamma\in\J$. Consider the bijection
  $\mathcal{P}(\exp(\gamma))\to\mathcal{P}(\gamma)$ sending
  $P\in\mathcal{P}(\exp(\gamma))$ to the path
  $P'\in\mathcal{P}(\gamma)$ defined by $P'(i):=P(i+1)$ for
  $i\in\N$. Recall that by definition
  $\simpled(P)=\exp(\gamma)\cdot\simpled(P')$. We thus obtain
  \[
  \simpled(\exp(\gamma))=\sum_{P\in\mathcal{P}(\exp(\gamma))}\simpled(P)=\exp(\gamma)\sum_{P\in\mathcal{P}(\gamma)}\simpled(P)=\exp(\gamma)\simpled(\gamma).\qedhere
  \]
\end{proof}

\begin{prop}
  \label{prop:D-leibniz}For all $x,y\in\no$,
  $\simpled(xy)=x\simpled(y)+y\simpled(x)$.
\end{prop}
\begin{proof}
  We first prove the conclusion on $\M$. Let $\m,\n\in\M$ and write
  $\m=\exp(\gamma),\n=\exp(\delta)$ with $\gamma,\delta\in\J$. By
  \prettyref{prop:D-exp-J}, we get
  $\simpled(\m)=\exp(\gamma)\simpled(\gamma)$,
  $\simpled(\n)=\exp(\delta)\simpled(\delta)$ and
  $\simpled(\m\n)=\exp(\gamma+\delta)\simpled(\gamma+\delta)$.  By
  \prettyref{prop:D-strong-additive}, we conclude
  $\simpled(\m\n)=\m\simpled(\n)+\simpled(\m)\n$.

  For the general case, let $x,y\in\no$ and write
  $x=\sum_{\m}x_{\m}\m$ and $y=\sum_{\n}y_{\n}\n$. By
  \prettyref{prop:D-strong-additive} again,
  $\simpled(xy)=\simpled(\sum_{\m,\n}x_{\m}y_{\n}\m\n)=\sum_{\m,\n}x_{\m}y_{\n}\simpled(\m\n)=\sum_{\m,\n}(x_{\m}\m\cdot
  y_{\n}\simpled(\n)+x_{\m}\simpled(\m)\cdot
  y_{\n}\n)=x\simpled(y)+y\simpled(x)$, as desired.
\end{proof}

\begin{cor}
  \label{cor:D-exp}For all $x\in\no$,
  $\simpled(\exp(x))=\exp(x)\simpled(x)$.
\end{cor}
\begin{proof}
  Let $x\in\no$. Write $x=\gamma+r+\varepsilon$ with $\gamma\in\J$,
  $r\in\R$, $\varepsilon\in o(1)$. Since $\varepsilon$ is
  infinitesimal, we can apply the strong additivity
  (\prettyref{prop:D-strong-additive}) and the Leibniz rule
  (\prettyref{prop:D-leibniz}) as in \prettyref{rem:D-exp-finite} to
  obtain
  $\simpled(\exp(\varepsilon))=\exp(\varepsilon)\simpled(\varepsilon)$.
  Since $\gamma\in\J$, we have
  $\simpled(\exp(\gamma))=\exp(\gamma)\simpled(\gamma)$ by
  \prettyref{prop:D-exp-J}. By \prettyref{cor:D-ker}, we also have
  $\simpled(\exp(r))=0=\exp(r)\simpled(r)$. By Leibniz' rule
  (\prettyref{prop:D-leibniz}) applied to the product
  $\exp(\gamma)\exp(r)\exp(\varepsilon)=\exp(x)$ we conclude that
  $\simpled(\exp(x))=\exp(x)\simpled(x)$, as desired.
\end{proof}

Therefore, $\simpled$ is a surreal derivation.

\begin{thm}
  \label{thm:D-surreal}The function $\simpled:\no\to\no$ is a surreal
  derivation extending $\simpled_{\li}$.
\end{thm}
\begin{proof}
  The function $\simpled$ satisfies Leibniz' rule by
  \prettyref{prop:D-leibniz}, strong additivity by
  \prettyref{prop:D-strong-additive}, it is compatible with
  exponentiation by \prettyref{cor:D-exp}, its kernel is $\R$ by
  \prettyref{cor:D-ker}, and it is an $H$-field derivation by
  \prettyref{cor:D-H-field}.
\end{proof}

\begin{rem}
  \label{rem:simple-D-bracket}The restriction of $\simpled:\no\to\no$
  to $\li$, namely the map $\simpled_{\li}$, takes values in the
  subfield $\bracket{\li}$ of $\no$. Since $\simpled$ is calculated
  using finite products and infinite sums, we can easily verify that
  $\simpled(\bracket{\li})\subseteq\bracket{\li}$.  Therefore, the
  restriction $\simpled_{\restriction\bracket{\li}}$ induces a
  structure of $H$-field on $\bracket{\li}$.
\end{rem}

In more generality, with the same proof we obtain:

\begin{thm}
  \label{thm:extending-pre-D}Let $\somed:\li\to\no^{>0}$ be a map such
  that:
  \begin{enumerate}
  \item for all $\lambda,\mu\in\li$,
    $\log(D(\lambda))-\log(D(\mu))\vless\max\left\{
      \lambda,\mu\right\} $;
  \item for all $\lambda\in\li$,
    $D(\exp(\lambda))=\exp(\lambda)D(\lambda)$;
  \item $D(\li)\subset\T$.
  \end{enumerate}
  Then $D$ extends to a surreal derivation on $\no$.
\end{thm}

Once we have a derivation, we can apply Ax's theorem to deduce some
transcendence results. If $V$ is a $\Q$-vector space and $W$ is a
subspace of $V$, we say that a set $H\subset V$ is \emph{$\Q$-linearly
  independent modulo $W$} if its projection to the quotient $V/W$ is
$\Q$-linearly independent.

\begin{thm}[\cite{Ax1971}]
  Let $(K,D)$ be a differential field. If
  $x_{1},\dots,x_{n},y_{1},\dots,y_{n}$ are such that
  $D(x_{i})=D(y_{i})/y_{i}$ for $i=1,\dots,n$, and if
  $x_{1},\dots,x_{n}$ are $\Q$-linearly independent modulo $\ker(D)$,
  then
  \[
  \mathrm{tr.deg.}_{\ker(D)}(x_{1},\dots,x_{n},y_{1},\dots,y_{n})\geq
  n+1.
  \]
\end{thm}

In our case, it suffices to take $(\no,\simpled)$ as differential
field and $y_{i}=\exp(x_{i})$ to deduce the following corollary.

\begin{cor}
  \label{cor:schanuel}If $x_{1},\dots,x_{n}\in\no$ are $\Q$-linearly
  independent modulo $\R$, then
  \[
  \mathrm{tr.\mathrm{deg}.}_{\R}(x_{1},\dots,x_{n},\exp(x_{1}),\dots,\exp(x_{n}))\geq
  n+1.
  \]
\end{cor}

We remark that this is just a special case of a much more general
statement regarding all models of the theory of $\R_{\exp}$. We recall
the general version for completeness.

\begin{thm}[\cite{Jones2008}, \cite{Kirby2010}]
  Let $R_{E}$ be a model of the theory of $\R_{\exp}$. If
  $x_{1},\dots,x_{n}\in R$ are $\Q$-linearly independent modulo
  $\mathrm{dcl}(\emptyset)$, then
  \[
  \mathrm{tr.deg.}_{\mathrm{dcl}(\emptyset)}(x_{1},\dots,x_{n},E(x_{1}),\dots,E(x_{n}))\geq
  n+k
  \]
  where $k$ is the exponential transcendence degree of
  $x_{1},\dots,x_{n}$ over $\mathrm{dcl}(\emptyset)$.
\end{thm}

The above statement can be proved by noting that the definable closure
operator coincides with the exponential-algebraic closure \cite[Thm.\
4.2]{Jones2008} and that the above Schanuel type statement holds
modulo the exponential-algebraic closure of the empty set \cite[Thm.\
1.2]{Kirby2010}.

\section{Integration}

We can easily prove that the derivation $\simpled$ of
\prettyref{def:D} is surjective, or in other words, every surreal
number has an integral.  Our proof is based on a theorem of Rosenlicht
that links the existence of integrals to the values of the logarithmic
derivative \cite{Rosenlicht1983}.

We quote here the relevant theorem. Let $K$ be a Hardy field. If
$f\in K$, we denote by $f'$ its derivative, and we let $v$ be the
valuation on $K$ whose valuation ring is the convex hull of
$\mathbb{Q}$ in $K$. Recall that $f\sim g$ means $v(f-g)>v(g)$.

\begin{fact}[{\cite[Thm.\ 1]{Rosenlicht1983}}]
  Let $K$ be a Hardy field and consider the set of valuations
  $\Psi:=\left\{ v(f'/f)\suchthat f\in K,\ v(f)\neq0\right\} $.  If
  $f\in K^{*}$ is such that $v(f)\neq\sup\Psi$, then there exists
  $u_{0}\in K^{*}$with $v(u_{0})\neq0$ such that whenever
  $u\in K^{*}$and $|v(u_{0})|\geq|v(u)|>0$ we have
  \[
  \left(f\cdot\frac{fu/u'}{(fu/u')'}\right)'\sim f.
  \]
\end{fact}

The result of Rosenlicht shows that every $f\in K^{*}$ with
$f\neq\sup\Psi$ has an asymptotic integral, i.e., a function $g$ whose
derivative $g'$ is asymptotic to $f$. In particular, if $\sup\Psi$
does not exist, then \emph{every} $f\in K^{*}$ has an asymptotic
integral.  The proof is purely algebraic and holds more generally in
the context of $H$-fields, and in particular it holds for the surreal
numbers $\no$ equipped with our derivation $\simpled:\no\to\no$ and
the valuation $-\vell$. To be able to apply Rosenlicht's result, the
first step is to check whether
$\left\{ \vell(\simpled(x)/x)\suchthat x\in\no,\ \vell(x)\neq0\right\}
$ has an infimum.

\begin{prop}
  \label{prop:psi-l-no-inf}The class
  $\Psi_{\li}:=\{\vell(\simpled(\lambda)/\lambda)\suchthat\lambda\in\li\}$
  has no infimum in $\J$.
\end{prop}
\begin{proof}
  Since $\simpled(\lambda)/\lambda=\simpled(\log(\lambda))$ and
  $\li=\log(\li)$, we have that
  $\Psi_{\li}=\{\vell(\simpled(\lambda))\suchthat\lambda\in\li\}$.
  Moreover, the sequence $y(\alpha):=\vell(\simpled(\ka{\alpha}))$ is
  co-initial in $\Psi_{\li}$ by \prettyref{prop:kappa-coinitial} and
  \prettyref{rem:D-incr-on-inf}, so it suffices to prove that the
  class $\left\{ y(\alpha)\suchthat\alpha\in\on\right\} $ has no
  infimum in $\J$. Recall that we have
  $y(\alpha)=\log(\simpled(\kappa_{-\alpha}))=-\sum_{\beta<\alpha}\sum_{i=1}^{\infty}\log_{i}(\kappa_{-\beta})$
  and observe that if $\beta<\alpha$ , then $y(\beta)\trunc y(\alpha)$
  and $y(\beta)>y(\alpha)$.

  Let $x<y(\alpha)$ for all $\alpha\in\on$, with $x\in\J$. We must
  show that $x$ is not an infimum of
  $\left\{ y(\alpha)\suchthat\alpha\in\on\right\} $ in $\J$. Since the
  supports $\supp(y(\alpha))$ are increasing in $\alpha$, their
  intersection with $\supp(x)$ must stabilize, namely there are
  $A\subseteq\supp(x)$ and $\gamma\in\on$ such that
  $\supp(y(\alpha))\cap\supp(x)=A$ for all $\alpha\geq\gamma$. Let
  $\m$ be the maximal monomial such that $x_{\m}\neq
  y(\gamma)_{\m}$.
  For all $\alpha\geq\gamma$, by construction of $\gamma$, and since
  $y(\gamma)\trunceq y(\alpha)$, the same $\m$ is also the maximal
  monomial such that $x_{\m}\neq y(\alpha)_{\m}$, and
  $y(\alpha)_{\m}=y(\gamma)_{\m}$. Since $x<y(\gamma)$ we must have
  $x_{\m}<y(\gamma)_{\m}$. Now take any $x'\in\J$ such that
  $x'|\m=x|\m$ and $x_{\m}<x'_{\m}<y(\gamma)_{\m}$. Then
  $x<x'<y(\alpha)$ for all $\alpha\ge\gamma$, and therefore for all
  $\alpha$. This means that $x$ is not an infimum of
  $\{y(\alpha)\suchthat\alpha\in\on\}$ in $\J$, as desired.
\end{proof}

In fact, the same proof also shows that $\Psi_{\li}$ has no infimum
even in $\no$.

\begin{cor}
  \label{cor:psi-no-inf}The class
  $\Psi:=\left\{ \vell(\simpled(x)/x)\suchthat x\in\no,\
    \vell(x)\neq0\right\} $ has no infimum in $\J$.
\end{cor}
\begin{proof}
  We have $\simpled(x)/x=\simpled(\log|x|)$. Moreover, $\vell(x)\neq0$
  if and only if $\log|x|>\N$. Since $\log(\no^{>0})=\no$, we have
  that $\Psi=\left\{ \vell(\simpled(x))\suchthat x>\N\right\} $. Since
  $\li$ is co-initial with all the infinite positive elements of $\no$
  (see \prettyref{prop:kappa-coinitial}), then
  $\Psi_{\li}=\{\vell(\simpled(\lambda))\suchthat\lambda\in\li\}$ is
  co-initial with $\Psi$ by \prettyref{rem:D-incr-on-inf}. But
  $\Psi_{\li}$ has no infimum in $\J$ by
  \prettyref{prop:psi-l-no-inf}, so $\Psi$ does not have it either.
\end{proof}

We can now apply \cite[Thm.\ 1]{Rosenlicht1983} to show that every
surreal number has an asymptotic integral, namely, for every
$x\in\no^{*}$ there is a $y\in\no^{*}$ such that
$x\sim\simpled(y)$. For later convenience, we construct an asymptotic
integral $y$ belonging to $\R^{*}\M^{\neq1}$.

\begin{prop}
  \label{prop:class-asy-int}There is a class function
  $A:\no^{*}\to\T^{\neq1}$ such that $x\sim\simpled(A(x))$ for all
  $x\in\no^{*}$.
\end{prop}
\begin{proof}
  We define the function $A:\no^{*}\to\T^{\neq1}$ as follows. Let
  $x\in\no^{*}$.  By \prettyref{cor:psi-no-inf}, $\vell(x)$ is not an
  infimum for $\Psi$. Therefore, by \cite[Thm.\ 1]{Rosenlicht1983}
  applied to $(\no,\simpled)$, there is $u_{0}\in\no$ with
  $\vell(u_{0})\neq0$ such that for any $u\in\no$ with
  $0<|\vell(u)|\leq|\vell(u_{0})|$ we have
  \[
  x\sim\simpled\left(x\cdot\frac{(xu/\simpled(u))}{\simpled(xu/\simpled(u))}\right).
  \]

  This gives us an asymptotic integral
  $y:=x\cdot\frac{(xu/\simpled(u))}{\simpled(xu/\simpled(u))}$ of $x$
  which in fact depends on the choice of $u$. For the sake of
  definiteness, we choose $u:=\ka{\alpha}$ with $\alpha$ minimal (such
  an $\alpha$ always exists, since the elements of $\kappa_{\no}$ are
  co-initial in the positive infinite numbers by
  \prettyref{prop:kappa-coinitial}).

  We make a minor adjustment to obtain an asymptotic integral
  belonging to $\R^{*}\M^{\neq1}$. Let $r\in\R$ be the coefficient of
  the monomial $1$ in $y$, so that $y-r\not\veq1$, and define $A(x)$
  as the leading term of $(y-r)$. Clearly $A(x)\sim y-r$, hence
  $\simpled(A(x))\sim\simpled(y-r)=\simpled(y)\sim x$, while
  $A(x)\in\R^{*}\M^{\neq1}$, as desired.
\end{proof}

Using the above observation, one could try to use \cite[Thm.\
47]{Kuhlmann2011a} to obtain actual integrals; however, one should
adapt the notion of ``spherically complete'' to the \emph{class} $\no$
and verify that the proof goes through. This argument is rather
delicate, as the field $\no$, with the valuation $-\vell$, may not be
spherically complete if seen from a more powerful model of set theory,
for instance when using an inaccessible cardinal.

For the sake of completeness, we give a different self-contained
argument.  In order to find a solution to the differential equation
$\simpled(y)=x$ for a given $x$, we simply iterate the above procedure
for finding an asymptotic integral, and we verify that the procedure
converges using a specialized version of Fodor's lemma.

\begin{lem}[Specialized Fodor's lemma]
  \label{lem:fodor-classes} Let $f:\on\setminus\{\emptyset\}\to\on$ be
  a class function such that $f(\alpha)<\alpha$ for all
  $\alpha\in\on\setminus\{\emptyset\}$.  Then there exists
  $\beta\in\on$ such that $f^{-1}(\beta)$ is a proper class.
\end{lem}
\begin{proof}
  Suppose by contradiction that for each $\beta\in\on$ the class
  $f^{-1}(\beta)$ is a set. We define the following class function
  $g:\on\to\on$ by induction: given $\alpha\in\on$, we let $g(\alpha)$
  be the minimum ordinal strictly greater than all the elements of
  $f^{-1}(\beta)\cup\{g(\beta)\}$ for $\beta<\alpha$. This is a
  strictly increasing continuous function $g:\on\to\on$. As is well
  known, the ordinal $\alpha_{0}:=\sup_{n<\omega}g^{(n)}(0)$ satisfies
  $g(\alpha_{0})=\alpha_{0}$. By definition, we have
  $\alpha_{0}=g(\alpha_{0})>f^{-1}(\beta)$ for all $\beta<\alpha_{0}$,
  and in particular $f(\alpha_{0})\neq\beta$ for all
  $\beta<\alpha_{0}$. Therefore, $f(\alpha_{0})\geq\alpha_{0}$,
  contradicting the hypothesis.
\end{proof}

\begin{prop}
  \label{prop:liou-closed}The surreal derivation $\simpled:\no\to\no$
  is surjective.
\end{prop}
\begin{proof}
  Clearly, $\simpled(0)=0$, so $0$ is in the image of $\simpled$.

  Now take a surreal number $x\in\no^{*}$. We define inductively a
  sequence of terms $t_{\alpha}\in\T^{\neq1}$ as follows. We start
  with $t_{0}:=A(x)$. If $t_{\beta}$ has been defined for every
  $\beta<\alpha$, and $x\neq\sum_{\beta<\alpha}\simpled(t_{\beta})$,
  we define
  \[
  t_{\alpha}:=A\left(x-\sum_{\beta<\alpha}\simpled(t_{\beta})\right),
  \]
  otherwise we stop. We claim that $\vell(t_{\beta})$ is strictly
  decreasing for all $\beta<\alpha$, so that
  $\sum_{\beta<\alpha}t_{\beta}$ is a surreal number and
  $\sum_{\beta<\alpha}\simpled(t_{\beta})$ is its derivative, ensuring
  that $t_{\alpha}$ is well defined. In fact, we may assume by
  induction that $\vell(t_{\beta})$ is strictly decreasing and we only
  need to check that $\vell(t_{\alpha})<\vell(t_{\beta})$, i.e.\
  $t_{\alpha}\vless t_{\beta}$, for all $\beta<\alpha$.

  By construction, $t_{\beta}\not\veq1$ for all $\beta<\alpha$, hence,
  by \prettyref{prop:D-vell-incr}, we have that
  $\vell(\simpled(t_{\beta}))$ is strictly decreasing for
  $\beta<\alpha$. Now fix $\gamma<\alpha$.  By definition of
  asymptotic integral,
  \[
  \simpled(t_{\alpha})\sim
  x-\sum_{\beta<\alpha}\simpled(t_{\beta})\vleq\max\left\{
    \left|x-\sum_{\beta\leq\gamma}\simpled(t_{\beta})\right|,\left|\sum_{\gamma<\beta<\alpha}\simpled(t_{\beta})\right|\right\}
  .
  \]
  Note that this is true even if the last sum is empty, namely when
  $\alpha=\gamma+1$. Since
  $\simpled(t_{\beta})\vless\simpled(t_{\gamma})$ for all
  $\gamma<\beta<\alpha$, then
  $\sum_{\gamma<\beta<\alpha}\simpled(t_{\beta})\vless\simpled(t_{\gamma})$.
  Moreover, again by definition of asymptotic integral,
  $x-\sum_{\beta\leq\gamma}\simpled(t_{\beta})=\left(x-\sum_{\beta<\gamma}\simpled(t_{\beta})\right)-\simpled(t_{\gamma})\vless\simpled(t_{\gamma})$.
  Therefore, $\simpled(t_{\alpha})\vless\simpled(t_{\gamma})$, and by
  \prettyref{prop:D-vell-incr} we get $t_{\alpha}\vless t_{\gamma}$,
  as desired.

  We now claim that there is an $\alpha$ such that
  $x=\sum_{\beta<\alpha}\simpled(t_{\beta})$.  Suppose by
  contradiction that $x\neq\sum_{\beta<\alpha}\simpled(t_{\beta})$ for
  all $\alpha\in\on$. Let $\m_{\alpha}$ be the leading monomial of
  $x-\sum_{\beta<\alpha}\simpled(t_{\beta})$. Recall that by
  construction
  $\m_{\alpha}\veq
  x-\sum_{\beta<\alpha}\simpled(t_{\beta})\sim\simpled(t_{\alpha})$;
  since $\vell(\simpled(t_{\alpha}))$ is strictly decreasing, the
  sequence $\m_{\alpha}$ is strictly decreasing as well, and in
  particular injective.  Let $f:\on\to\on$ be the class function that
  sends $\alpha$ to the minimum $\beta\in\on$ such that
  $\m_{\alpha}\in\supp(\simpled(t_{\beta}))\cup\supp(x)$; clearly,
  such a $\beta$ always exists and it must be strictly less than
  $\alpha$.

  Since $f(\alpha)<\alpha$ for all $\alpha\in\on$, by
  \prettyref{lem:fodor-classes} there exists a $\beta\in\on$ such that
  $f^{-1}(\beta)$ is a proper class. However, by definition of $f$ the
  class $\left\{ \m_{\alpha}\,:\,\alpha\in f^{-1}(\beta)\right\} $ is
  actually a subset of $\supp(\simpled(t_{\beta}))\cup\supp(x)$.
  Since the map $\alpha\mapsto\m_{\alpha}$ is injective, this implies
  that $\supp(\simpled(t_{\beta}))\cup\supp(x)$ contains a proper
  class, a contradiction.

  Therefore, for some $\alpha$ we have
  $x=\sum_{\beta<\alpha}\simpled(t_{\beta})=\simpled\left(\sum_{\beta<\alpha}t_{\beta}\right)$,
  as desired.
\end{proof}

\begin{thm}
  \label{thm:liou-closed-small}The differential field $(\no,\simpled)$
  is a Liouville closed $H$-field with small derivation in the sense
  of \cite[p.\ 3]{Aschenbrenner2002}.
\end{thm}
\begin{proof}
  By \prettyref{prop:liou-closed}, the function $\simpled$ is
  surjective.  In particular, the differential equations
  $\simpled(x)=y$ and $\simpled(x)/x=\simpled(\log|x|)=y$ always have
  solution in $\no$, and therefore $(\no,\simpled)$ is
  Liouville-closed. Moreover, since $\simpled(\omega)=1$, we have that
  if $x\vless1$, then $\simpled(x)\vless\simpled(\omega)=1$, which
  means by definition that the derivation is small, as
  desired.
\end{proof}

\begin{rem}
  The conclusion of \prettyref{cor:psi-no-inf} applies to
  $\bracket{\li}$ as well. Since the remaining construction is done
  using just field operations and infinite sums, we can easily verify
  that $\bracket{\li}$, equipped with the derivation
  $\simpled_{\restriction\bracket{\li}}$ (see
  \prettyref{rem:simple-D-bracket}), is Liouville-closed as well.
\end{rem}

\section{\label{sec:transseries}Transseries}

As anticipated in \prettyref{sec:nested}, the fact that the nested
truncation $\ntrunceq$ is well-founded is related in an essential way
to the structure of $\no$ as a field of transseries. We discuss here
in which sense $\no$ can be seen as a field of transseries and compare
the result to a previous conjecture.

\subsection{\label{sub:elt4}Axiom ELT4 of \cite{Kuhlmann2014}}

We mentioned in the introduction that a rather natural object to
consider is the smallest subfield of $\no$ containing $\li$ and closed
under some natural operations.

\begin{defn}
  We call $\bracket{\li}$ the smallest subfield of $\no$ containing
  $\R(\li)$ and closed under infinite sums, exponentiation and
  logarithm.
\end{defn}

A natural question is whether $\no=\bracket{\li}$; we can see that
this is equivalent to the first part of Conjecture 5.2 in
\cite{Kuhlmann2014}.  However, we can verify that the inclusion is
strict. To prove this, we characterize $\bracket{\li}$ in terms of
paths.

\begin{prop}
  \label{prop:bracket-paths-in-L}For all $x\in\no$,
  $x\in\bracket{\li}$ if and only if for every path
  $P\in\mathcal{P}(x)$ there exists $i$ such that
  $P(i)\in\li$.
\end{prop}
\begin{proof}
  Let $\mathbb{F}$ be the class of all $x\in\no$ such that for every
  $P\in\mathcal{P}(x)$ there exists $i$ such that $P(i)\in\li$. If
  $x\nin\bracket{\li}$, then clearly there is some term
  $r\exp(\gamma)$ in $x$ with $\gamma\nin\bracket{\li}$. Iterating
  this procedure we produce an infinite path $P\in\mathcal{P}(x)$ with
  $P(0)=r\exp(\gamma)$ and $P(i)\nin\bracket{\li}$ for every
  $i\in\N$. In particular, $P(i)\nin\li$ for all $i$, hence
  $x\notin\F$. Since $x$ was arbitrary, we have proved
  $\F\subseteq\bracket{\li}$.

  For the other inclusion, it is enough to observe that $\mathbb{F}$
  is a field containing $\R\cup\li$ and closed under infinite sums,
  $\exp$ and $\log$. The verification is easy once we recall that when
  $x$ is finite $\exp(x)$ and $\log(1+x)$ are given by power series
  expansions. The details are as follows:
  \begin{enumerate}
  \item $\mathbb{F}$ is clearly closed under infinite sums, contains
    $\R\cup\li$, and if $x\in\mathbb{F}$, then each term of $x$ is in
    $\F$;
  \item for $\gamma\in\J$, we have $\gamma\in\mathbb{F}$ if and only
    if $r\exp(\pm\gamma)\in\mathbb{F}$ for all $r\in\R^{*}$;
  \item using (2), if $t,u\in\T\cap\mathbb{F}$, then
    $t\cdot u\in\T\cap\mathbb{F}$ and $t^{-1}\in\T\cap\mathbb{F}$;
  \item by infinite distributivity, if $x,y\in\mathbb{F}$, then
    $x\cdot y\in\mathbb{F}$;
  \item expanding the definitions of $\exp$ and $\log$ (see
    \prettyref{thm:characterization-exp} and
    \prettyref{rem:expansion-log}) and using the above (1)-(4), if
    $x\in\mathbb{F}$, then $\exp(x)$ and $\log(x)$ are in
    $\mathbb{F}$.
  \end{enumerate}

  Therefore, $\bracket{\li}\subseteq\F$, hence $\F=\bracket{\li}$, as
  desired.
\end{proof}

The above proposition shows that $\bracket{\li}$ is a ``field of
exponential-logarithmic transseries'' in the sense of \cite[Def.\
5.1]{Kuhlmann2014}.  We omit here the full definition of
exponential-logarithmic transseries and just recall their main
defining property.

\begin{defn}
  Let $\F$ be a subfield of $\no$. Following \cite{Mourgues1993} we
  say that $\F$ is \textbf{truncation closed} if for every $f\in\F$
  and $\m\in\M$ we have $f|\m\in\F$.
\end{defn}

For instance, $\bracket{\li}$ is a truncation closed subfield of
$\no$. The following definition is a slight variation of \cite[Def.\
5.1]{Kuhlmann2014}.

\begin{defn}[{\cite[Def.\ 5.1]{Kuhlmann2014}}]
  A truncation closed subfield $\mathbb{F}$ of $\no$ closed under
  $\log$ satisfies ELT4 if and only if the following holds:
  \begin{itemize}
  \item[\textbf{ELT4.}] For all sequences of monomials
    $\m_{i}\in\M\cap\F$, with $i\in\N$, such that
    \[
    \m_{i}=\exp(\gamma_{i+1}+r_{i+1}\m_{i+1}+\delta_{i+1})
    \]
    where $r_{i+1}\in\R^{*}$, $\gamma_{i+1},\delta_{i+1}\in\J$, and
    $\gamma_{i+1}+r_{i+1}\m_{i+1}+\delta_{i+1}$ is in standard form,
    there is $k\in\N$ such that $r_{i+1}=1$ and
    $\gamma_{i+1}=\delta_{i+1}=0$ for all $i\geq k$.
  \end{itemize}
\end{defn}

\begin{rem}
  \label{rem:elt4}ELT4 implies that the sequence $(\m_{i})$ eventually
  satisfies $\m_{i}\in\li$. We can rephrase this in term of paths: a
  truncation closed subfield $\F$ of $\no$ closed under $\log$
  satisfies ELT4 if and only if for every $x\in\F$ and every path
  $P\in\mathcal{P}(x)$ there exists $k$ such that
  $P(k+1)\in\li$.
\end{rem}

\begin{prop}
  \label{prop:bracket-elt4}$\bracket{\li}$ is the largest truncation
  closed subfield of $\no$ closed under $\log$ and satisfying
  ELT4.
\end{prop}
\begin{proof}
  By \prettyref{rem:elt4} and \prettyref{prop:bracket-paths-in-L},
  $\bracket{\li}$ satisfies ELT4 and every other truncation closed
  subfield $\F$ of $\no$ closed under $\log$ and satisfying ELT4 is
  included in $\bracket{\li}$.
\end{proof}

In \cite[Conj.\ 5.2]{Kuhlmann2014} it was conjectured that $\no$
satisfies ELT4, which is equivalent to saying that
$\bracket{\li}=\no$.  However, this is not the case.

\begin{thm}
  \label{thm:elt4-fails}We have $\bracket{\li}\subsetneq\no$.
\end{thm}
\begin{proof}
  Let $(\m_{i})$ be a sequence of monomials in $\M^{>1}$ such that
  $\m_{i+1}\vless\log(\m_{i})$.

  For $i\in\N$, let $C_{i}$ be the non-empty convex class defined by
  \[
  C_{i}:=\exp(\m_{1}+\exp(\m_{2}+\ldots+\exp(\m_{i}+o(\m_{i}))\ldots)).
  \]
  Since $\m_{i+1}\vless\log(\m_{i})$, we have
  $\vell(\exp(\m_{i+1}+o(\m_{i+1})))<2\m_{i+1}\vless\log(\m_{i})\veq\vell(\m_{i})$,
  and in particular $\exp(\m_{i+1}+o(\m_{i+1}))\subseteq o(\m_{i})$.
  Therefore, $C_{i+1}\subseteq C_{i}$. By the saturation properties of
  surreal numbers, the intersection $\bigcap_{i}C_{i}$ is non empty.

  Let $x\in\bigcap_{i}C_{i}$. We can write, for every $i\in\N$,
  \[
  x=x_{0}=\exp(\m_{1}+\exp(\m_{2}+\ldots+\exp(\m_{i}+x_{i})\ldots))
  \]
  where $x_{i}\vless\m_{i}$ for $i>0$. By construction we have
  $x_{i}=\exp(\m_{i+1}+x_{i+1})$.  Note, however, that this may not be
  the Ressayre representation of $x_{i}$, as $x_{i+1}$ is not
  necessarily in $\J$.

  Write $x_{i}=\gamma_{i}+r_{i}+\varepsilon_{i}$, with
  $\gamma_{i}\in\J$, $r_{i}\in\R$, $\varepsilon_{i}\in o(1)$. By the
  assumption $x_{i}\vless\m_{i}$ we get
  $\gamma_{i}\vless\m_{i}$. Moreover, since
  $x_{i}>\exp(\frac{1}{2}\m_{i+1})\vgreater1$, we have
  $\gamma_{i}\neq0$.

  Now define $P(i)$ as the leading term of $\gamma_{i}$ for $i\in\N$.
  We claim that $P$ in a path in $\mathcal{P}(x)$. Recall that if
  $y=\gamma+r+\varepsilon$, with $\gamma\in\J$, $r\in\R$,
  $\varepsilon\in o(1)$, then $\vell(\exp(y))=\gamma$. It follows that
  \[
  \vell(P(i))=\vell(\gamma_{i})=\vell(x_{i})=\vell(\exp(\m_{i+1}+\gamma_{i+1}+r_{i+1}+\varepsilon_{i+1}))=\m_{i+1}+\gamma_{i+1},
  \]
  Since $P(i+1)$ is a term of $\gamma_{i+1}$, it is also a term of
  $\vell(P(i))$, hence $P$ is a path, and clearly
  $P\in\mathcal{P}(x)$.

  Since $\vell(P(i))=\vell(\gamma_{i})=\m_{i+1}+\gamma_{i+1}$, where
  $\gamma_{i+1}\neq0$, we have $P(i)\nin\M\supset\li$ for all $i$.  By
  \prettyref{prop:bracket-paths-in-L} we have that
  $x\nin\bracket{\li}$, and therefore $\bracket{\li}\subsetneq\no$.
\end{proof}

Despite the fact that some paths may not end in $\li$, recall that if
$x\in\no\setminus\R$ and $P$ is its dominant path, then there exists
$i$ such that $P(i)\in\li$ (see \prettyref{cor:vell-n-fold} or
\prettyref{lem:leading-D}).

\subsection{Axiom T4 of \cite{Schmeling2001}}

We have seen that axiom ELT4 fails in $\no$. However, as we prove in
this section, $\no$ satisfies a weaker axiom called ``T4'' in
\cite[Def.\ 2.2.1]{Schmeling2001}. In fact, we shall see that T4 is
essentially equivalent to the fact that the relation $\ntrunceq$ of
nested truncation is well-founded. This will show that $\no$ is a
field of transseries as axiomatized by Schmeling.

We recall the definition of transseries in \cite{Schmeling2001}.  One
starts with an ordered field $C$ equipped with an increasing group
homomorphism $\exp:(C,+)\to(C^{*},\cdot)$, with $\exp(x)\geq1+x$ for
all $x\in C$ and $\img(\exp)=C^{>0}$. We are then given an ordered
group $\Gamma$, an additive group $B\subseteq C((\Gamma))$ containing
$C((\Gamma^{\leq0}))$ and an increasing homomorphism
$\exp:(B,+)\to(C((\Gamma))^{*},\cdot)$ extending $\exp:C\to C^{*}$ to
$B$. We say that $C((\Gamma)))$ equipped with $\exp$ is a
\textbf{field of transseries} if
\begin{itemize}
\item[\textbf{T1.}] \textbf{}$\img(\exp)=C((\Gamma))^{>0}$;
\item[\textbf{T2.}] \textbf{}$\Gamma\subseteq\exp(C((\Gamma^{>0})))$;
\item[\textbf{T3.}]
  \textbf{$\exp(x)=\sum_{n=1}^{\infty}\frac{x^{n}}{n!}$} for all
  $x\in C((\Gamma^{<0}))$;
\item[\textbf{T4.}] \textbf{}for all sequences of monomials
  $\m_{i}\in\Gamma$, with $i\in\N$, such that
  \[
  \m_{i}=\exp(\gamma_{i+1}+r_{i+1}\m_{i+1}+\delta_{i+1})
  \]
  where $r_{i+1}\in C^{*}$,
  $\gamma_{i+1},\delta_{i+1}\in C((\Gamma^{>0}))$, and
  $\gamma_{i+1}+r_{i+1}\m_{i+1}+\delta_{i+1}$ is in standard form,
  there is $k\in\N$ such that $r_{i+1}=\pm1$ and $\delta_{i+1}=0$ for
  $i\geq k$.
\end{itemize}
If we take $C=\R$, $\Gamma=\M=\exp(\J)$, and $B=\no$, then
$\no=\R((\M))$ equipped with $\exp$ is clearly a model of T1-T3. Axiom
T4 is related to the nested truncation $\ntrunceq$: it is not
difficult to see that assuming T4 one can easily deduce that
$\ntrunceq$ is well-founded.  We shall now verify that since
$\ntrunceq$ is well-founded (\prettyref{thm:nested-trunc-simplifies}),
axiom T4 holds on $\no$, thereby proving that $\no$ is a field of
transseries.

\begin{defn}
  \label{def:T4}Consider a path $P$ and write
  \[
  P(i)=r_{i}\exp(\gamma_{i+1}+P(i+1)+\delta_{i+1})
  \]
  where $0\neq r_{i}\in\R$, $\gamma_{i+1},\delta_{i+1}\in\J$, and
  $\gamma_{i+1}+P(i+1)+\delta_{i+1}$ is in standard form.

  We say that $P$ \textbf{satisfies T4} if there exists $k\in\N$ such
  that $r_{i+1}=\pm1$ and $\delta_{i+1}=0$ for all $i\geq k$,
  otherwise we say that $P$ \textbf{refutes T4}.

  We say that $x\in\no$ \textbf{satisfies T4} if all paths in
  $\mathcal{P}(x)$ satisfy T4, otherwise we say that $x$
  \textbf{refutes T4}.
\end{defn}

Clearly, T4 is equivalent to saying that every $x\in\no$ satisfies T4.

\begin{lem}
  \label{lem:nested-rank-path-unique}Let $x\in\no$ and
  $P\in\mathcal{P}(x)$.  If $\ntrank(P(i))=\ntrank(x)$ for all
  $i\in\N$, then $P$ satisfies T4.
\end{lem}
\begin{proof}
  By \prettyref{lem:nested-rank-path}, $P(0)=r_{0}\m_{0}$ with
  $r_{0}\in\R^{*}$ and $\m_{0}$ minimal in $\supp(x)$, and for all
  $i\in\N$, $P(i+1)=r_{i+1}\m_{i+1}$ with $r_{i+1}\in\R^{*}$ and
  $\m_{i+1}$ minimal in $\vell(P(i))$.  By Propositions
  \ref{prop:ntrank-gamma-exp-gamma}, \ref{prop:ntrank-monomial} and
  \ref{prop:ntrank-term}, it follows that
  $\ntrank(\vell(P(i)))=\ntrank(P(i))=\ntrank(x)$ for all $i$ and
  $r_{i}=\pm1$ for all $i\in\N$. Therefore, $P$ satisfies T4.
\end{proof}

We can now prove that T4 holds on $\no$.

\begin{thm}
  \label{thm:t4}Axiom T4 of \cite[Def.\ 2.2.1]{Schmeling2001} holds in
  $\no$ (with $C=\R$ and $\Gamma=\M$), hence $\no$ is a field of
  transseries in the sense of that paper.
\end{thm}
\begin{proof}
  We prove that all $x\in\no$ satisfy T4. Let $x\in\no$, and assume by
  induction that $y$ satisfies T4 for all $y\in\no$ with
  $\ntrank(y)<\alpha:=\ntrank(x)$.

  Let $P\in\mathcal{P}(x)$ be any path. If $\ntrank(P(j))<\alpha$ for
  some $j\in\N$, then by inductive hypothesis the path
  $i\mapsto P(j+i)$ in $\mathcal{P}(P(j))$ satisfies T4, hence $P$
  itself satisfies T4. On the other hand, if $\ntrank(P(j))=\alpha$
  for all $j\in\N$, then $P$ satisfies T4 by
  \prettyref{lem:nested-rank-path-unique}.  Since $P$ was an arbitrary
  path, $x$ satisfies T4, as desired.
\end{proof}

\section{\label{sec:pre-derivations}Pre-derivations}

The purpose of this section is to show that $\simpled_{\li}$
(\prettyref{def:simple-pre-D}) is the simplest function (in the sense
of \ref{thm:simplest-pre-D}) with positive values satisfying the
inequalities of \prettyref{prop:vell-log-log}.  As anticipated in the
introduction, we call such functions ``pre-derivations''.

\begin{defn}
  \label{def:pre-D}A \textbf{pre-derivation} is a map
  $\somed_{\li}:\li\to\R^{>0}\M$ such that
  \[
  \log\left(\somed_{\li}(\lambda)\right)-\log\left(\somed_{\li}(\mu)\right)\vless\max\left\{
    \lambda,\mu\right\}
  \]
  and $\somed_{\li}(\exp(\lambda))=\exp(\lambda)\somed_{\li}(\lambda)$
  for all $\lambda,\mu\in\li$.
\end{defn}

By \prettyref{thm:extending-pre-D}, any pre-derivation can be extended
to a surreal derivation. We shall verify that $\simpled_{\li}$ has an
inductive definition that involves a variant of the inequalities of
\prettyref{def:pre-D}. As a corollary, $\simpled_{\li}$ is the
simplest pre-derivation. We first observe that pre-derivations must
satisfy the following condition.

\begin{prop}
  \label{prop:pre-D}If $\somed_{\li}$ is a pre-derivation, then
  \[
  \log\left(\frac{\somed_{\li}(\lambda)}{\prod_{i=0}^{l-1}\log_{i}(\lambda)}\right)-\log\left(\frac{\somed_{\li}(\mu)}{\prod_{i=0}^{m-1}\log_{i}(\mu)}\right)\vless\max\left\{
    \log_{l}(\lambda),\log_{m}(\mu)\right\}
  \]
  for all $\lambda,\mu\in\li$ and $l,m\in\N$.
\end{prop}
\begin{proof}
  The conclusion follows trivially from
  \[
  \log(\somed_{\li}(\log_{l}(\lambda)))-\log(\somed_{\li}(\log_{m}(\mu)))\vless\max\left\{
    \log_{l}(\lambda),\log_{m}(\mu)\right\}
  \]
  since
  $\frac{\somed_{\li}(\lambda)}{\prod_{i=0}^{l-1}\log_{i}(\lambda)}=\somed_{\li}(\log_{l}(\lambda))$
  and
  $\frac{\somed_{\li}(\mu)}{\prod_{i=0}^{m-1}\log_{i}(\mu)}=\somed_{\li}(\log_{m}(\mu))$.
\end{proof}

We now use the above inequalities to give an inductive definition for
$\simpled_{\li}$.

\begin{lem}
  \label{lem:min-alpha}Let $x\in\no$ be such that $x>\N$. If
  $\alpha\in\on$ is the minimum ordinal such that
  $\ka{\alpha}\kleq x$, then $\ka{\alpha}\simpleq x$.
\end{lem}
\begin{proof}
  Let $z\in\no$ be the unique number such that $x\kequal\kappa_{-z}$.
  It follows that $\kappa_{-z}\simpleq x$. Therefore, $\alpha\in\on$
  is the minimum ordinal such that $-\alpha\leq-z$. Since the
  representation
  $-\alpha=\emptyset\mid\left\{ -\beta\suchthat\beta<\alpha\right\} $
  is simple, and $-z<-\beta$ for all $\beta<\alpha$, we have
  $-\alpha\simpleq-z$.  It follows that
  $\ka{\alpha}\simpleq\kappa_{-z}\simpleq x$, as desired.
\end{proof}

\begin{lem}
  \label{lem:inductive-simple-pre-D}For all $\lambda\in\li$,
  $\simpled_{\li}(\lambda)$ is the simplest number $x\in\no^{>0}$ such
  that
  \[
  \log\left(\frac{x}{\prod_{i=0}^{l-1}\log_{i}(\lambda)}\right)-\log\left(\frac{\simpled_{\li}(\mu)}{\prod_{i=0}^{m-1}\log_{i}(\mu)}\right)\vless\max\left\{
    \log_{l}(\lambda),\log_{m}(\mu)\right\}
  \]
  for all $\mu\in\li$ such that $\mu\simple\lambda$ and for all
  $l,m\in\N$.
\end{lem}
\begin{proof}
  Let $\lambda\in\li$ and let $x\in\no^{>0}$ be a number satisfying
  the above inequalities. We need to prove that
  $\simpled_{\li}(\lambda)\simpleq x$.  If $\lambda=\omega$, then
  $\simpled_{\li}(\lambda)=1$, and we already know that $1\simpleq x$
  since $x>0$. For arbitrary $\lambda$, we claim that
  $\log(\simpled_{\li}(\lambda))\trunceq\log(x)$. When
  $\lambda\neq\omega$, this clearly implies that
  $\simpled_{\li}(\lambda)\ntrunceq x$, hence
  $\simpled_{\li}(\lambda)\simpleq x$ by
  \prettyref{thm:nested-trunc-simplifies}.

  Since $\simpled_{\li}$ is a pre-derivation, we have
  \[
  \log\left(\frac{\simpled_{\li}(\lambda)}{\prod_{i=0}^{l-1}\log_{i}(\lambda)}\right)-\log\left(\frac{\simpled_{\li}(\mu)}{\prod_{i=0}^{m-1}\log_{i}(\mu)}\right)\vless\max\left\{
    \log_{l}(\lambda),\log_{m}(\mu)\right\}
  \]
  for all $l,m\in\N$ and $\mu\in\li$. It follows that
  \[
  \log\left(\frac{x}{\prod_{i=0}^{l-1}\log_{i}(\lambda)}\right)-\log\left(\frac{\simpled_{\li}(\lambda)}{\prod_{i=0}^{l-1}\log_{i}(\lambda)}\right)\vless\max\left\{
    \log_{l}(\lambda),\log_{m}(\mu)\right\}
  \]
  for all $l,m\in\N$ and $\mu\simple\lambda$. Expanding the two
  logarithms, we get
  \begin{equation}
    \log(x)-\log(\simpled_{\li}(\lambda))\vless\max\left\{ \log_{l}(\lambda),\log_{m}(\mu)\right\} \label{eq:x-simpled-lambda}
  \end{equation}
  for all $l,m\in\N$ and $\mu\simple\lambda$.

  In order to prove $\log(\simpled_{\li}(\lambda))\trunceq\log(x)$,
  let $\m$ be a monomial in the support of
  $\log(\simpled_{\li}(\lambda))$.  We need to prove that
  \[
  \log(x)-\log(\simpled_{\li}(\lambda))\vless\m.
  \]

  Let $\alpha$ be the minimum ordinal such that
  $\ka{\alpha}\kleq\lambda$.  By \prettyref{lem:min-alpha}, we have
  $\ka{\alpha}\simpleq\lambda$.  Note moreover that
  $\ka{\alpha}\kleq\lambda\kless\ka{\beta}$ for all $\beta<\alpha$. We
  distinguish two cases.

  If $\lambda=\ka{\alpha}$, then
  $\log(\simpled_{\li}(\lambda))=-\sum_{\beta<\alpha}\sum_{i=1}^{\infty}\log_{i}(\ka{\beta})$.
  Therefore, $\m$ is of the form $\log_{i}(\ka{\beta})$ for some
  $\beta<\alpha$ and $i\in\N$. Note that
  $\ka{\beta}\simple\ka{\alpha}=\lambda$ and
  $\ka{\beta}\kgreater\ka{\alpha}$. It follows that
  $\log_{l}(\lambda)<\log_{i}(\ka{\beta})$ for all $l\in\N$. Taking
  $\mu=\ka{\beta}$ and $m=i$ in \prettyref{eq:x-simpled-lambda}, we
  get
  $\log(x)-\log(\simpled_{\li}(\lambda))\vless\log_{i}(\ka{\beta})=\m$,
  as desired.

  If $\lambda\neq\ka{\alpha}$, then
  $\log(\simpled_{\li}(\lambda)) =
  -\sum_{\ka{\beta}\kgeq\lambda}\sum_{i=1}^{\infty}\log_{i}(\ka{\beta})+\sum_{i=1}^{\infty}\log_{i}(\lambda)$,
  and $\ka{\alpha}\simple\lambda$. By the choice of $\alpha$ we also
  have $\ka{\alpha}\kleq\lambda$, which means that for all $l\in\N$
  there exists $m\in\N$ such that
  $\log_{m}(\ka{\alpha})<\log_{l}(\lambda)$.  Since
  $\ka{\alpha}\simple\lambda$, we can take $\mu=\ka{\alpha}$ in
  \prettyref{eq:x-simpled-lambda} and deduce that for all $l\in\N$ we
  have
  \[
  \log(x)-\log(\simpled_{\li}(\lambda))\vless\log_{l}(\lambda).
  \]
  If $\m=\log_{l}(\lambda)$ for some $l\in\N$, we are done. If
  $\m=\log_{i}(\ka{\beta})$ for some $i\in\N$ and some
  $\ka{\beta}\kgeq\lambda$, then there exists $l$ such that
  $\log_{l}(\lambda)<\log_{i}(\ka{\beta})$, and therefore
  \[
  \log(x)-\log(\simpled_{\li}(\lambda))\vless\log_{l}(\lambda)\vless\log_{i}(\ka{\beta})=\m,
  \]
  as desired.
\end{proof}

\begin{rem}
  \prettyref{lem:inductive-simple-pre-D} shows that one can
  \emph{define} inductively $\simpled_{\li}(\lambda)$ as the simplest
  $x\in\no^{>0}$ satisfying the inequalities of the lemma. However,
  the fact that $x$ is the \emph{simplest} such number is not
  essential, and other choices of $x\in\T$ satisfying the same
  inequalities are possible and lead to other surreal
  derivations.
\end{rem}

\begin{thm}
  \label{thm:simplest-pre-D}Let $\somed_{\li}:\li\to\no^{>0}$ be a
  pre-derivation. If $\lambda\in\li$ is a number of minimal simplicity
  such that $\somed_{\li}(\lambda)\neq\simpled_{\li}(\lambda)$, then
  $\simpled_{\li}(\lambda)\simple\somed_{\li}(\lambda)$.
\end{thm}
\begin{proof}
  Let $\lambda\in\li$ is a number of minimal simplicity such that
  $\somed_{\li}(\lambda)\neq\simpled_{\li}(\lambda)$.  By assumption,
  $\somed_{\li}(\mu)=\simpled_{\li}(\mu)$ for all $\mu\simple\lambda$.
  Since $\somed_{\li}$ is a pre-derivation, by \prettyref{prop:pre-D}
  it follows that for all $\mu\simple\lambda$ and $l,m\in\N$ we have
  \[
  \log\left(\frac{\somed_{\li}(\lambda)}{\prod_{i=0}^{l-1}\log_{i}(\lambda)}\right)-\log\left(\frac{\simpled_{\li}(\mu)}{\prod_{i=0}^{m-1}\log_{i}(\mu)}\right)\vless\max\left\{
    \log_{l}(\lambda),\log_{m}(\mu)\right\} .
  \]
  By \prettyref{lem:inductive-simple-pre-D}, this implies that
  $\simpled_{\li}(\lambda)\simple\somed_{\li}(\lambda)$, as
  desired.
\end{proof}

\begin{rem}
  A similar argument can be applied to the function $\simpled_{\li}'$
  of \prettyref{def:natural-nonsimple-pre-D} to prove that for all
  $\lambda\in\li$, $\simpled_{\li}'(\lambda)$ is the simplest
  \emph{infinite} number $x\in\no^{>0}$ such that for all
  $\mu\simple\lambda$ and $l,m\in\N$ we have
  \[
  \log\left(\frac{x}{\prod_{i=0}^{l-1}\log_{i}(\lambda)}\right)-\log\left(\frac{\simpled_{\li}'(\mu)}{\prod_{i=0}^{m-1}\log_{i}(\mu)}\right)\vless\max\left\{
    \log_{l}(\lambda),\log_{m}(\mu)\right\} .
  \]
  In particular, $\simpled_{\li}'$ is the simplest pre-derivation with
  only infinite values.
\end{rem}

\providecommand{\bysame}{\leavevmode\hbox to3em{\hrulefill}\thinspace}
\providecommand{\MR}{\relax\ifhmode\unskip\space\fi MR }
\providecommand{\MRhref}[2]{%
  \href{http://www.ams.org/mathscinet-getitem?mr=#1}{#2}
}
\providecommand{\href}[2]{#2}

\end{document}